\documentclass[11pt,letterpaper]{amsart}

\usepackage{amsmath, amsthm, amssymb, amscd, amsfonts}
\usepackage{enumerate}
\usepackage{hyperref}
\usepackage{verbatim}
\usepackage{xcolor}
\usepackage{tikz}
\usetikzlibrary{calc}
\usepackage{bm}

\DeclareMathOperator{\diam}{diam}
\DeclareMathOperator{\dist}{dist}

\title{Coarse and pointwise tangent fields}

\author{Guy C. David}
\address{Department of Mathematical Sciences, Ball State University, Muncie, IN 47306}
\email{gcdavid@bsu.edu}

\author{Sylvester Eriksson-Bique}
\address{Department of Mathematics and Statistics
P.O. Box 35
FI-40014 University of Jyväskylä}
\email{sylvester.d.eriksson-bique@jyu.fi}

\author{Raanan Schul}
\address{Department of Mathematics\\ Stony Brook University\\ Stony Brook, NY 11794}
\email{raanan.schul@stonybrook.edu}

\thanks{G. C. David was partially supported by the National Science Foundation under Grant No. DMS-2054004. S. Eriksson-Bique was partially supported by the Research Council of Finland Grant No. 354241. R. Schul was partially supported by the National Science Foundation under Grant No. DMS-2154613 }

\date{\today}
\subjclass[2020]{28A75, 30L05, 51F30}
\keywords{null set, porosity, tangent field, Analyst's Traveling Salesman}

\DeclareMathOperator{\minUnion}{minUnion}

\begin{document}

\maketitle

\theoremstyle{plain}
\newtheorem{theorem}{Theorem}
\newtheorem{exercise}{Exercise}
\newtheorem{corollary}[theorem]{Corollary}
\newtheorem{scholium}[theorem]{Scholium}
\newtheorem{claim}[theorem]{Claim}
\newtheorem{observation}[theorem]{Observation}
\newtheorem{lemma}[theorem]{Lemma}
\newtheorem{sublemma}[theorem]{Lemma}
\newtheorem{proposition}[theorem]{Proposition}
\newtheorem{conjecture}[theorem]{Conjecture}
\newtheorem{maintheorem}{Theorem}
\newtheorem{maincor}[maintheorem]{Corollary}
\renewcommand{\themaintheorem}{\Alph{maintheorem}}

\theoremstyle{definition}
\newtheorem{fact}[theorem]{Fact}
\newtheorem{example}[theorem]{Example}
\newtheorem{definition}[theorem]{Definition}
\newtheorem{remark}[theorem]{Remark}
\newtheorem{question}[theorem]{Question}
\newtheorem{assumption}[theorem]{Assumption}

\numberwithin{equation}{section}
\numberwithin{theorem}{section}

\newcommand{\obar}[1]{\overline{#1}}
\newcommand{\haus}[1]{\mathcal{H}^n(#1)}
\newcommand{\prob}{\mathbb{P}}
\newcommand{\Tan}{\text{Tan}}
\newcommand{\WTan}{\text{WTan}}
\newcommand{\CTan}{\text{CTan}}
\newcommand{\CWTan}{\text{CWTan}}
\newcommand{\LIP}{\text{LIP}}
\newcommand{\RR}{\mathbb{R}}
\newcommand{\R}{\mathbb{R}}
\newcommand{\HH}{\mathcal{H}}
\newcommand{\cH}{\mathcal{H}}
\newcommand{\B}{\mathcal{B}}
\newcommand{\ZZ}{\mathbb{Z}}
\newcommand{\bH}{\mathbb{H}}
\newcommand{\bX}{\mathbb{X}}
\newcommand{\G}{\mathbb{G}}
\newcommand{\cL}{\mathcal{L}}
\newcommand{\cC}{\mathcal{C}}
\newcommand{\cD}{\mathcal{D}}
\newcommand{\cS}{\mathcal{S}}
\newcommand{\cB}{\mathcal{B}}
\newcommand{\cG}{\mathcal{G}}
\newcommand{\cF}{\mathcal{F}}
\newcommand{\hDelta}{\Delta}
\newcommand{\ball}{\tilde{B}}
\newcommand{\GuyLambda}{\lambda}

\begin{abstract}
Alberti, Cs\"ornyei and Preiss introduced a notion of a ``pointwise (weak) tangent field'' for a subset of Euclidean space -- a field that contains almost every tangent line of every curve passing through the set -- and showed that all area-zero  
sets in the plane admit one-dimensional tangent fields. We extend their results in two distinct directions. First, a special case of our pointwise result shows that each doubling subset of Hilbert space admits a pointwise tangent field in this sense, with dimension bounded by the Nagata (or Assouad) dimension of the set.

Second, inspired by the Analyst's Traveling Salesman Theorem of Jones, we introduce new, ``coarse'' notions of tangent field for subsets of Hilbert space, which take into account both large and small scale structure. We show that doubling subsets of Hilbert space admit such coarse tangent fields, again with dimension bounded by the Nagata (or Assouad) dimension of the set. For porous sets in the plane, this result can be viewed as a quantitative version of the Alberti--Cs\"ornyei--Preiss result, though our results hold in all (even infinite) dimensions. 

\end{abstract}

\tableofcontents

\section{Introduction}

If $M$ is a smooth submanifold in $\mathbb{R}^n$, the tangent space $T_p M$ of $M$ at a point $p$ can be viewed as the space of possible derivatives of smooth curves in $M$ that pass through $p$. In this paper, we try to assign a similar notion of tangent space to an essentially arbitrary subset of Hilbert space, not necessarily a manifold. Inspired by the Analyst's Traveling Salesman Theorem of Jones \cite{Jones} and the work of Albert--Cs\"ornyei--Preiss \cite{ACP}, we build certain ``coarse'' and ``pointwise'' tangent fields and control their dimension. 

{ 
Let us give a few more details, starting in the planar case. Suppose that $E$ is a (Borel) set in the plane and $\tau$ is a (Borel) assignment of a line in $\mathbb{R}^2$ to each point of $E$. Call $\tau$ a {\it pointwise weak tangent line field}\footnote{Alberti--Cs\"ornyei--Preiss refer to their notion as a ``weak tangent field'', but to distinguish their notion from others in this paper, we refer to it as a ``pointwise weak tangent field''.} for $E$ if it has the following property: For every rectifiable curve $\Gamma$ in the plane,  
\begin{equation}\label{eq:tangentlinefield}
	\mathcal{H}^1(\{x\in \Gamma \cap E: \tau_\Gamma(x)\neq \tau(x)\})=0,
\end{equation} 
where $\tau_\Gamma(x)$ is the tangent line to $\Gamma$ at $x$ (defined for $\mathcal{H}^1$ a.e. $x\in\Gamma$). It follows from classical geometric measure theory that if $E$ is a $1$-rectifiable set, then $E$ has a pointwise weak tangent line field. At the other end of the spectrum, if $E$ is purely $1$-unrectifiable, then an arbitrary assignment $\tau$ will work, since already $\mathcal{H}^1(\Gamma \cap E)=0$ for every rectifiable curve $\Gamma$. On the other hand, if $E$ is an open set in the plane, then clearly it does not have a pointwise weak tangent line field in this sense.

Alberti--Cs\"ornyei--Priess proved the following astounding result, which says that having null {\it area}, i.e. two-dimensional Hausdorff measure, is sufficient \cite{ACP}.

\begin{theorem}[Alberti-Cs\"ornyei-Preiss]\label{thm:ACP}
If a Borel set $E\subset \RR^2$ has $\mathcal{H}^2(E)=0$, then $E$ has a pointwise weak tangent line field.
\end{theorem}
An analog of Theorem \ref{thm:ACP} in $\mathbb{R}^n$ would ask that an $\mathcal{H}^n$-null set in $\mathbb{R}^{n-1}$ have a pointwise weak $(n-1)$-dimensional plane field. Such a result is currently not known to hold if $n\geq 3$ (though see the announced results of Cs\"ornyei-Jones \cite{Csornyei-Jones}); we will discuss this further below. Even the known planar case has had a significant impact on the study of Lipschitz mappings and differentiability; see, e.g., the original paper of Alberti-Cs\"ornyei-Preiss  \cite{ACP}, as well as the work of Bate \cite{Bate} and Maleva--Preiss \cite{MalevaPreiss2019}.

{ Our goal in the present paper is twofold. First, by a different approach than that of Albert--Cs\"ornyei--Preiss, we extend their notion of pointwise tangent field to many subsets of higher-dimensional (or even infinite-dimensional) Hilbert space (though not arbitrary null sets). Second, we provide a ``coarse'' or ``quantitative'' analog of Theorem \ref{thm:ACP}, one that also applies in an arbitrary ambient Hilbert space. This coarse result is modeled on the Analyst's Traveling Salesman Theorem of Jones, and its proof occupies the majority of the paper.

%First, we provide a ``coarse'' or ``quantitative'' analog of Theorem \ref{thm:ACP}, one that in addition applies in an arbitrary or even infinite dimensional ambient Hilbert space. This coarse result is modeled on the Analyst's Traveling Salesman Theorem of Jones. Second, we return to the pointwise setting of Alberti--Cs\"ornyei--Preiss and, by a different argument, extend their result to many subsets of higher-dimensional Euclidean space (though not arbitrary null sets). 

We now describe the setup for our main results in more detail.
}

}

\subsection{Fitting \texorpdfstring{$d$}{d}-planes and dimension theory}
Our goal is to find conditions under which a subset of Hilbert space admits certain coarse or pointwise tangent fields that we will define below. To this end, we first make the following definition:

\begin{definition}\label{def:fit}
Suppose that $E$ is a subset of a Hilbert space $\bX$ and fix $d\in\mathbb{N}$. We say that $E$ \emph{badly fits $d$-planes} if the following holds:

There is a constant $\epsilon_0>0$ such that, if $V$ is a $d$-plane and $B$ is a ball in $\bX$, then
$$ V\cap B \not\subset N_{\epsilon_0 \diam(V \cap B)}(E)$$

\end{definition}
In other words, $E$ badly fits $d$-planes if $d$-dimensional planes are never able to pass too close to $E$, regardless of location or scale. This definition is inspired by the notion of ``line-fitting'' for metric spaces introduced by Laakso in the Appendix of \cite{TysonWu}. Essentially, a subset of Hilbert space is line-fitting in Laakso's sense if and only if it does \textbf{not} badly fit $1$-planes, in our sense.

We will see below in subsection \ref{subsec:dimension} that this notion is related to well-studied notions of dimension. In particular, if $E$ has \emph{Nagata dimension} or \emph{Assouad dimension} strictly less than $d$, then it badly fits $d$-planes.

\subsection{Main results on pointwise weak tangent fields}
Earlier, we explained a notion of ``pointwise weak tangent field'' due to Alberti--Cs\"ornyei--Preiss. While we do not generalize their result, we are able to extend it in some new directions.

Recall that if $\Gamma$ is a rectifiable curve in a Hilbert space, then it admits a tangent line $\tau_\Gamma(p)$, in the usual sense, at $\mathcal{H}^1$-a.e. point $p\in \Gamma$. For instance, $\tau_\Gamma(p)$ can be defined $\mathcal{H}^1$-a.e. by fixing an arc-length parametrization $\gamma$ of $\Gamma$ and choosing $\tau_S(\gamma(t))$ to be the span of $\gamma'(t)$.

Extending the notion given in \cite{ACP}, we define the following.

\begin{definition}
Let $E$ be a Borel subset of a Hilbert space $\bX$ and $k\geq 1$. Let $\tau\colon E \rightarrow \bigcup_{j=0}^k\mathcal{L}_j$ be a Borel assignment of a linear subspace of dimension at most $k$ to each point of $E$. 

We say that $\tau$ is a \emph{$k$-dimensional pointwise weak tangent field} to $E$ if the following holds: For every rectifiable curve $\Gamma$ in $\bX$,
$$ \tau_\Gamma(p) \subseteq \tau(p) $$
for $\mathcal{H}^1$-a.e. $p\in \Gamma\cap E$.\footnote{Actually, Alberti--Cs\"ornyei--Preiss require this to hold only when $\Gamma$ is a $C^1$ curve, but the two definitions give equivalent notions by standard geometric measure theory arguments.}
\end{definition}

Alberti, Cs\"ornyei, and Preiss show in \cite[Theorem 4.3]{ACP}, stated as Theorem \ref{thm:ACP} above, that every Borel set in $\mathbb{R}^2$ of Lebesgue measure zero admits a $1$-dimensional pointwise weak tangent field. They note in \cite[Section 8c]{ACP} that the analogous statement in higher dimensions (whether null sets in $\mathbb{R}^n$ admit $(n-1)$-dimensional pointwise weak tangent fields) is unknown. Despite much work on related questions and an announcement in the positive direction by Cs\"ornyei and Jones \cite{Csornyei-Jones}, this question remains open to the best of our knowledge. 

We are able to obtain a version of their theorem in arbitrary dimension, but not for arbitrary null sets:
\begin{theorem}\label{thm:pointwise}
Let $\bX$ be a (finite or infinite dimensional) Hilbert space. Let $E\subseteq \bX$ be a countable union of sets $E_i$ that each badly fits $d$-planes.

Then $E$ admits a $(d-1)$-dimensional pointwise weak tangent field.
\end{theorem}
This result is strictly weaker than Theorem \ref{thm:ACP} when $\bX=\mathbb{R}^2$: A subset of $\mathbb{R}^2$ that badly fits $2$-planes must be porous, but there are null sets in $\mathbb{R}^2$ that are not $\sigma$-porous. (We recall that a subset of a Hilbert space $\bX$ is \emph{porous} if every ball in $\bX$ contains a ball of comparable radius disjoint from this subset; in some areas, this is called ``lower porous''.)

However, Theorem \ref{thm:ACP} is not known to hold in $\mathbb{R}^n$ for $n\geq 2$, much less in infinite-dimensional Hilbert space, so Theorem \ref{thm:pointwise} may serve as a partial substitute in those cases. In particular, Theorem \ref{thm:pointwise} implies that every $\sigma$-porous set in $\mathbb{R}^n$ admits an $(n-1)$-dimensional pointwise weak tangent field. More generally, Theorem \ref{thm:pointwise} (combined with Lemma \ref{lem:notgood} below) immediately implies the following:
\begin{corollary}\label{cor:pointwisenagata}
Let $\bX$ be a (finite or infinite dimensional) Hilbert space and $n\in\mathbb{N}$. Let $E\subseteq \bX$ be a countable union of sets $E_i$ that each have Nagata dimension at most $n$.

Then $E$ admits an $n$-dimensional pointwise weak tangent field.
\end{corollary}

To contrast with Theorem \ref{thm:pointwise}, we observe that the results of Alberti--Cs\"ornyei--Preiss \cite{ACP} imply with a little work something a bit stronger than Theorem \ref{thm:ACP}. 
\begin{proposition}
If $\bX$ is a Hilbert space and $E\subset \mathbb{X}$ has $\cH^2(E)=0$, then $E$ admits a one dimensional pointwise weak tangent field. In particular, if $E$ has Hausdorff dimension $< 2$, then $E$ admits a pointwise weak tangent field of dimension at most its Hausdorff dimension.
\end{proposition}
\begin{proof}
We sketch the argument. Without loss of generality $\mathbb{X}$ is separable. Consider a countable dense collection of rank-two bounded linear maps $\Pi=\{\pi:\mathbb{X}\to \R^2\}$, and consider $E_\pi = \pi(E)$. (In fact, it suffices to consider orthogonal projections onto some dense collection of two-dimensional subspaces.) Then $\cH^2(E_\pi)=0$ for all $\pi\in \Pi$, and \cite{ACP} implies that each $E_\pi$ has a one dimensional pointwise weak tangent-field $\tau_\pi$. Now, define a pointwise weak tangent field for $x\in E$ as the maximal subspace $\tau(x)\subset \mathbb{X}$ such that $\pi(\tau(x))\subset \tau_\pi(\pi(x))$ for each $\pi\in \Pi$. This subspace must be at most one-dimensional, since if $\tau(x)$ contained two linearly independent directions, there would be a projection $\pi\in \Pi$ s.t. $\pi(\tau(x))$ would be two dimensional, contradicting the fact that $\tau_\pi(\pi(x)))$ is one dimensional. Finally, one needs to show that this $\tau(x)$ is a pointwise weak tangent field. If $\gamma:I\to \mathbb{X}$ is any rectifiable curve, then since $\Pi$ is countable, for a.e. $t\in I$ we have $\pi(\gamma'(t))=(\pi \circ \gamma)'(t)\in \tau_{\pi}(\pi(\gamma(t)))$. Since $\tau(\gamma(t))$ is a maximal subspace whose projection is contained in $\tau_{\pi}(\pi(\gamma(t)))$, we must have $\langle \gamma'(t) \rangle \subset \tau(\gamma(t))$. 

The Hausdorff dimension claim follows once we observe that if ${\rm dim}_H(E)<1$, then $E$ has a zero dimensional pointwise weak tangent field. But, in this case $\cH^1(\gamma\cap E)\leq \cH^1(E)=0$ for all rectifiable curves $\gamma$, and thus $\tau(x)=\{0\}$ is a pointwise weak tangent field. 
\end{proof}

\begin{remark}
Conditional on the announced result of Cs\"ornyei and Jones, the previous argument would similarly imply that any subset $E\subset \mathbb{X}$ admits a weak tangent field of dimension at most its Hausdorff dimension. In contrast, Theorem \ref{thm:pointwise} implies -- see Corollary \ref{cor:pointwisenagata} -- that every set $E\subset\bX$ has a pointwise weak tangent field of dimension at most its Nagata dimension. In some cases, this bound is better (smaller) than the Hausdorff dimension, since in general the Hausdorff and Nagata dimensions are not comparable.
\end{remark}

\begin{remark}
The proof of Theorem \ref{thm:ACP} relies at heart on a combinatorial fact about covering finite planar sets by Lipschitz graphs (see \cite[Theorem 2.1]{ACP}). This fact does not generalize easily to ambient dimensions greater than $2$ (see \cite[Section 8b]{ACP}), which is why Theorem \ref{thm:ACP} is not known to hold in this setting. Our proof of Theorem \ref{thm:pointwise} does not use any such covering lemma.
\end{remark}

\subsection{Linear approximation and the Analyst's Traveling Salesman Theorem}
We now wish to describe our other set of main results, which concern a new notion of ``coarse tangent field''. To do so, we first give some background in quantitative geometric measure theory.

Suppose that $\bX$ is a (finite or infinite dimensional) Hilbert space and $F$ is a subset of $\bX$. Given a ball $B$ in $\bX$, we may ask: how close is $F\cap B$ to a line $L$? There are two ways to measure this: a ``unilateral'' version,
$$ \beta_F(B) = \frac{1}{\diam(B)} \inf_L \sup_{x\in F \cap B} \dist(x,L),$$
and a ``bilateral'' version
$$ \theta_F(B) = \frac{1}{\diam(B)} \inf_L \left( \sup_{x\in F \cap B} \dist(x,L) + \sup_{y\in L \cap B} \dist(x,F)\right).$$
In both cases we take the infimum over all affine lines $L$ in $\bX$ (and both quantities are set to be $0$ if $F\cap B=\emptyset$). If $\beta_F(B)$ is small then $F\cap B$ lies in a thin tube, and if $\theta_F(B)$ is small then $F\cap B$ looks very much like a line segment (or a null set) at the scale of $B$. Of course, we always have $0 \leq \beta_F \leq \theta_F \lesssim 1.$ 

In 1990, Peter Jones \cite{Jones} introduced the $\beta$-numbers and showed in his ``Analyst's Traveling Salesman Theorem'' that a multiscale sum of these quantities can be used to measure lengths of rectifiable curves in the plane; his results were later extended to $\mathbb{R}^n$ and Hilbert space by Okikiolu \cite{Okikiolu} and the third author \cite{Schul}. 

Let $\bX$ be a Hilbert space and $F\subseteq X$. Let $\{N_k\}_{k\in\mathbb{Z}}=\{\{z^k_\alpha\}_{\alpha\in I_k}\}_{k\in\mathbb{Z}}$ be a nested sequence of $2^{-k}$-nets in in $F$ (i.e.,  $N_k\subset N_{k+1}$). The collection of balls
$$ \{ B(z^k_\alpha, 2\cdot 2^{-k}) : k\in\mathbb{Z}, \alpha\in I_k\}$$
is called a \emph{multiresolution family for $F$}. If $B$ is a ball in a Hilbert space and $A>0$, we write $AB$ for the ball with the same center and $A$ times the radius. We than have:

\begin{theorem}[Jones, Okikiolu, Schul]\label{thm:tst}
Let $F$ be a subset of a Hilbert space $\bX$ with multiresolution family $\cF$, and $A$ a sufficiently large inflation parameter. Then:
\begin{enumerate}[(i)]
\item There is a connected set $\Gamma$ containing $F$ with
$$ \ell(\Gamma) \lesssim \diam(F) + \sum_{B\in\cF} \beta_F(A B)^2 \diam(B).$$
\item If $F$ is already connected, then
\begin{equation}\label{eq:tstsum}
\sum_{B\in\cF} \beta_F(A B)^2 \diam(B) \lesssim \ell(F).
\end{equation}
\end{enumerate}
The implied constants depend only on $A$. Here $\ell$ is the one-dimensional Hausdorff measure $\mathcal{H}^1$.
\end{theorem}

Part (ii) of Theorem \ref{thm:tst} says, loosely speaking,
that every finite-length curve in Hilbert space has a good approximating line at ``most'' locations and scales along the curve. A corollary of our main result (Corollary \ref{cor:Nagata-Assuad}) says roughly that if one restricts the sum in \eqref{eq:tstsum} to the balls in $\cF$ that lie close to a given set $E\subseteq \bX$ of (Assouad or Nagata) dimension $<d$, then one can essentially predetermine (i.e. irrespective of $F$) 
a $(d-1)$-plane that contains the approximating line for the definition of $\beta_F$. For example, if one takes $E$ to be a Sierpinski carpet in the plane, then its dimension is less than $2$ and one can obtain a quantitative analogue of Theorem \ref{thm:ACP} in this setting.

\textbf{Our main results (Theorems \ref{thm:main} and \ref{thm:onedim}) are in fact more general than this}: they apply for sets $E$ in an infinite dimensional Hilbert space that satisfy a condition weaker than having dimension $<d$.

\subsection{Coarse tangent fields}
Fix a (finite or infinite dimensional) Hilbert space $\bX$. We write $\mathcal{L}_k$ for the collection of $k$-dimensional linear subspaces of $\bX$ and $\mathcal{A}$ for the collection of affine lines in $\bX$. If $V\in\mathcal{L}_k$ and $L\in \mathcal{A}$, we write $L\parallel V$ if $L=L_0 + w$ where $L_0\subseteq V$ and $w\in\bX$.

Let $\cF$ be a countable collection of balls in $\bX$. Given $k\in\mathbb{N}$, a \emph{$k$-dimensional coarse plane field on $\cF$} is an assignment $\tau\colon \cF \rightarrow \bigcup_{j=0}^k \mathcal{L}_j$. In other words, $\tau$ assigns to each ball in $\cF$ a linear subspace with dimension $\leq k$.
 
Now let $\Gamma$ be a subset of $\bX$ (which in this paper will usually be a rectifiable curve). If $\tau$ is a coarse plane field on $\cF$ and $B$ is a ball in $\cF$, we may define a ``restricted'' version of $\beta$ and $\theta$ by
$$ \beta^{\tau}_{\Gamma}(C) = \frac{1}{\diam(C)}\inf\left\{ \sup_{x\in \Gamma \cap C} \dist(x, L): L\in\mathcal{A}, L\parallel \tau(C) \right\}$$
and
$$ \theta^{\tau}_{\Gamma}(C) = \frac{1}{\diam(C)}\inf\left\{ \sup_{x\in \Gamma \cap C} \dist(x, L) + \sup_{y\in L \cap C} \dist(y, \Gamma): L\in\mathcal{A}, L\parallel \tau(C) \right\}.$$
The key point here is that \emph{when computing $\beta^\tau_\Gamma(C)$ and $\theta^\tau_\Gamma(C)$, we permit ourselves only to consider approximating lines that are parallel to the pre-assigned plane $\tau(C)$.} In particular, we have
$$ 0 \leq \theta_\Gamma \leq \theta^\tau_\Gamma \lesssim 1$$
and similarly for $\beta_\Gamma$ vs. $\beta_\Gamma^\tau$.

If $\cF$ is a collection of balls in $\bX$ and $A\geq 1$, we set
$$ \mathcal{F}_A = \{AB : B\in \cF\}.$$

\begin{definition}\label{def:coarsetangent}
Let $\cF$ be a family of balls in a Hilbert space $\bX$ and $A\geq 1$. Suppose that $\tau$ is a $k$-dimensional coarse plane field on $\cF_A$ and $\epsilon>0$.

We call $\tau$ a \emph{$k$-dimensional coarse $\epsilon$-tangent field for $\cF$ with inflation $A$} if the following holds: There is a constant $C=C(\cF, \tau,\epsilon, A)$ such that 
\begin{equation}\label{eq:epstangentdef}
 \sum_{\substack{B\in\cF\\ \Gamma \cap B \neq \emptyset \\ \diam(B) \leq \diam(\Gamma)\\ \theta^\tau_\Gamma(AB)\geq \epsilon }} \diam(B) \leq  C\ell(\Gamma),
\end{equation}
whenever $\Gamma\subseteq \bX$ is a rectifiable curve. 

It may be that a fixed coarse plane field $\tau$ serves as a coarse $\epsilon$-tangent field for $\cF$ with inflation $A$ simultaneously for \textbf{every} $\epsilon>0$ (with the constant in \eqref{eq:epstangentdef} depending on $\epsilon$). In that case, we we say that $\tau$ is a \emph{coarse tangent field for $\cF$ with inflation $A$}.
\end{definition}

{The idea behind a coarse tangent field can be summarized as follows.} If $\cF$ is a family of balls and and $\tau$ is a coarse plane field on $\cF$, then we may ``guess'' that for each ball $B\in\cF$, every curve $\Gamma$ that passes near $B$ must be well-approximated by a line parallel to $\tau(B)$. This guess may sometimes be wrong. However, if $\tau$ is a coarse $\epsilon$-tangent field for $\cF$, then, for every curve $\Gamma$, this guess is \emph{almost always accurate to within $\epsilon$}: the collection of balls where our pre-chosen linear approximation was incorrect by more than error $\epsilon$ is well controlled.

In practice, we will consider only the case where $\cF$ is a multiresolution family of balls associated to a fixed set $E$. In this case, one might consider a coarse tangent field for $\cF$ as a coarse tangent field associated to the set $E$.

\subsection{Main results on coarse tangent fields}

The main ``coarse'' result of the paper is Theorem \ref{thm:main}. Given an arbitrary doubling subset of Hilbert space, we construct coarse $\epsilon$-tangent fields associated to it, with dimension controlled by the dimension $d$ for which $E$ badly fits $d$-planes.

\begin{theorem}\label{thm:main}
Let $\bX$ be a (finite or infinite dimensional) Hilbert space. Let $E\subseteq \bX$ be a doubling subset that badly fits $d$-planes for some $d\geq 1$. Suppose $\mathcal{F}$ is a multiresolution family of balls for $E$ and fix an inflation parameter $A\geq 1$. 

Then for each $\epsilon>0$, $\mathcal{F}$ admits a $(d-1)$-dimensional coarse $\epsilon$-tangent field $\tau_\epsilon$ with inflation $A$. For each $\epsilon$, the constant $C$ in \eqref{eq:epstangentdef} depends only on $A, \epsilon, d$, the constant $\epsilon_0$ from Definition \ref{def:fit}, and the doubling constant of $E$.
\end{theorem}

\begin{remark} For every doubling subset $E\subset \bX$ and multiresolution family $\mathcal{F}$ there exists a ``trivial'' coarse $\epsilon$-tangent field $\tau_\epsilon$, obtained by setting $\tau_\epsilon(B)$ to be parallel to the affine span of a $\sim \epsilon \diam(B)$-net in $2AB$. However, the dimension of such a tangent field depends on $\epsilon$ and in Hilbert space (typically) blows up as $\epsilon$ approaches $0$. The key difficulty in our theorem is that the dimension of the field is bounded in a sharp way by quantities independent of $\epsilon$.

\end{remark}

We note that every doubling subset of Hilbert space has some finite Nagata dimension $n$ and hence, as we show below, it badly fits $d$-planes for $d=n+1$. Thus, Theorem \ref{thm:main} applies (with appropriate $d$) to every doubling subset of Hilbert space. (The dimension of the tangent field can also be bounded by the Assouad dimension of the set $E$, since this is an upper bound for the Nagata dimension \cite{LeDonneRajala}. Alternatively, one can use the Assouad dimension directly to deduce the badly fitting property.)

In Theorem \ref{thm:main}, the choice of coarse $\epsilon$-tangent field is allowed to depend on $\epsilon$ (as well as on $E$ , $\mathcal{F}$, and the various other parameters in the statement). When $d=2$ and the theorem constructs coarse line fields, we are able to show that the coarse tangent field can be chosen independent of $\epsilon$:

\begin{theorem}\label{thm:onedim}
Let $\bX$ be a (finite or infinite dimensional) Hilbert space.  Let $E\subseteq \bX$ be a doubling subset that badly fits $2$-planes. Suppose $\mathcal{F}$ is a multiresolution family of balls for $E$ and fix an inflation parameter $A\geq 1$.

Then $\mathcal{F}$ admits a $1$-dimensional coarse tangent field $\tau$ with inflation $A$. For each $\epsilon$, the constant $C$ in \eqref{eq:epstangentdef} depends only on $A, \epsilon, d$, the constant $\epsilon_0$ from Definition \ref{def:fit}, and the doubling constant of $E$.
\end{theorem}
In particular, Theorem \ref{thm:onedim} applies to every porous set in $\mathbb{R}^2$, for instance the Sierpi\'nski carpet, as well as examples like the Menger sponge in $\mathbb{R}^3$. %We recall that a subset of a Hilbert space $\bX$ is \emph{porous} if every ball in $\bX$ contains a ball of comparable radius disjoint from this subset. (In some areas, this is called ``lower porous''.)

As noted above, dimension theory gives a simple sufficient condition for a space to badly fit $d$-planes: a space with Nagata or Assouad dimension strictly less than $d$ must badly fit $d$-planes. Therefore, we have the following 

\begin{corollary}\label{cor:Nagata-Assuad}
Let $\bX$ be a (finite or infinite dimensional) Hilbert space. Let $E\subseteq \bX$ be a doubling subset with Nagata dimension at most $n$ (or Assouad dimension strictly less than $n+1$). Suppose $\mathcal{F}$ is a multiresolution family of balls for $E$ and fix an inflation parameter $A\geq 1$.

Then for each $\epsilon>0$, $\mathcal{F}$ admits a $n$-dimensional coarse $\epsilon$-tangent field with inflation $A$. For each $\epsilon$, the constant $C$ in \eqref{eq:epstangentdef} depends only on $A, \epsilon, n$ and the Nagata dimension and doubling constants of $E$.

Moreover, if $n=1$, then $\mathcal{F}$ admits a $1$-dimensional coarse tangent field with inflation $A$.
\end{corollary}

\begin{remark}\label{rmk:unnested}
For simplicity and consistency with the literature, our definition of multiresolution family requires that the nets $N_k$ which are the centers of balls in $\cF$ are nested ($N_k\subseteq N_{k+1}$). However, this nested property of the nets is not really necessary for the multiresolution family $\cF$ for $E$ appearing in Theorems \ref{thm:main} and \ref{thm:onedim}. In other words, if $N_k$ are (not necessarily nested) $2^{-k}$-nets and 
$$ \cF = \{B(x,2(2^{-k}): x\in N_k, k\in\mathbb{Z}\},$$
then the theorems still hold. This follows from our theorems by a simple argument using the doubling property of $E$. Sometimes (e.g., in unbounded sets $E$) it is easier to build un-nested nets.
\end{remark}

\begin{remark}
A potentially useful feature of our coarse, quantitative results (Theorems \ref{thm:main} and \ref{thm:onedim}, Corollary \ref{cor:Nagata-Assuad}) is that they are already interesting when $E$ is a finite set, in a way we now describe. (This is in contrast to, e.g., Theorem \ref{thm:ACP}, which is trivial for finite sets.)

If $E$ is a set of $N$ points in $\bX$, then trivially $E$ badly fits $1$-planes (with some parameter $\epsilon_0$) and has Nagata dimension $0$ (with some Nagata dimension constant); therefore it has a $1$-dimensional coarse tangent field, though possibly with a large implied constant in \eqref{eq:epstangentdef}.

On the other hand, if one has bounds \emph{independent of $N$} on the data for which $E$ badly fits $d$-planes, or bounds \emph{independent of $N$} on the Nagata dimension and Nagata dimension constant of $E$, then Theorem \ref{thm:main} provides coarse $\epsilon$-tangent fields associated to $E$ with dimension and constant independent of $N$.
\end{remark}

Lastly, we note that the ``badly fitting $d$-planes'' property is not only sufficient, but also necessary to admit a coarse $\epsilon$-tangent field:
\begin{theorem}\label{thm:converse}
Let $\bX$ be a (finite or infinite dimensional) Hilbert space, $E\subseteq \bX$, $\cF$ a multiresolution family for $E$, $A\geq 1$, and $d\in\mathbb{N}$. Suppose that $\cF$ admits a $(d-1)$-dimensional coarse $\epsilon$-tangent field with inflation $A$, for some $\epsilon>0$ sufficiently small (depending on $d$, $A$). Then $E$ badly fits $d$-planes.
\end{theorem}

\subsection{Related results and open questions}

Our main results generate a number of other questions and directions. Some we solve here, while others we leave as questions.

\subsubsection{Independence of $\epsilon$}
In Theorem \ref{thm:main}, the coarse $\epsilon$-tangent field that we construct depends on the parameter $\epsilon$. 

\begin{question}\label{q:independence}
In Theorem \ref{thm:main}, can the coarse $\epsilon$-tangent field be chosen independent of $\epsilon$? In other words, in the notation of that theorem, does $\cF$ admit a coarse tangent field with inflation $A$?
\end{question}

As noted in Theorem \ref{thm:onedim}, we are able to resolve this if $d\leq 2$, but not in higher dimensions.

\subsubsection{Stronger summability for coarse tangent fields}
The notion of coarse tangent field in Definition \ref{def:coarsetangent} requires only a ``weak type'' bound on the sum of the $\theta^\tau_\Gamma(B)$ for $B\in\cF$. Given the Analyst's Traveling Salesman Theorem of Jones (Theorem \ref{thm:tst}), it would be natural to wonder if some type of stronger $L^2$-summability holds. Focusing on the case $\bX=\mathbb{R}^2$ and $d=2$, we may ask the following precise question: \emph{Suppose $E$ is a porous set in $\mathbb{R}^2$ and $\cF$ is a multiresolution family for $E$. Is there a coarse line field $\tau\colon \cF \rightarrow \mathcal{L}_1$ such that}
$$ \sum_{\substack{B\in \cF\\ \Gamma \cap B \neq \emptyset}} \beta^\tau_\Gamma(AB)^2 \diam(B) \lesssim \ell(\Gamma)?$$

We show below that the answer to this question is ``no'':

\begin{theorem}\label{thm:example}
There is a compact, porous set $P\subseteq \mathbb{R}^2$, a multiresolution family $\cF$ for $P$, and a constant $A\geq 1$, with the following property: For each coarse line field $\tau\colon \cF \rightarrow \mathcal{L}_1$, we have

$$\sum_{\substack{B\in \cF\\ \Gamma \cap B \neq \emptyset}} \beta^\tau_\Gamma(AB)^2 \diam(B)=\infty$$
for some curve $\Gamma \subseteq P$ of finite length.
\end{theorem}
In particular, it follows easily from this that this $\Gamma$ also satisfies
$$ \sum_{\substack{B\in \cF\\ \Gamma \cap B \neq \emptyset\\ \diam(B)\leq\diam(\Gamma)}} \theta^\tau_\Gamma(AB)^2 \diam(B) = \infty.$$
In fact, there will be many such ``bad curves'' $\Gamma\subseteq P$; see Lemma \ref{lem:counterexample} for a more precise statement that implies Theorem \ref{thm:example}.

This leads to a natural question:
\begin{question}
What is the infimum of all exponents $p$ for which the following holds?

Suppose $E$ is a doubling subset of a Hilbert space $\bX$ and $E$ badly fits $d$-planes. Let $A\geq 1$ and $\cF$ a multiresolution family for $E$. Then there is a $(d-1)$-dimensional coarse plane field $\tau$ on $\cF$ such that
\begin{equation}\label{eq:summability}
\sum_{\substack{B\in \cF\\ \Gamma \cap B \neq \emptyset}} \beta^\tau_\Gamma(AB)^p \diam(B) \lesssim \ell(\Gamma)
\end{equation}
for all rectifiable curves $\Gamma\subseteq \bX$. 
\end{question}
A $\tau$ satisfying \eqref{eq:summability} for some $p>0$ would be a coarse tangent field for $\cF$, so any finite value of $p$ answering this question would answer Question \ref{q:independence}.

The Analyst's Traveling Salesman Theorem itself implies that $p<2$ is impossible, while Theorem \ref{thm:example} shows that $p=2$ is impossible. It may well be the case that one can take $p$ arbitrarily close to $2$, at least when $d=2$.

\subsubsection{Banach and metric spaces}
That Theorem \ref{thm:main} holds for subsets of Hilbert space ultimately comes down, in our proof, to the Analyst's Traveling Salesman Theorem (Theorem \ref{thm:tst}) in Hilbert space. Ultimately, it is crucial for us that the exponent $2$ appears in \textbf{both} parts of Theorem \ref{thm:tst}; this is used in Section \ref{sec:TST}.

Versions of the Traveling Salesman Theorem hold in other Banach spaces (see \cite{DavidSchul, BadgerMcCurdy}), but in this case the relevant exponent for the two halves of Theorem \ref{thm:tst} do not match in general. It may be that the arguments can be pushed through in a different way in some other Banach spaces, or it may be that the situation is unique to Hilbert space; we do not know the answer.

\begin{question}\label{q:banach}
Are there other Banach spaces $\bX$ for which Theorem \ref{thm:main} holds?
\end{question}

Beyond this, one may also ask whether some version of Theorem \ref{thm:main} holds in metric spaces. This would be somewhat analogous to deep work of Bate \cite{Bate_perturb} in the pointwise setting, which adapts parts of \cite{ACP} to metric spaces.

\subsubsection{Other collections of balls}
Theorems \ref{thm:main} and \ref{thm:onedim} construct coarse $\epsilon$-tangent fields for collections $\cF$ that arise as multiresolution families associated to underlying sets $E$. However, the definition of a coarse tangent (or $\epsilon$-tangent) does not \emph{a priori} require the existence of an underlying set $E$, merely a collection of balls.

\begin{question}
Let $\cF$ be a collection of balls in a Hilbert space and $\epsilon>0$, $A\geq 1$. (We do not assume that $\cF$ is a multiresolution family for some set.) Under what conditions can we guarantee that $\cF$ admits a coarse $\epsilon$-tangent field with inflation $A$?
\end{question}

A very specific instance is the following: Suppose $\cF$ is a \textbf{disjoint} collection of balls in the unit square $[0,1]^2$ of $\bX = \mathbb{R}^2$. Does $\cF$ admit a $1$-dimensional coarse tangent field (with some inflation $A>1$)? More generally, one might consider collections of balls in $\mathbb{R}^n$ that satisfy a \emph{Carleson packing condition}, but we do not pursue this line further here.

\subsubsection{Coarse differentials of Lipschitz mappings}
The dual notion of a tangent vector is that of a differential, and to each notion of tangent space there is usually an associated notion of a differential. There are many examples of this phenomenon: corresponding to the tangent-space of a manifold there is the co-tangent space spanned by differentials of smooth functions, corresponding to the space of Weaver derivations there is the Weaver differential  \cite{heinonennonsmooth, weaver, Sc16} and corresponding to the Cheeger tangent space  there is the Cheeger differential \cite{cheeger} --- see also \cite{teriseb,bateteriseb} for more examples of this. Associated to the coarse tangent field there is also a notion of differential $df$. We indicate briefly its definition and how our result implies its existence, and leave a more detailed study of it for future work.

Consider now a Lipschitz function $f:\mathbb{X}\to \R$ and a coarse plane field $\tau$ for some multiresolution family $\cF$ for $E$. The differential for $f$ should be thought of as an example of a \emph{coarse $1$-form} $\omega$, which is an association to each $B\in \cF$ of a linear map $\omega_B:\tau(B)\to \mathbb{R}$. Given a rectifiable curve $\Gamma$ we ask how good of a linear approximation $\omega$  is to $f$ along this curve. This is measured by a $\beta$-number defined as follows
\begin{equation}
\beta_{f,\Gamma}^{\omega,\tau}(B)=\frac{1}{\diam(B)} \inf\{\sup_{x\in \Gamma \cap B} |f(\gamma(t))-\omega(\pi_{\tau(B)}(\gamma(t)))-a|: a\in \R\},
\end{equation}
where $\pi_{\tau(B)}$ is the orthonormal projection onto $\tau(B)$. We say that a coarse $1$-form $df$ is a coarse $\epsilon$-differential for $f$ (with inflation factor $A\geq 1$) if 
\begin{equation}\label{eq:epsdiffdef}
 \sum_{\substack{B\in\cF\\ \Gamma \cap B \neq \emptyset \\ \diam(B) \leq \diam(\Gamma)\\ \beta^{df, \tau}_{f, \Gamma}(AB)\geq \epsilon }} \diam(B) \leq  C\ell(\Gamma),
\end{equation}

A slightly different way to think of this is given by considering the graph of $f$ and its coarse tangent fields. Given a multiresolution family $\cF$ and a Lipschitz function $f: \mathbb{X}\to \mathbb{R}$, we can associate a family of balls to the graph of $f$ by $\cF_f=\{B_f : B\in \cF\}$, where $B_f$ is the smallest ball containing $B\times f(B)$ in $\mathbb{X}\times \mathbb{R}$.  Given $\omega$ and $f$, we can define a coarse plane field on $\cF_f$ by $\tau_{f,\omega}(B_f)=\{(x,\omega(x)): x\in \tau(B)\}.$ Then, a sufficient condition for $df$ to be a coarse $\epsilon$-differential for $f$ with inflation factor $A$ is if $\tau_{f,df}$ is a coarse $\epsilon$-tangent field for $\cF_f$ with inflation factor $A$. This condition is a bit stronger than \eqref{eq:epsdiffdef} since it also assumes that $\Gamma$ is close to a line, and not just that $f$ is well-approximated by a linear function along $\Gamma$ and within the ball $B$. By applying Theorem \ref{thm:main} to the graph of $f$ we obtain the following.

\begin{corollary}
Let $\bX$ be a Hilbert space, $E\subset\bX$ a doubling subset that badly fits $d$-planes, and $\cF$ a multiresolution family for $E$. Then, for every $\epsilon>0$ and $A>1$ there exists an $\epsilon'>0$ and $A'>1$, such that $\cF$ admits a $(d-1)$-dimensional coarse $\epsilon'$-tangent field $\tau$ with inflation factor $A'$ and any $1$-Lipschitz function $f$ admits a coarse $\epsilon$-differential with respect to $\tau$ with inflation factor $A$.
\end{corollary}

\begin{proof} We sketch the argument. If $E\subset \mathbb{X}$ poorly fits $d$-planes, then $E_f=\{(x,f(x)): x\in E\}$ poorly fits $d$-planes. By Theorem \ref{thm:main}, there exists $\epsilon'$-coarse tangent field $\tau$ with inflation factor $A'>1$ for $E$. Now, $\cF_f$ may not quite be a multi-resolution family for $E_f$, but a standard doubling argument implies that there does exist a multiresolution family $\cF_{f,2}$ of $E_f$ so that for each $B\in \cF_f$ we have some ball $B_2\in \cF_{f,2}$ with at most twice $B$'s radius such that $AB\subset 2AB_2$.  Theorem \ref{thm:main} gives the existence of a coarse $\epsilon$-tangent field $\tau_2$ for $\cF_{f,2}$ with inflation factor $2A$. From this, it is easy to see that a coarse $\epsilon$-tangent field  for $\cF_{f,2}$ defines a coarse $\epsilon$-tangent field $\overline{\tau}_f$ for $\cF_f$. We can assume that $\tau_\pi(B)=\pi_{\mathbb{X}}(\overline{\tau}_f(B_f))$ satisfies $\tau_\pi(B)\subset \tau(B)$, since if this does not hold, then we can replace the field by $\overline{\tau}_f(B_f)\cap \pi_{\mathbb{X}}^{-1}(\tau(B))$.  (Here $\pi_{\mathbb{X}}, \pi_\mathbb{R}$ are the orthonormal projections on the factors of $\mathbb{X}\times \mathbb{R}$.)  Indeed, if $\Gamma$ is a rectifiable curve in $\mathbb{X}\times \R$, then we can first estimate the scales where $\pi_{\mathbb{X}}(\Gamma)$ is far away from to $\tau(B)$, and removing these we have that $\Gamma$ is well-approximated by a line contained in $\pi_{\mathbb{X}}^{-1}(\tau(B))$.

Once $\epsilon>0$ is small enough, one can find an $\omega$ s.t. $\overline{\tau}_f\subset \tau_{f,\omega}$ for some coarse $1$-form $\omega$.   Indeed, $\overline{\tau}_f$ is a graph over $\tau_\pi(B)=\pi_{\mathbb{X}}(\overline{\tau}_f(B_f))\subset \tau(B)$ with respect to the linear map $\omega_1(B)=\pi_{\mathbb{R}} \circ (\pi_{\mathbb{X}}|_{\overline{\tau}_f(B)})^{-1}$. The form $\omega$ is then defined by possibly extending $\omega_1$ from $\tau_{\pi}$ to all of $\tau(B)$.  The claim then follows from the discussion preceding the statement, since $\tau_{f,\omega}$ is a coarse field for $E_f$.   
\end{proof}
If $d=2$, then the differential and field can be chosen independent of $\epsilon>0$ by applying Theorem \ref{thm:onedim} instead of Theorem \ref{thm:main} in the argument.

\subsection{Outline of the paper}
Sections \ref{sec:prelim} contains some notation as well as some background results on dyadic cubes in metric spaces and the relation between dimension theory and the ``badly fitting $d$-planes'' property. Section \ref{sec:pointwise} contains the proof of our ``pointwise'' result, Theorem \ref{thm:pointwise}.

The remainder of the paper, taking up most of the length, is devoted to the results on coarse tangent fields. Section \ref{sec:TST} contains some results connected to the Analyst's Traveling Salesman Theorem of Jones; the main result there is Proposition \ref{prop:str}.

The proofs of the ``coarse'' Theorems \ref{thm:main}, \ref{thm:onedim}, and \ref{thm:converse} are in Sections \ref{sec:hilbertproof} and \ref{sec:finalproofs}. Before that, in Section \ref{sec:d=2}, we give a sketch of the main argument in the special case $\bX=\mathbb{R}^2$ and $d=2$. Section \ref{sec:d=2} is not logically necessary, and we include it only to help the reader follow the main ideas in the following two sections more closely.

Finally, Section \ref{sec:example} contains the example that proves Theorem \ref{thm:example}.

\section{Preliminaries}\label{sec:prelim}

\subsection{Basics}
Throughout the paper, $\bX$ will denote a (finite or infinite dimensional) Hilbert space. We use the notation $\angle(v,w)$ for the angle between two vectors in $\bX$. If $\tau,\sigma$ are subspaces of $\bX$, we define the angle from $\tau$ to $\sigma$ by
$$ \angle(\tau,\sigma) = \sup_{v\in\tau} \inf_{w\in\sigma} \angle(v,w)$$
and the ``two-sided'' distance between the subspaces by
$$ D(\tau,\sigma) = \max\{\angle(\tau,\sigma), \angle(\sigma,\tau)\}.$$
We write $\langle \tau,\sigma \rangle$ to denote the subspace of $\bX$ spanned by $\tau\cup\sigma$. For $x,y\in \mathbb{X}$ we denote the line segment connecting them by $[x,y]$.

A closed ball in $\bX$ centered at $x$ and with radius $r$ is denoted $B(x,r)$. If $\delta>0$ and $E\subseteq\bX$, then the \emph{$\delta$-neighborhood} of $E$ is
$$ N_\delta(E) = \{x\in\bX: \dist(x,E) \leq \delta\}.$$
A maximal $\delta$-separated subset of $E$ is called a \emph{$\delta$-net} in $E$. If $E$ is bounded, then standard arguments imply that $E$ admits a family $\{N_k\}_{k\in\mathbb{Z}}$ of $2^{-k}$-nets that are nested, i.e., $N_k\subseteq N_{k+1}$ for each $k\in\mathbb{Z}$. See also Remark \ref{rmk:unnested}.

A \emph{rectifiable curve} in $\bX$ is a Lipschitz image of $[0,1]$ in $\bX$ or, equivalently, a compact, connected set of finite one-dimensional Hausdorff measure, which we write as $\ell$.

\subsection{Dyadic cubes} 
In the arguments below, we use a construction due to Christ \cite{Christ} of a system of ``dyadic cubes'' in metric spaces (which for us will only be subsets of Hilbert space). Our requirements on these cubes are quite weak, and we do not need, for instance, any information about the size of the boundaries of the cubes. Essentially, all that we need from these constructions is an arrangement of the balls in a multiresolution family into a tree structure.

\begin{proposition}[Christ]\label{prop:christ}
Let $(M,d)$ be a metric space. There is a collection of open subsets $\{Q^k_\alpha\subseteq M : k\in\mathbb{Z}, \alpha\in I_k\}$ and constants $s=2^{-N}$ for some $N\in\mathbb{N}$, $a_0>0$, and $C_1>0$ such that the following hold:
\begin{enumerate}[(i)]
\item If $k \leq \ell$, $\alpha\in I_k$, $\beta\in I_\ell$, then either $Q^\ell_\beta \subseteq Q^k_\alpha$ or $Q^\ell_\beta \cap Q^k_\alpha = \emptyset$.
\item For each $(k,\alpha)$ and $\ell<k$, there is a unique $\beta \in I_l$ with $Q^k_\alpha\subseteq Q^\ell_\beta$. 
\item For each $k\in\mathbb{Z}$ and $\alpha\in I_k$, there is a point $z^k_\alpha$ with
$$ B_M(z^k_\alpha, a_0 s^k) \subseteq Q^k_\alpha \subseteq B_M(z^k_\alpha, C_1 s^k).$$
The points $\{z^k_\alpha: \alpha\in I_k\}$ can be chosen from an arbitrary $s^k$-net.
\item If $M$ is a subset of a Hilbert space $\bX$ and we set $B(Q^k_\alpha)=B(z^k_\alpha, C_1 s^k)\subset \bX$ for each $k,\alpha$, then every ball centered on $M$ is contained in some $B(Q^k_\alpha)$ of comparable diameter.
\item If $M$ is a subset of a Hilbert space $\bX$ and $A\geq 2$, then
$$ k\leq \ell, Q^\ell_\beta \subseteq Q^k_\alpha \Rightarrow AB(Q^\ell_\beta) \subseteq AB(Q^k_\alpha).$$
\end{enumerate}
The constants $s, a_0, C_1$ are absolute and do not depend on any properties of $M$.
\end{proposition}
To observe that doubling is not required for Proposition \ref{prop:christ}, see the discussion in \cite{Christ} following the statement of that result, which says that ``(3.2) through (3.5)...concern only the quasi-metric space structure, and we have nothing at our disposal in the proof but the quasi-triangle inequality.'' Christ's items (3.2) through (3.5) are equivalent to items (i) through (iii) of Proposition \ref{prop:christ} below, while (iv) and (v) of Proposition \ref{prop:christ} follows easily from Christ's construction. The fact that we can take our parameter $s$ (which is called $\delta$ in \cite{Christ}) to be dyadic also follows from the argument in \cite{Christ}, which requires only that it is sufficiently small.

In our applications, the points $z^k_\alpha$ will always be chosen from a collection of \textit{nested} $s^k$-nets.

We write $\Delta_M$ to be the cubes constructed by Proposition \ref{prop:christ} in a given metric space $M$. With $M$ understood from context, we may just write $\Delta$ and group the cubes by scale as 
$$ \Delta_k = \{Q^k_\alpha: \alpha\in I_k\}.$$ 
If $Q=Q^k_\alpha \in \Delta_k$, then we set $z(Q)=z^k_\alpha$.

If $Q\in \Delta_n$, we write $Q^{\uparrow}$ or $Q^{\uparrow 1}$ for the unique cube in $\hDelta_{n-1}$ containing it, i.e., its ``parent''. Cubes with the same parent are ``siblings''. Inductively, we define $Q^{\uparrow (k+1)} = (Q^{\uparrow k})^{\uparrow}$ and $Q^{\uparrow 0}=Q$. If $E$ is a subset of a Hilbert space and $Q=Q^k_\alpha \in \Delta_E$, then as in (iv) we will write
$$ B(Q) = B(z^k_\alpha, C_1 s^k) \subset \bX.$$
We observe that if $M$ is a subset of a Hilbert space, then by choosing the points $\{z^k_\alpha\}$ to come from a fixed sequence of nested $2^{-nk}$-nets, we can ensure that the collection $\{B(Q):Q\in\Delta_M\}$ is a subset of a multiresolution family on $M$.

\subsection{Dimension theory}\label{subsec:dimension}
In Corollary \ref{cor:Nagata-Assuad} we use some notions of dimension.

\begin{definition}
Let $M$ be a metric space. A covering $\mathcal{B}$ of $M$ is called \emph{$r$-bounded} if each set in the covering has diameter at most $r$. The \emph{$s$-multiplicity} of $\mathcal{B}$ is the infimum of all integers $n$ such that every subset of $M$ with diameter at most $s$ meets at most $n$ sets from $\mathcal{B}$.

Finally, the \emph{Nagata dimension} of the metric space $M$ is the infimum of all integers $n$ with the following property: there is a constant $c>0$ such that, for all $s>0$, $M$ admits a $cs$-bounded cover with $s$-multiplicity at most $n+1$. 

We call $c$ the ``Nagata dimension constant'' of $M$.
\end{definition}

The Nagata dimension is a quantitative analog of the purely topological Lebesgue covering dimension. It has extensive connections to quasisymmetric and Lipschitz geometry; we refer the reader to \cite{LangSchlichenmaier} for many details.

We will see below that if a subset of a Hilbert space has Nagata dimension strictly less than $d$, then it badly fits $d$-planes. The same statement holds for another well-studied notion of dimension, the Assouad dimension, since it bounds Nagata dimension from above \cite{LeDonneRajala}. We refer the reader to, e.g., Chapter 10 of \cite{Heinonen} for a definition of Assouad dimension, which we do not use directly here.

We will also use the standard notion of a doubling metric space: A metric space $M$ is \emph{doubling} if there is a constant $N$ such that each ball in $M$ can be covered by $N$ balls of half the radius. This condition is equivalent to finiteness of the Assouad dimension and therefore implies finiteness of the Nagata dimension.

For us, the Nagata dimension enters via the following lemma:

\begin{lemma}\label{lem:notgood}
Let $\bX$ be a Hilbert space and $E$ a subset of $\bX$ with Nagata dimension at most $n$. Then $E$ badly fits $(n+1)$ -planes.  

The constant $\epsilon_0$ in Definition \ref{def:fit} depends only on $n$ and the Nagata dimension constant of $E$.
\end{lemma}
This fact follows from the more general Proposition 2.18 of \cite{LeDonneRajala}, but we give a straightforward direct proof.
\begin{proof}
Fix $\bX$, $E$ as in the assumptions. Let $c$ be the Nagata dimension constant for $E$, which we may assume is at least $1$ without loss of generality.

We begin with the following observation: There is a constant $s>0$ such that, if $D\subseteq \bX$ is isometric to a unit-diameter ball in $\mathbb{R}^{n+1}$, then $D$ does not admit an $10cs$-bounded covering with $s$-multiplicity at most $n+1$. Indeed, such a ball in $\mathbb{R}^{n+1}$ has Nagata dimension $n+1$ (see \cite[Theorem 2.2]{LangSchlichenmaier}), so such an $s$ must exist.

Let $\epsilon_0 = s/2$. We claim that $E$ badly fits $(n+1)$-planes, with parameter $\epsilon_0$. Suppose not.  Then there is a $d$-plane $V\subseteq \bX$ and a ball $B$ with $\dist(V\cap B, E)<\epsilon_0\diam(V\cap B)$. If we set $D=V\cap B$, then $D$ is isometric to a ball in $\mathbb{R}^{n+1}$, and it lies in $N_{\epsilon_0 \diam(D)}(E)$. By rescaling, we may assume that $\diam(D)=1$.

The set $E$ has Nagata dimension $n$ with constant $c$, and so it has a $3cs$-bounded cover $\mathcal{B}$ with $3s$-multiplicity at most $n+1$. We may assume that each element of $\mathcal{B}$ is contained in $E$. 

For each $B\in\mathcal{B}$, set $B' = N_s(B) \cap D$ and collect all the sets $B'$ into $\mathcal{B}'$. Then, by our assumption on $D$ and our choice of $\epsilon_0$, $\mathcal{B}'$ is a $10cs$-bounded cover of $D$.

Moreover, $\mathcal{B}'$ has $s$-multiplicity at most $n+1$. Indeed, suppose $A'\subseteq D$ and $\diam(A')\leq s$. Let $\{B'_i\}_{i\in I}$ be the sets in $\mathcal{B}'$ that touch $A'$. Then $A=N_s(A)\cap E$ has diameter at most $3s$ and each corresponding ball $B_i\in\mathcal{B}$ with $i\in I$ touches $A$. It follows that $|I|\leq n+1$ and hence $\mathcal{B}'$ has $s$-multiplicity at most $n+1$.

We thus have a $10cs$-bounded covering of $D$ with $s$-multiplicity at most $n+1$, which contradicts our observation about $(n+1)$-balls made earlier.
\end{proof}

\section{Proof of the ``pointwise'' Theorem \ref{thm:pointwise}}\label{sec:pointwise}
In this section, we prove our qualitative result on the existence of pointwise weak tangent fields, Theorem \ref{thm:pointwise}. As remarked in the introduction, this section is inspired by the work of Alberti--Cs\"ornyei--Preiss \cite{ACP}. We note however that our approach does not use the covering arguments of that work.

\begin{proof}[Proof of Theorem \ref{thm:pointwise}]
The existence of pointwise weak tangent fields is preserved under taking subsets and countable unions, and the property of badly fitting $d$-planes is preserved under taking closures. Therefore, it suffices to show that if $E$ is a subset of a Hilbert space $\bX$ that badly fits $d$-planes, then $E$ has a pointwise weak tangent field.

If $E\subset \bX$ is a compact set, we call a subspace $V\subset \bX$ a very weak sub-tangent of $E$ at $x\in E$ if 
\begin{equation}\label{eq:limittangent}
\lim_{r\to 0} \left(\frac{\sup_{y\in V\cap B(0,r)} d(x+y,E)}{r}\right)=0.
\end{equation}
Without loss of generality, we may assume that $\bX$ is the closed linear span of $E$ and thus separable.

Let $\mathcal{T}_x$ be the collection of very weak sub-tangents at $x$. Notice that every subspace in $\mathcal{T}_x$ is at most $(d-1)$-dimensional, since $E$ badly fits $d$-planes. Let $\mathcal{M}_x$ be the collection of maximal subspaces in $\mathcal{T}_x$ and define $\tau(x) = \bigcap_{M\in \mathcal{M}_x} M$. It is immediate that $\tau(x)$ has dimension at most $d-1$.

We show that $\tau$ is a tangent field. Let $\gamma:[0,1]\to \RR^d$ be a rectifiable curve parametrized by constant speed, and assume that $T=\{t\in [0,1]: \gamma(t)\in E\}$ has $|T|>0$. We aim to show that for a.e. $t\in T$ we have $\gamma'(t)\in \tau(\gamma(t))$. If not, there is a positive measure subset $T'\subset T$ so that there exists a subspace $M(\gamma(t))\in \mathcal{M}_x$ s.t. $\gamma'(t)\not\in M(\gamma(t))$ for a.e. $t\in T'$. By a fairly standard Borel-selection argument (see e.g. Bogachev \cite[Theorem 6.9.1.]{Bogachev}), we can choose for a.e. $t\in T'$ such a subspace and assume that $t\to M(\gamma(t))$ is measurable. (It is not too hard to see that $\mathcal{M}_x,\mathcal{T}_x$ define Borel subsets in $\bX\times \bigcup_{d\in \mathbb{N}} GL(d,\bX)$.)

By using a decomposition argument with respect to dimension and a standard Lusin's theorem argument, we can pass to a compact subset $T''\subset T'$ of positive measure, for which we have that $M(\gamma(t))$ varies continuously with respect to $t\in T''$. Further, by Egorov's theorem, we can assume that the limit in \eqref{eq:limittangent} is uniform for all $x=\gamma(t)$ and $V=M(\gamma(t))$. 

Next, for a.e. $t\in T''$ we have that $\gamma(t)$ is differentiable, $\gamma'(t)\neq 0$, $t$ is a Lebesgue point of $T''$, and $t$ is an approximate continuity point of $\gamma'$. Fix such a $t$ and let $x=\gamma(t)$ and $M'=\langle M(\gamma(t)),\gamma'(t)\rangle$. We argue that $M'\in \mathcal{T}_x$, which contradicts the maximality of $M(x)$ and yields our claim.

By differentiability, approximate continuity and the Lebesgue point property, we have for every $s\in \mathbb{R}$ that there exists a $t_s\in T''$ so that $\lim_{s\to 0}\frac{d(x+s\gamma'(t),\gamma(t_s))}{|s|}=0$

Let $y\in B(0,r)\cap M'$ and decompose $y=s_y\gamma'(t)+y_{M}$ for $y_M\in M$ and $s_y\in \mathbb{R}$. Then 
\begin{align*}
&\limsup_{r\to 0} \frac{\sup_{y\in M'\cap B(0,r)} d(x+y,E)}{r} \\
&\hspace{50pt}\leq\limsup_{r\to 0} \frac{\sup_{y\in M'\cap B(0,r)} d(x+y,\gamma(t_{s_y})+y_{M}) + d(\gamma(t_{s_y})+y_{M},E)}{r}\\
&\hspace{50pt}=\limsup_{r\to 0} \frac{\sup_{y\in M'\cap B(0,r)} d(x+s_y\gamma',\gamma(t_{s_y})) + \sup_{y\in M'\cap B(0,r)} d(\gamma(t_{s_y})+y_{M},E)}{r}=0.
\end{align*} 
In the final step, the first term tends to zero since $\lim_{r\to 0}\sup_{y\in B(0,r)\cap M'}\frac{|s_y|}{r}<\infty$, and the second tends to zero because the limit in \eqref{eq:limittangent} converges to zero uniformly for $t\in T''$ and since $\lim_{r\to 0}\sup_{y\in B(0,r)\cap M'}\frac{|y_M|}{r}<\infty$. 
\end{proof}

\section{Some necessary results related to the Traveling Salesman Theorem}\label{sec:TST}

Our main goal in this section is to prove Proposition \ref{prop:str} below, which will be essential to the proof of Theorem \ref{thm:main}. Many of the ideas in this section are contained in the literature in some ways, although we do not know that they have been assembled in this form before.

Throughout this section, we fix a (finite or infinite dimensional) Hilbert space $\bX$, a doubling subset $E\subseteq \bX$, and a rectifiable curve $\Gamma \subseteq \bX$. With $s\in (0,1)$ as in Proposition \ref{prop:christ}, we fix nested sequnces of nets for both $\Gamma$ and $E$ and associated collections $\Delta_\Gamma$ of cubes for $\Gamma$ and $\Delta_E$ of cubes for $E$ as in Proposition \ref{prop:christ}.

Before stating Proposition \ref{prop:str}, we need a few preliminary facts. The first is the following strengthening of Theorem \ref{thm:tst}(ii), which is actually already a consequence of the proof of the theorem. 

\begin{proposition}\label{p:TST-betas-bounded}
Let $\GuyLambda>1$ be given. 
Then 
\begin{equation}\label{eq:tstGamma}
\sum_{ \substack{Q\in \Delta_\Gamma\\ \diam(B(Q))\leq \diam(\Gamma)}} \theta_\Gamma(\GuyLambda B(Q))^2\diam(B(Q)) \lesssim \ell(\Gamma).
\end{equation}
and
\begin{equation}\label{eq:tstE}
\sum_{ \substack{Q\in \hDelta_E\\ \GuyLambda B(Q)\cap \Gamma\neq \emptyset\\ \diam(B(Q))\leq \diam(\Gamma)}} \theta_\Gamma(\GuyLambda B(Q))^2\diam(B(Q)) \lesssim \ell(\Gamma).
\end{equation}

The implied constants depends only on $\GuyLambda$ and (in \eqref{eq:tstE}) on the doubling constant of $E$.
\end{proposition}

\begin{remark}\label{r:tst-half-1}
        We remark that Proposition \ref{p:TST-betas-bounded} relies on the Pythagorean inequality $c^2-a^2\geq  h^2$, where $c$ is the hypotenuse of a right angle triangle with leg $h$ and base $a$. In a non-Hilbertian Banach space, similar results are known to sometimes hold, but with modified exponents.  See  Theorem 1.7 of \cite{BadgerMcCurdy}
				as well as Corollary D of \cite{DavidSchul}.
    
\end{remark}
\begin{proof}[Proof of Proposition \ref{p:TST-betas-bounded}]
We remark first that \eqref{eq:tstE} follows easily from \eqref{eq:tstGamma} (for a comparable $\GuyLambda$). If $Q\in\hDelta_E$ is as in \eqref{eq:tstE}, then $\lambda B(Q)$ is contained in $\GuyLambda' B(R)$ for some cube $R\in\Delta_\Gamma$  as in \eqref{eq:tstGamma} with $\diam(B(R))\approx \diam(B(Q))$ and $\GuyLambda' \approx \GuyLambda$. The number of such $B(Q)$ associated to a fixed $B(R)$ is controlled by the doubling constant of $E$.

Therefore, we may focus on \eqref{eq:tstGamma}. We may assume that $\Gamma$ has finite length, otherwise there is nothing to show. Fix a Lipschitz parametrization $\gamma\colon [0,1] \rightarrow \Gamma$ of $\Gamma$ with Lipschitz constant controlled by $2\ell(\Gamma)$. (See, e.g., Theorem 4.4 of \cite{AlbertiOttolini}.) 
By Theorem \ref{thm:tst}, we have that 
    \begin{equation}\label{e:tomato-0}
     \sum_{Q\in \Delta_\Gamma, \diam(Q)\leq \diam(\Gamma)} \beta_\Gamma(10\GuyLambda B(Q))^2\diam(Q) \lesssim_{\GuyLambda} \ell(\Gamma).
    \end{equation}
Indeed, the balls $B(Q)$ for $Q\in\Delta_\Gamma$ form a subset of a multiresolution family on $\Gamma$, and so this follows from Theorem \ref{thm:tst}.

Thus, in proving the proposition, we may reduce to considering only cubes $Q\in \Delta_\Gamma$ with $\beta_\Gamma(10\GuyLambda B(Q))<\delta$, for a fixed parameter $\delta>0$ to be chosen below.  We may further assume that $\Gamma\setminus 10\GuyLambda B(Q)\neq\emptyset$ for each $Q$ in the sum, as the sum over cubes not satisfying this condition is bounded by $\sim\diam(\Gamma)$.

We introduce some notation from \cite{Krandel-TST}. An \emph{arc} in $\gamma$ is the restriction of $\gamma$ to a closed interval. If $B$ is a ball in $\bX$, then $\Lambda(B)$ denotes the collection of (maximal) arcs with domains in $\gamma^{-1}(\overline{2B})$ whose images intersect $B$. If $\tau$ is an arc in $\gamma$, then
$$\tilde{\beta}(\tau) = \sup_{x\in\text{image}(\tau)} \frac{\dist(x,[\tau(a),\tau(b)])}{\diam(\text{image}(\tau))}.$$

For each $Q\in\Delta_\Gamma$, choose any $\gamma^0_Q\in \Lambda(\GuyLambda B(Q))$. From Lemma 3.4 and Lemma 3.6 of \cite{Krandel-TST}, we are guaranteed an extension $\gamma_Q$ of $\gamma^0_Q$ by  subarcs (of $\gamma$) of total diameter $<\diam(\gamma^0_Q)/(10\GuyLambda) <\diam(B(Q))/10$ 
and so that 
\[\sum_{Q\in \Delta_\Gamma}\tilde{\beta}(\gamma_Q)^2\diam(\gamma_Q) \lesssim \ell(\Gamma)\]
Thus, in proving our proposition, we may further reduce to looking at cubes so that 
$$\tilde{\beta}(\gamma_Q)<\delta\theta_\Gamma(\GuyLambda B(Q)).$$
After these reductions, we are now left with cubes $Q$ with corresponding arcs $\gamma_Q$ so that 
\begin{itemize}
\item  $\gamma_Q$ intersects $\GuyLambda B(Q)$ and leaves $2\GuyLambda B(Q)$
\item $\tilde{\beta}(\gamma_Q)<\delta \theta_\Gamma(\GuyLambda B(Q)) <\delta$
\item $\beta_\Gamma(10\GuyLambda B(Q))<\delta$
\end{itemize}
Let $L_Q$ be the line going through the endpoints of $\gamma_Q$, 
and $I_Q\subset L_Q$ the interval connecting these endpoints.
Then 
\begin{equation}\label{e:tomato-1}
\beta_\Gamma(3\GuyLambda B(Q))\diam(B(Q))\gtrsim \sup\{\dist(x,L_Q):x\in \GuyLambda B(Q)\cap \Gamma\}.
\end{equation}

Each point in $L_Q\cap \GuyLambda B(Q)$ has a point in $\gamma_Q$ within distance  comparable to $\tilde{\beta}(\gamma_Q)\diam(B(Q))<\delta\theta_\Gamma(\GuyLambda B(Q))\diam(B(Q))$. With $\delta>0$ being sufficiently small (depending on the implied constants),
this implies that 
\[\theta_\Gamma(\GuyLambda B(Q))\diam(\GuyLambda B(Q))<2\sup\{\dist(x,L_Q):x\in \GuyLambda B(Q)\cap \Gamma\}.\]
Combining this with \eqref{e:tomato-1} and \eqref{e:tomato-0} we have the proposition.
\end{proof}

We next need a lemma that says that $\Gamma$ is near $E$ in ``most'' of the cubes of $\hDelta_E$ that it touches. Given a ball $B\subset \bX$, we write
$$ d_{\Gamma,E}(B) = \diam(B)^{-1}\sup\{\dist(x,E):x\in B\cap \Gamma\},$$
with the understanding that $d_{\Gamma, E}(B) = 0$ if $\Gamma \cap B =\emptyset$.

\begin{lemma}\label{l:TST-analogue-for-d}
Given $\GuyLambda>1$, we have
$$ \sum_{\substack{Q\in \hDelta_E\\ \diam(B(Q))\leq \diam(\Gamma)}} d_{\Gamma,E}(\GuyLambda B(Q))^2\diam(B(Q)) \lesssim \ell(\Gamma),
$$
where the implied constant depends only on $\GuyLambda$ and the doubling constant of $E$.
\end{lemma}
\begin{proof}
Throughout the proof, we will write $d$ for $d_{\Gamma,E}$. Observe that each cube $Q\in\hDelta_E$ with $\GuyLambda B(Q)\cap\Gamma\neq\emptyset$ and $\diam(B(Q))\leq \diam(\Gamma)$ satisfies
$$ d(\GuyLambda B(Q)) \leq 1$$
and
\begin{equation}\label{eq:d}
d(\GuyLambda B(Q))\diam(B(Q)) \leq \diam(B(Q)) \leq \diam(\Gamma). 
\end{equation}

It suffices to consider $Q$ with $d(\GuyLambda B(Q))\diam(B(Q))>0$. To each such $Q\in \hDelta_E$ with $\Gamma$ intersecting $\GuyLambda B(Q)$, assign an arc $A_Q$ of $\Gamma\cap 2\GuyLambda B(Q)$ with 
$$ \dist(x, E) \approx_{\GuyLambda} d(\GuyLambda B(Q))\diam(B(Q)) \approx_{\GuyLambda} \ell(A_Q) \text{ for all } x\in A_Q.$$
This can be done by choosing a point $z\in\Gamma\cap \GuyLambda B(Q)$ with $\dist(z,E)\geq \frac{1}{2}d(\GuyLambda B(Q))\diam(\lambda B(Q))$ and setting $A_Q$ to be an arc of $\Gamma$ containing $z$ of length $\frac{1}{100\GuyLambda}d(\GuyLambda B(Q))\diam(\lambda B(Q))$, which must exist by \eqref{eq:d}.

Fix $x\in \Gamma$. The doubling property of $E$ implies that  there are a bounded number of cubes $Q\in\hDelta_E$ at each fixed scale with $x\in A_Q\subseteq 2\lambda B(Q)$. Let $Q_0$ be a minimal cube such that $x\in A_{Q_0}$. In that case, $\dist(x,E)\approx d(\GuyLambda B(Q_0))\diam(B(Q_0))$. If $Q\in\hDelta_E$ and $x$ is also in $A_Q$, then  $\dist(x,E)\approx d(\GuyLambda B(Q))\diam(B(Q))$ and so
$$ d(\GuyLambda B(Q)) \approx d(\GuyLambda B(Q_0)) \left(\frac{\diam(B(Q_0))}{\diam(B(Q))}\right).$$
Summing a geometric series, it follows that
$$ \sum_{\substack{Q\in\hDelta_E\\ \diam(B(Q))\leq \diam(\Gamma)\\ x\in A_Q}} d(\GuyLambda B(Q)) \lesssim d(\GuyLambda B(Q_0)) \lesssim 1.$$

We then have
\begin{align*}
 \sum_{\substack{Q\in \hDelta_E\\ \diam(B(Q))\leq \diam(\Gamma)}} d(\GuyLambda B(Q))^2\diam(Q) &\lesssim \sum_{\substack{Q\in \hDelta_E\\ \diam(B(Q))\leq \diam(\Gamma)}}  d(\GuyLambda B(Q)) \ell(A_Q)\\
 &= \int_\Gamma \left(\sum_{\substack{Q\in\hDelta_E\\ \diam(B(Q))\leq \diam(\Gamma)\\ x\in A_Q}} d(\GuyLambda B(Q)) \right) dx\\
 &\lesssim \ell(\Gamma).
\end{align*}

\end{proof}

Let us define a \emph{stopping time region} as a collection $\cS\subseteq \hDelta_E$ with the properties that
\begin{enumerate}
    \item $\cS$ contains a ``top'' cube $Q(S)$ that contains all other cubes of $S$.
    \item $\cS$ is ``connected'': if $Q \subseteq Q' \subseteq Q''$ and both $Q, Q''$ are in $\cS$, then so is $Q'$. 
\end{enumerate}
Note: the definition of stopping time region employed often by David and Semmes, e.g. in \cite{DavidSemmes}, typically also requires that for each $Q\in S$, either all children of $Q$ lie in $\cS$ or none of them do. We will consider a version of this later on in the proof of Lemma \ref{l:first-stop-construction}, but we do not include it in the definition here.

Let $\cC_\Gamma = \{Q\in\hDelta_E : B(Q) \cap \Gamma \neq \emptyset, \diam(B(Q))\leq \diam(\Gamma)\}$. 
Note that $\cC_\Gamma$ is ``upwards connected'' in the sense that if $Q\in\cC_\Gamma$, then so are all its ancestors $Q^{\uparrow k}$ with $\diam(Q^{\uparrow k})\leq \diam(\Gamma)$. This follows from Proposition \ref{prop:christ}(v).

We need the following result, which is a type of ``coronization'' (in the language of \cite{DavidSemmes}). It says that, at the expense of throwing away a controlled collection of ``bad'' cubes, the collection $\cC_\Gamma$ can be partitioned into stopping time regions on which $\Gamma$ is very flat, very close to $E$, and within which the line of best fit does not change too much. Although similar results are known in many settings, the facts that $\Gamma$ is not Ahlfors regular and that we work in Hilbert space rather than $\mathbb{R}^n$ means that we cannot directly quote from the literature.

\begin{proposition}\label{prop:str}
Let $\Gamma$ be a rectifiable curve in $\bX$ and $\alpha,\delta, d_0>0$, $\GuyLambda>1$. The collection $\cC_\Gamma$ can be written as a partition
$$ \cC_\Gamma = \mathcal{B} \cup \bigcup_i \cS_i$$
with the following properties:
\begin{enumerate}
    \item Each $\cS_i$ is a stopping time region.
    \item If $R\in \cS_i$, then $d_{E,\Gamma}(\GuyLambda B(R)) < d_0$. 
    \item If $R\in \cS_i$, then $\theta_\Gamma(\GuyLambda B(R))<\delta$. 
    \item If $R\in \cS_i$, then $\angle(\tau_\Gamma(Q(\cS_i)), \tau_\Gamma(R)) < \alpha$, where $\tau_\Gamma(Q)$ denotes the direction of a line minimizing the infimum in $\theta_\Gamma(\GuyLambda B(Q))$.
    \item $\sum_{i} \diam(B(Q(\cS_i))) \lesssim \ell(\Gamma)$.
    \item $\sum_{Q\in\mathcal{B}} \diam(B(Q)) \lesssim \ell(\Gamma)$.
\end{enumerate}
The implied constants depend only on $d,d_0,\alpha,\delta, \GuyLambda$, and the doubling constant of $E$.
    
\end{proposition}

The rest of this section is devoted to proving Proposition \ref{prop:str}. 
While the goal is to decompose $\mathcal{C}_\Gamma\subseteq \hDelta_E$ into stopping time regions, we will first decompose $\Delta_\Gamma$. We will then use this decomposition to get an analogous one for $\cC_\Gamma$.

We will construct the decomposition of $\Delta_\Gamma$ in a few steps.  There will be a preliminary partition
\begin{equation}\label{e:decomp-0}
    \Delta_\Gamma = \cB^0 \cup \bigcup_j \cS^0_j,
\end{equation}
and after that we will further break each $\cS^0_j$ into more stopping time regions and add some additional ``bad'' cubes.

A remark about constants: In the course of the proof, we will have small positive constants $\epsilon_J$ and $\epsilon_\theta$. The given constant $\alpha$ in Proposition \ref{prop:str} determines $\epsilon_J$ in Lemma \ref{l:control-stop-angle-large}. The constant $\epsilon_\theta$ is used earlier, and needs to be much smaller than $\min(\epsilon_J,1,\alpha/(1000\GuyLambda))$.

The decomposition \eqref{e:decomp-0} goes in the following way. Each cube $Q\in\Delta_\Gamma$ such that $\theta(\GuyLambda B(Q))\geq \epsilon_\theta$ or $d(\GuyLambda B(Q))\chi_E(\frac12\GuyLambda B(Q)) \geq d_0$ is placed in $\cB^0$. (Here we use the shorthand $\chi_A(B)=1$ if $A\cap B\neq\emptyset$ and $0$ otherwise.) The stopping time regions $\cS^0_j$ are then the ``connected components'' of $\Delta_\Gamma \setminus \mathcal{B}^0$, under the tree structure of $\Delta_\Gamma$. Cubes in $\Delta_\Gamma \setminus \mathcal{B}^0 = \bigcup_{j} \cS^0_j$ are exactly those that satisfy the condition:

\begin{equation}\label{stop-time-item:large-beta}
 \theta(\GuyLambda B(Q))<\epsilon_\theta \text{ and } d(\GuyLambda B(Q))\chi_E\left(\frac12\GuyLambda B(Q)\right)<d_0.
\end{equation}

This gives \eqref{e:decomp-0}.
\begin{lemma}
The collection $\cB^0$ satisfies the packing condition
$$ \sum_{Q\in \cB^0} \diam(B(Q)) \lesssim \ell(\Gamma),$$
with implied constant depending only on the parameters in Proposition \ref{prop:str}.	
    \end{lemma}
\begin{proof}
    This relies on Proposition \ref{p:TST-betas-bounded} and
  Lemma \ref{l:TST-analogue-for-d}.  The control on the collection of cubes $Q\in\Delta_\Gamma$ that satisfy $\theta(\GuyLambda B(Q))\geq\epsilon_\theta$ follows immediately from Proposition \ref{p:TST-betas-bounded} (for $\Delta_\Gamma$).

Next, we must consider $Q\in\Delta_\Gamma$ such that $d(\GuyLambda B(Q))\geq d_0$ and $\frac12 \GuyLambda B(Q) \cap E \neq \emptyset$. By Proposition \ref{p:TST-betas-bounded}, we already control those for which $\beta(10\GuyLambda B(Q)) \geq \epsilon_\theta$, so we are free to also assume the reverse inequality for these $Q$.

Each one of these cubes $Q$ has $\GuyLambda B(Q) \subseteq \GuyLambda B(R_Q)$ for some $R_Q\in\Delta_E$, $\diam(B(R_Q)) \approx \diam(B(Q))$, and $d(\GuyLambda B(R_Q)) \gtrsim d_0$. It follows from Lemma \ref{l:TST-analogue-for-d} that $\sum \diam(B(R_Q))\lesssim \ell(\Gamma)$. It therefore suffices to argue that the map $Q\mapsto R_Q$ is bounded-to-one. Fixing $R\in\Delta_E$ observe that all $Q\in \Delta_\Gamma$ of this type with $R=R_Q$ must be centered close to a fixed line segment. It follows that $\lambda B(R)$ can only contain a controlled number of different such $\lambda B(Q)$, $Q\in \Delta_\Gamma$ of scale comparable to $R$. This completes the proof.
 
  \end{proof}

Next, we will take each preliminary stopping time region $\cS^0_j$ and decompose it further as
\begin{equation}\label{e:decomp-1}
        \cS^0_j =  \bigcup_{i,j} \cS^1_{i,j}. 
    \end{equation}

First consider a maximal stopping time region $\cS\subseteq \cS^0_j$ with top cube $Q(\cS^0_j)$ that has the following property: For each $Q\in \cS$ and each sibling $Q'$ of $Q$, 
\begin{equation}\label{stop-time-item:large-jones}
\sum_{I\in \cS, I \supset Q'} \theta^2(\GuyLambda B(I)) < \epsilon_J^2.
\end{equation} 
Because $\epsilon_\theta<\epsilon_J$ and $Q\in\cS^0_j$, such a stopping time region $\cS\subseteq \cS^0_j$ must exist. We then proceed inductively: consider a maximal cube $Q$ of $\cS^0_j$ that is not in any of the $\cS^1_{i,j}$ constructed so far. For each such cube $Q$, we construct a maximal stopping time region $\cS\subseteq \cS^0_j$ with top cube $Q$ that satisfies \eqref{stop-time-item:large-jones}, and add $\cS$ to the list of the $\cS^1_{i,j}$. This process exhausts $\cS^0_j$ into the decomposition \eqref{e:decomp-1}.

We next argue that the new stopping time regions $\cS^1_{i,j}$ have good properties (Lemma \ref{l:first-stop-construction}), and then we will control how many of them there are (Lemma \ref{l:control-large-jones-times}).

\begin{lemma}\label{l:first-stop-construction}
Suppose $\GuyLambda>2s$ and $\epsilon_\theta>0$ is sufficiently small (depending on $\GuyLambda$ and $\alpha$.)
For each $\cS^1_{i,j}$, there is a set $\Gamma^1_{i,j}$ with the following properties:
    \begin{enumerate}[(i)]

\item $\Gamma^1_{i,j}$ is a bi-Lipschitz image of an interval of size $\diam(B(Q(S^1_{i,j})))$ with constants $(1,1+c\epsilon_J^2)$,
\item for each $Q\in \cS^1_{i,j}$,  $\Gamma^1_{i,j}\cap B(Q)$ has Hausdorff distance $<\frac{\alpha}{10 \GuyLambda}\cdot \diam(B(Q))$ to $\Gamma\cap B(Q)$.
    \end{enumerate}
    \end{lemma}
    While there are similar (and in many ways more general) statements in the literature (for example, Lemma 6.1 in \cite{ENV}, Theorem 1.5 in \cite{DavidToro_holes}, and Theorem 6.3 in \cite{Lerman}), none of them fit our need exactly, so we provide a short proof here.

    \begin{remark}\label{r:tst-half-2}
        The proof of Lemma \ref{l:first-stop-construction} uses the Pythagorean inequality $c^2-a^2\leq h^2$
        where $c$ is the hypotenuse of a right angle triangle with leg $h$ and base $a$. 
        This is the opposite inequality than the one used in Remark \ref{r:tst-half-1}.
        In a non-Hilbertian Banach space, analogous results would require modifying the exponents in a similar fashion to Lemma 6.1 of \cite{ENV}. %Edelen-Naber-Valtorta "EFFECTIVE REIFENBERG THEOREMS IN HILBERT AND BANACH SPACES" 2019, specifically lemma 6.3 in the arxiv version.
    \end{remark}
    \begin{proof}[Proof of Lemma \ref{l:first-stop-construction}]
        Let $\cS=\cS^1_{i,j}$ and set $X_{\cS}=\{z(Q):Q\in \cS\}\subseteq \Gamma$.
We will also denote by $(X_{\cS})_n$ all the points in ${\cS}$ corresponding the the centers of the $0,1,...,n$-th generation children of $Q({\cS})$ In particular, note that  $(X_{\cS})_0$ is a singleton and $(X_{\cS})_n\subset (X_{\cS})_{n+1}$.       
        We claim that $X_{\cS}$ can be ordered so that, for each ball $B(Q)$ with $Q\in \cS$, the projection of 
        points of an $\approx \epsilon_\theta \diam(B(Q))$-net in $(X_{\cS})\cap B(Q)$ onto a best fitting line is either increasing or decreasing in order. 
				We construct this order by induction as follows.
        
        The singleton $(X_{\cS})_0$ clearly has an order as desired.
        Suppose that $(X_{\cS})_n$ has an order as desired.  We wish to extend it to an order on $(X_{\cS})_{n+1}$. 
        Let $x,y\in  (X_{\cS})_{n+1}$, and assume without loss of generality that $x\notin (X_{\cS})_n$.
 
Suppose first that $x$ and $y$ are contained in some $Q\in \cS$ that is an $n$th generation descendant of $Q(\cS)$. Since $\theta_\Gamma(\GuyLambda B(Q))<\epsilon_\theta$, we may order $x,y$ in a way consistent with both the order on $(X_{\cS})_n$ and the projection onto the best-fitting line for $\GuyLambda B(Q)$.

If $x,y$ are not contained in a common $n$th generation $Q\in \cS$, then $x$ and $y$ have distance much greater than the diameter of a $n$-th generation descendant of $Q(\cS)$. In this case, the order between $x$ and $y$ is decided using the nearest points of $(X_{\cS})_n$.

This order satisfies the desired property. Now we use the order as follows.

Add to $X_{\cS}$ at most two points (from $\Gamma$), so that there is a first and last point in $X_{\cS}$.
Let $X_n$ be a sequence of nested nets for $X_{\cS}$, so that 
$X_0$ has  exactly  two points $x_0^0, x_0^1$ which are the smallest and largest in $\Gamma$ with respect to the order.
Use $X_n$ to build a polygon whose edges are always between consecutive points with respect to the order. Let $\gamma_n$ be the ordered constant speed parametrization of this polygon by $[0,1]$.
Since $\GuyLambda > 2s$, the polygon parametrized by $\gamma_n$ is contained in the union of 
$\{\frac12\GuyLambda B(Q): z(Q)\in (X_{\cS})_n\}$.
It follows from the Pythagorean theorem that the growth in the derivative $\gamma_{n+1}'(x)$ compared to that of $\gamma_n'(x)$ is bounded by a factor of $1+c\theta_\Gamma(\GuyLambda B(Q))^2$ for some $c\geq 1$, where $\frac12\lambda B(Q)\ni \gamma_n(x)$. 
Thus, 
 $\gamma_n$ is bi-Lipschitz with constant controlled by $|X_0^0-X_0^1|\cdot (1,1+2c\epsilon_J^2)$ (using that $e^x\leq 1+2x$ for small $x$).
        Finally, take $\gamma$ to be the limit of a subsequence of the $\gamma_n$.
        Note that $\theta(\GuyLambda B(Q))<\epsilon_\theta<\alpha/1000\GuyLambda$, and so 
      $\gamma$ approximates $\Gamma^1_{i,j}$ up to error $\alpha/500\GuyLambda$ in each $B(Q), Q\in\cS$.

    \end{proof}

\begin{lemma}\label{l:control-large-jones-times}
The collection $\{\cS^1_{i,j}\}$ satisfies the packing condition
$$ \sum_{i,j} \diam(B(Q(\cS^1_{i,j}))) \lesssim \ell(\Gamma).$$
    \end{lemma}
    The proof below follows the ideas from Lemma 11.4 in \cite{AzzamSchul}. Again, as in Remark \ref{r:tst-half-1}, part of the Pythagorean Theorem is used.
		
\begin{proof}

Fix a single stopping time region $\cS=\cS^0_j$ as in \eqref{e:decomp-0}.
Let $\min(\cS)$ denote its collection of minimal cubes (that are not singletons) and let $\minUnion(\cS^0_j)=\bigcup_i \min(\cS^1_{i,j})$.
For $R\in \minUnion(S^0_j)\setminus \mathcal{B}^1_j$, denote by $\cS(R)$ the stopping time region $\cS^1_{i,j}$ that contains $R$.
Then for all $R\in \minUnion(\cS^0_j)$,
$$\sum_{Q:\cS(R)\ni Q\supset R}\theta^2(\GuyLambda B(Q))\approx \epsilon_J^2.$$ 
The upper bound holds by construction of $\cS^1_{i,j}$. For the other direction, observe that $R$ has a child $R'$ that is not in $\cS(R)$, and $\theta_\Gamma(\lambda B(R'))\lesssim \theta_\Gamma(\lambda B(R))$.  Thus, whether $\cS(R)$ stopped because of condition \eqref{stop-time-item:large-beta} or \eqref{stop-time-item:large-jones}, we have the lower bound.

We may therefore write
\begin{align}
   \sum_{j} \sum_{R\in \minUnion(\cS^0_j)} \diam(B(R)) &\lesssim
    \frac1{\epsilon_J^2}  \sum_{j}   \sum_{R\in \minUnion(\cS^0_j)} \diam(B(R)) \sum_{Q:\cS(R)\ni Q\supset R}\theta^2(\GuyLambda B(Q))\\
    &\leq \frac1{\epsilon_J^2} \sum_{j}   \sum_{i} \sum_{Q\in \cS^1_{i,j}} \theta^2(\GuyLambda B(Q))
                       \sum_{R\in \minUnion(\cS^0_j):R\subset Q\in \cS(R)} \diam(B(R))
\label{e:sum-for-2.3-1}
\end{align}

Fix for the moment a cube $Q\in \cS^1_{i,j}$, for some $i,j$, and recall the (bi-Lipschitz) curve $\Gamma^1_{i,j}$ constructed in Lemma \ref{l:first-stop-construction}. We will consider the cubes $R$ appearing on the right-hand side of \eqref{e:sum-for-2.3-1} for this fixed $Q$. These are the cubes $R\in \min(\cS^1_{i,j})$ contained in $Q$, and so for one thing they are disjoint cubes. For each of these $R$, Lemma \ref{l:first-stop-construction} (and Proposition \ref{prop:christ}) allows us to choose a ball $B^R\subseteq B(R)$ with
$$ \diam(B^R \cap \Gamma^1_{i,j}) \gtrsim\diam(R)$$
and such that all these balls $B^R$ are disjoint.

Thus we get that the right-hand side of \eqref{e:sum-for-2.3-1} is bounded by 
$$\approx \frac1{\epsilon_J^2}  \sum_{j}  \sum_{i} \sum_{Q\in S^1_{i,j}} \theta^2(\GuyLambda B(Q))\diam(B(Q))\lesssim \ell(\Gamma),$$
where in the last inequality we used 
Proposition \ref{p:TST-betas-bounded}.

Finally, we observe that each $Q(\cS^1_{i,j})$ must be a child of some cube in some $\minUnion(\cS^0_j)$ or some cube in some $\cB^0_j$, and all of these have now been controlled.

\end{proof}    
We next argue that each bi-Lipschitz curve $\Gamma^1_{i,j}$ does not have too much ``rotation''. (In $\RR^n$ this would follow from standard coronization results for uniformly rectifiable sets, but this does not seem to have been written down in Hilbert space.)

To fix some notation, let $\Delta[0,1]$ be the standard dyadic filtration on $[0,1]$, let $g:[0,1]\to \bX$ be bi-Lipschitz, and let $\cS\subset \Delta[0,1]$ be a stopping time region with $Q(\cS)=[0,1]$. We write $g_{\cS}$ to denote the ``$\cS$-polygonal approximation'' to $g$. This means that
\begin{itemize}
\item $g_{\cS}$ is continuous,
\item $g_{\cS}=g$ at all endpoints of intervals in $\cS$, 
\item $g_{\cS}=g$ at all points that are contained in arbitrarily small intervals of $\cS$,
\item if $I$ is a minimal interval in ${\cS}$, then $g_{\cS}$ is affine on $I$, and
\item if $I$ is a dyadic interval not in $\cS$ but whose sibling is in $\cS$, then $g_{\cS}$ is affine on $I$.
\end{itemize}

\begin{lemma}\label{l:control-stop-angle-large}
    Let $\alpha_0>0$ be given. 
    There is a $t_0>0$ and $q\in (0,1)$ so that for  $t\in [0,t_0)$ the following holds:
		
    Let $g:[0,1]\to \bX$ be $(1,1+t^2)$-biLipschitz and let $\cS\subset \Delta[0,1]$ be a stopping time region with $Q(\cS)=[0,1]$. Then 
    there is  a stopping time region $\cS'\subset \cS$ such that 
		\begin{enumerate}[(i)]
		\item $Q(\cS')=Q(\cS)$,
		\item $g_{\cS'}$ is the graph of a Lipschitz function with Lipschitz constant at most $\alpha_0$, and
		\item the collection of minimal intervals of $\cS'$ that are not minimal in $\cS$ has total length at most $1-q$.
		\end{enumerate}
		 
\end{lemma}
In other words, we cut off $\cS$ wherever $g$ turns too much, but we are able to do it infrequently enough that the polygonal approximations to $g$ did not change too often. 
\begin{proof}[Proof of Lemma \ref{l:control-stop-angle-large}]  

    We may assume $t_0<\min(\alpha_0/20,1)$. For each dyadic interval $I=[a,b]\subseteq [0,1]$, let $L_I = [g(a), g(b)]\subseteq \bX$. Note that if $J$ is a child of $I$, then the angle between $L_I$ and $L_J$ is at most $10t_0$.
    Set  $L_{\cS}=L_{[0,1]}$. If $L$ is the full line containing $L_{\cS}$, let us write $\bX = L \times \mathbb{Y}$. 
    Add $[0,1]$ to $\cS'$, and continue adding children $I$ from $\cS$ until
    the angle between $L_I$ and $L_S$ is more than $\alpha_0-10t_0$.

After this process completes, $g_{\cS'}$ can be viewed as the graph of a Lipschitz function $h\colon L_{\cS} \rightarrow \mathbb{Y}$. If $I$ is a minimal interval of $\cS'$ that is \textbf{not} a minimal interval of $\cS$, then $|h'|\in[\frac12\alpha_0, \alpha_0]$ on $I$. 

Such an interval therefore contributes to the length of $g_{\cS'}([0,1])$ by a multiplicative  factor of at least $\geq 1+\frac18 \alpha_0^2$ (times the size of its projection onto $L_S$). However, the length of the graph is  bounded by the length of $L_S\cdot(1+t^2)$.
So the sum of the lengths of new intervals where $S'$ had to be stopped have total length ( projected onto $L_S$) bounded by 
$$|L_S| \frac{1+t^2}{1+\frac18\alpha_0^2}$$

Since we may take  $t^2\leq t_0^2\ll \alpha_0^2$, there is a $q>0$ such that 
$\frac{1+t^2}{1+\frac12\alpha_0^2}<1-q$.  
\end{proof}

We are now ready to complete the proof of Proposition \ref{prop:str}. 

Consider a single stopping time region $\cS^1_{i,j}\subseteq \Delta_\Gamma$; by rescaling if necessary we may assume that $\diam(B(Q(\cS^1_{i,j}))=1$. By Lemma \ref{l:first-stop-construction}, $\cS^1_{i,j}$ comes associated with a bi-lipschitz curve $\Gamma^1_{i,j}$, with bi-Lipschitz parametrization $\gamma^1_{i,j}\colon [0,1]\rightarrow \Gamma^1_{i,j}$.

Since $\gamma^1_{i,j}$ is bi-Lipschitz, we may build a stopping time region $\cS\subseteq\Delta[0,1]$ so that
\begin{itemize}
\item for each $I\in \cS$, there is a cube $Q\in \cS^1_{i,j}$ with $\gamma^1_{i,j}(I)$ intersecting $B(Q)$ and having comparable diameter, and
\item for each $Q\in \cS^1_{i,j}$, there is an interval $I\in \cS$ with $\gamma^1_{i,j}(I)$ intersecting $B(Q)$ and having comparable diameter.
\end{itemize}

Apply Lemma \ref{l:control-stop-angle-large} with  $g=\gamma^1_{i,j}$ and parameter $\alpha_0=\alpha/(100\GuyLambda)$ to $\cS$, creating $\cS'$. Then, on each minimal interval $I$ of $\cS'$ that is not minimal in $\cS$, repeat on the cubes of $\cS$ contained in $I$. Continue this inductively. This decomposes $\cS$ as $\cup \cS'_{k}$. Moreover, the geometric decay parameter $1-q$ in Lemma \ref{l:control-stop-angle-large} ensures that
\begin{equation}\label{eq:intervalsum}
 \sum_{k} |Q(\cS'_k)| \lesssim \diam(B(Q(\cS_{i,j}))).
\end{equation}

As above, we may now associate each $\cS'_k\subseteq \Delta[0,1]$ to a stopping time region $\cS_{i,j,k}\subseteq \cS$ so that
    \begin{equation}\label{e:decomp-2}
        \cS^1_{i,j} =  \bigcup_{i,j,k} \cS_{i,j,k}.
    \end{equation}
In each $\cS_{i,j,k}$ the set $\Gamma^1_{i,j}$ does not rotate more than $\alpha_0$, i.e., it can be viewed as a Lipschitz graph of slope $\alpha_0$. over the interval spanning $B(Q(\cS_{i,j,k}))$. Furthermore, \eqref{eq:intervalsum} implies that
\begin{equation}\label{eq:cubesum}
 \sum_{k} \diam(B(Q(\cS_{i,j,k}))) \lesssim \diam(B(Q(\cS_{i,j}))).
\end{equation}

We now use our decompositions to obtain Proposition \ref{prop:str} by declaring a cube $Q\in \cC_\Gamma$ to be in $\cB$ if $100\GuyLambda B(Q)\cap B(Q')\neq \emptyset$ for some 
$$ Q' \in \mathcal{B}^0 \cup \bigcup_{i,j,k} Q(\cS_{i,j,k}).$$
of similar scale. (Note that the doubling of $E$ means there is a control over how many times each $Q'$ is used in this way).  We then  define stopping time regions $\cS$ in $\mathcal{C}_\Gamma$ as the connected components (in the tree structure of $\mathcal{C}_\Gamma$) of $\mathcal{C}_\Gamma \setminus \mathcal{B}$. 

Finally, we check that in each $\cS$ obtained in this way we have the desired properties for Proposition \ref{prop:str}. Properties (2), (3), (5), and (6) follow from the associated properties for the stopping time regions constructed in $\Delta_\Gamma$ and the doubling property of $E$. The only thing that really needs verification is property (4), namely that if $R\in \cS$, then $\angle(\tau_\Gamma(Q(\cS)), \tau_\Gamma(R)) < \alpha$. This follows from the choice of $\alpha_0\ll \alpha$ and the $\alpha/10$-approximation of Lemma \ref{l:first-stop-construction}. This completes the proof of Proposition \ref{prop:str}.

\section{A sketch of the coarse argument in a special case}\label{sec:d=2}
The proof of Theorem \ref{thm:main} will occupy Sections \ref{sec:hilbertproof} and \ref{sec:finalproofs}. In this section, we will give a sketch of the proof of Theorem \ref{thm:main} in the special case $\bX=\mathbb{R}^2$ and $d=2$. This exhibits some of the main ideas of the general argument, while allowing us to take some shortcuts and elide some of the more technical parts of the full proof.

To be clear, nothing in this section is logically necessary to the proof of Theorem \ref{thm:main}, and all the details of the argument in all dimensions are given in Sections  \ref{sec:hilbertproof} and \ref{sec:finalproofs}. Our goal in this section is just to provide some scaffolding to help understand the detailed arguments below. In that spirit, we ignore some technical issues and are cavalier about constants in this section.

Thus, we now suppose that $\bX=\mathbb{R}^2$ and that $E\subseteq\mathbb{R}^2$ badly fits $2$-planes. (This is easily seen to be equivalent to $E$ being a \emph{porous} set: each ball of $\mathbb{R}^2$ contains a ball of comparable size that is disjoint from $E$.) We also fix a multiresolution family $\cF$ for $E$ and constants $A\geq 1$, $\epsilon>0$. According to Theorem \ref{thm:main}, $\cF$ should admit a coarse $\epsilon$-tangent field with inflation $A$, and this is what we will sketch in this section.

A first observation is that, instead of working directly with $\cF$, we can instead put a dyadic cube structure $\Delta=\Delta_E$ on $E$ using Proposition \ref{prop:christ}. We will then assign a line $\tau(Q)$ to each $Q\in \Delta$ in such a way that
\begin{equation}\label{eq:planesum}
\sum_{\substack{Q\in\Delta\\ \Gamma \cap B(Q)\neq\emptyset\\ \diam(B(Q))\leq \diam(\Gamma)\\ \theta^\tau_\Gamma(AB(Q)) \geq \epsilon}} \diam(B(Q)) \lesssim \ell(\Gamma)
\end{equation}
for every rectifiable curve $\Gamma$ in $\mathbb{R}^2$. Theorem \ref{thm:main} then easily follows using the doubling property of $E$.

\subsection{Construction of the coarse \texorpdfstring{$\epsilon$}{eps}-tangent line field in the plane}
Let us now describe how to construct the assignment of lines $\tau\colon \Delta \rightarrow \mathcal{L}_1$.
 
We will fix small parameters $\epsilon_2\ll \epsilon_1\ll \epsilon$ and a large parameter $\Lambda \approx 1/\epsilon_1$. We also write $\ball(Q)$ for $AB(Q)$.

If $Q\in \Delta$ and $L$ is an affine line in $\mathbb{R}^2$, we say that $L$ is \textit{good for $Q$} if
\begin{itemize}
\item $ L \cap 3\ball(Q) \neq \emptyset,$ and
\item $ L \cap 5\ball(Q) \subseteq N_{\epsilon_2\diam(\ball(Q))}(E).$
\end{itemize}
That is, a line is good for $Q$ if it passes near $Q$ and remains quite close to $E$ while it is near $Q$. 

For each affine line $L\subset \mathbb{R}^2$, let $N(L,Q)$ be the smallest non-negative integer $n$ such that $L$ is \emph{not} good for $Q^{\uparrow n}$. For each line $\tau\in \cL_1$, let $N(\tau,Q)$ be the maximum of $N(L,Q)$ over all lines $L$ parallel to $\tau$.

For each $Q\in\Delta$, we now set $\tau(Q)$ to be a choice of $\tau\in\mathcal{L}_1$ that maximizes $N(\tau,Q)$.

\subsection{Why this line field works in the plane}
We now consider a rectifiable curve $\Gamma$ in the plane and explain why \eqref{eq:planesum} holds. We choose small parameters $\alpha,\delta,d_0$ and $\lambda = A\Lambda$ and apply Proposition \ref{prop:str}; in particular, we are free to choose the first three parameters extremely small compared to $\epsilon,\epsilon_1,\epsilon_2$.

We are now granted a decomposition of the set
$$ \mathcal{C}_\Gamma = \{Q\in\Delta: B(Q)\cap \Gamma\neq\emptyset, \diam(B(Q))\leq\diam(\Gamma)\}$$
as
\begin{equation}\label{eq:planestr}
\mathcal{C}_\Gamma = \mathcal{B} \cup \bigcup_{i} \mathcal{S}_i
\end{equation}
with the properties of Proposition \ref{prop:str}. Each stopping time region $\cS_i$ has a top cube $Q(\cS_i)$.

Fix one of these $\cS_i$ and focus on it. Let $\tau_\Gamma$ be the direction of a line achieving the infimum in $\theta_\Gamma(\Lambda \ball(Q(\cS_i)))$, i.e., the line of best fit for $\Gamma$ in this ball. Note that by Proposition \ref{prop:str}, $\Gamma$ is well-approximated by lines in each cube of $\cS_i$ \emph{and} these lines turn very little as $Q$ varies in $\cS_i$. Thus, $\tau_\Gamma$ fits $\Gamma$ very closely not only in $\Lambda \ball(Q(\cS_i))$ but also in $\Lambda \ball(Q)$ for every $Q\in \cS_i$.

Now fix a very dense finite collection of lines $\cL^\epsilon \subset \cL_1$. Let $\tau^\epsilon(Q)$ be the closest element of $\cL^\epsilon$ to $\tau(Q)$. For each $\tau\in\cL^\epsilon$, we set
$$ \mathcal{T}_{i,\tau} = \{Q\in \cS_i : \angle(\tau^\epsilon(Q), \tau_\Gamma) \geq \epsilon_1\}.$$
These are cubes for which we have ``guessed poorly'': our assignment of $\tau(Q)$ (and hence $\tau^\epsilon(Q))$ is quite far from the direction the curve actually travels near $Q$.

Suppose now that $Q$ is a cube in $\mathcal{T}_{i,\tau}$. Since $E$ is porous, $B(Q)$ contains a ball $B$ of comparable diameter that is far from $E$, e.g., $10B \subseteq E^c$. The ``hole'' $B$ casts a ``shadow'' in the $\tau^\epsilon(Q)$ direction, i.e.,
$$ \{x : (x+\tau^\epsilon(Q)) \cap B \neq \emptyset\}$$
Recall that $\Gamma$ is very close to a line in the $\tau_\Gamma$ direction, and $\angle(\tau^\epsilon(Q), \tau_\Gamma) \geq \epsilon_1$. Because of the angle lower bound between $\tau^\epsilon(Q)$ and $\tau_\Gamma$, we may ensure that this shadow of $B$ contains a cube $H\in\mathcal{C}_\Gamma$ that is \emph{near} $B(Q)$ and has \emph{size comparable to} that of $B(Q)$. See Figure \ref{fig:hole} for a sketch of this construction.

\tikzset{every picture/.style={line width=0.75pt}} %set default line width to 0.75pt        
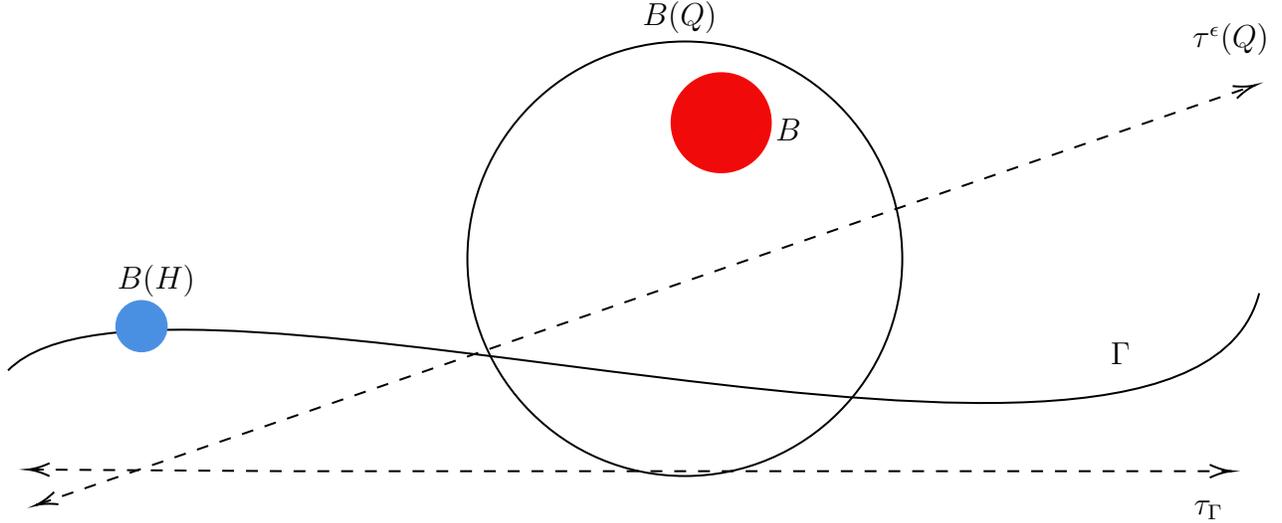
\begin{figure} \label{fig:hole}
\begin{tikzpicture}[x=0.75pt,y=0.75pt,yscale=-1,xscale=1]

%Curve Lines [id:da4892301318948873] 
\draw    (13.33,208.05) .. controls (89.33,130.05) and (608.33,309.05) .. (644.33,169.05) ;
%Shape: Circle [id:dp38455898587231463] 
\draw   (245,151.67) .. controls (245,91.1) and (294.1,42) .. (354.67,42) .. controls (415.23,42) and (464.33,91.1) .. (464.33,151.67) .. controls (464.33,212.23) and (415.23,261.33) .. (354.67,261.33) .. controls (294.1,261.33) and (245,212.23) .. (245,151.67) -- cycle ;
%Straight Lines [id:da5640625946671005] 
\draw  [dash pattern={on 4.5pt off 4.5pt}]  (25.33,258.06) -- (154.33,258.88) -- (628.33,258.88) ;
\draw [shift={(630.33,258.88)}, rotate = 180] [color={rgb, 255:red, 0; green, 0; blue, 0 }  ][line width=0.75]    (10.93,-3.29) .. controls (6.95,-1.4) and (3.31,-0.3) .. (0,0) .. controls (3.31,0.3) and (6.95,1.4) .. (10.93,3.29)   ;
\draw [shift={(23.33,258.05)}, rotate = 0.36] [color={rgb, 255:red, 0; green, 0; blue, 0 }  ][line width=0.75]    (10.93,-3.29) .. controls (6.95,-1.4) and (3.31,-0.3) .. (0,0) .. controls (3.31,0.3) and (6.95,1.4) .. (10.93,3.29)   ;
%Shape: Circle [id:dp5838823124545547] 
\draw  [color={rgb, 255:red, 244; green, 7; blue, 7 }  ,draw opacity=1 ][fill={rgb, 255:red, 240; green, 10; blue, 10 }  ,fill opacity=1 ] (348,83) .. controls (348,69.19) and (359.19,58) .. (373,58) .. controls (386.81,58) and (398,69.19) .. (398,83) .. controls (398,96.81) and (386.81,108) .. (373,108) .. controls (359.19,108) and (348,96.81) .. (348,83) -- cycle ;
%Shape: Circle [id:dp7935799241640878] 
\draw  [color={rgb, 255:red, 74; green, 144; blue, 226 }  ,draw opacity=1 ][fill={rgb, 255:red, 74; green, 144; blue, 226 }  ,fill opacity=1 ] (67.95,185.69) .. controls (67.95,178.68) and (73.63,173) .. (80.64,173) .. controls (87.65,173) and (93.33,178.68) .. (93.33,185.69) .. controls (93.33,192.7) and (87.65,198.38) .. (80.64,198.38) .. controls (73.63,198.38) and (67.95,192.7) .. (67.95,185.69) -- cycle ;
%Straight Lines [id:da7853691842029573] 
\draw  [dash pattern={on 4.5pt off 4.5pt}]  (29.22,275.4) -- (640.44,64.7) ;
\draw [shift={(642.33,64.05)}, rotate = 160.98] [color={rgb, 255:red, 0; green, 0; blue, 0 }  ][line width=0.75]    (10.93,-3.29) .. controls (6.95,-1.4) and (3.31,-0.3) .. (0,0) .. controls (3.31,0.3) and (6.95,1.4) .. (10.93,3.29)   ;
\draw [shift={(27.33,276.05)}, rotate = 340.98] [color={rgb, 255:red, 0; green, 0; blue, 0 }  ][line width=0.75]    (10.93,-3.29) .. controls (6.95,-1.4) and (3.31,-0.3) .. (0,0) .. controls (3.31,0.3) and (6.95,1.4) .. (10.93,3.29)   ;

% Text Node
\draw (332,20.4) node [anchor=north west][inner sep=0.75pt]    {$B( Q)$};
% Text Node
\draw (610,272.4) node [anchor=north west][inner sep=0.75pt]    {$\tau _{\Gamma }$};
% Text Node
\draw (609.33,31.4) node [anchor=north west][inner sep=0.75pt]    {$\tau ^{\epsilon }( Q)$};
% Text Node
\draw (399,79.4) node [anchor=north west][inner sep=0.75pt]    {$B$};
% Text Node
\draw (67,153.4) node [anchor=north west][inner sep=0.75pt]    {$B( H)$};
% Text Node
\draw (568,192.4) node [anchor=north west][inner sep=0.75pt]    {$\Gamma $};

\end{tikzpicture}
\caption{A sketch of the construction of $H$ from $Q$. Lines through $B(H)$ parallel to $\tau^\epsilon(Q)$ intersect the ``hole'' $B$ and thus pass far from $E$.}
\end{figure}

Collecting all these shadowed cubes $H$ for all $Q\in\mathcal{T}_{i,\tau}$, we may then use a covering argument to achieve the following. Each $Q\in\mathcal{T}_{i,\tau}$ is associated to a cube $H(Q) \in \mathcal{C}_\Gamma$ with the following properties:
\begin{itemize}
\item $B(H(Q))$ is well inside in the ``shadow'' of some ``hole'' in $\ball(Q(\cS_i))$ (i.e., a large expansion of $B(H(Q))$ is in this shadow),
\item $\diam(B(H(Q)) \gtrsim \diam(B(Q))$,
\item $\dist(B(H(Q)), B(Q)) \lesssim \diam(B(Q))$, and
\item the balls $B(H(Q))$ are disjoint.
\end{itemize}
In particular, the last of these conditions and the flatness of $\Gamma$ implies that
\begin{equation}\label{eq:planeHsum}
\sum_{Q\in\mathcal{T}_{i,\tau}} \diam(B(H(Q))\lesssim \diam(Q(\cS_i))
\end{equation}

We now wish to control the mapping $Q\mapsto H(Q)$. We first argue that each cube $Q\in\mathcal{T}_{i,\tau}$ cannot be contained inside its associated $B(H(Q))$, or even inside a large expansion of $B(H(Q))$. For convenience, suppose $Q\in\Delta_n$ and $Q(\cS_i)\in \Delta_0$. Suppose that $Q$ itself were contained inside a large expansion of $B(H(Q))$. 
Then $5\ball(Q)$ would be contained in the shadow of one of the holes $B\subseteq E^c$ contained above. It follows that the line $\tau^\epsilon(Q)$, along with any line sufficiently close to it, would fail to be good for $Q^{\uparrow k}$ with $k\leq n$. Thus, in the notation of the previous subsection, $N(\tau(Q),Q)\leq n$. On the other hand, the curve $\Gamma$ is both very close to (a translation of) the line $\tau_\Gamma$ and very close to the set $E$ in each $Q\in\cS_i$. From this, it follows that $N(\tau_\Gamma, Q) > n$. Since $\tau(Q)$, by our construction, was supposed to maximize $N(\tau,Q)$ over all choices of $\tau$, this yields a contradiction. Therefore, $Q$ is \textbf{not} contained inside a large expansion of $B(H(Q))$.

From this and the properties of $H(Q)$, it follows that $B(Q)$ and $B(H(Q))$ are comparable in size and near each other. The doubling property of $E$ then implies that 
\begin{equation}\label{e:bdd-to-one}
  \textrm{the assignment }  Q\mapsto H(Q) \textrm{ is bounded-to-one.}
\end{equation}
From \eqref{e:bdd-to-one}  and \eqref{eq:planeHsum}, we see that
\begin{equation}\label{eq:planebadsum}
\sum_{Q\in \mathcal{T}_{i,\tau}} \diam(B(Q)) \lesssim \diam(Q(\cS_i)).
\end{equation}

The argument for \eqref{eq:planesum} is now almost complete. Suppose now that $Q$ is one of the cubes summed over in \eqref{eq:planesum}. One possibility is that $Q\in \mathcal{B}$ where $\mathcal{B}$ is as in \eqref{eq:planesum}. The total diameter of all such cubes is immediately controlled by Proposition \ref{prop:str}(5).

Otherwise, $Q\in \cS_i$ for some $i$. In that case, $\Gamma$ lies very close to the line $\tau_\Gamma$ in $Q$. If $\tau(Q)$ were close to $\tau_\Gamma$ (say, within angle $2\epsilon_1 \ll \epsilon$), then we would have ``guessed correctly'' and $\theta^\tau_\Gamma(\ball(Q))$ would be less than $\epsilon$. Hence, in this case, $Q$ would not appear in the sum \eqref{eq:planesum} at all.

If in fact $\tau(Q)$ were \emph{not} close to $\tau_\Gamma$, then $Q$ would be in $\mathcal{T}_{i,\tau}$ for one of (a controlled number of!) possible choices of $\tau\in\mathcal{L}_\epsilon$. In that case, the total diameter of all these cubes in a single $\cS_i$ is controlled by $\diam(Q(\cS_i))$, by \eqref{eq:planebadsum}, and the total diameter of all the $B(Q(\cS_i)))$ is controlled by $\ell(\Gamma)$ by Proposition \ref{prop:str}(6).

Thus, \eqref{eq:planesum} holds and we have completed our sketch of the proof of Theorem \ref{thm:main} in the special case $\bX=\mathbb{R}^2$ and $d=2$.

\section{Coarse tangent fields depending on \texorpdfstring{$\epsilon$}{eps} in Hilbert space}\label{sec:hilbertproof}
In this section, we do most of the work to prove Theorem \ref{thm:main} for subsets of Hilbert space and arbitrary $d\geq 1$; the final details will be handled in Section \ref{sec:finalproofs}. The argument follows the general framework of the planar argument sketched in the previous section, but the added generality requires a number of changes (and of course, more detailed arguments). There are two major differences from the planar case:
\begin{enumerate}[(i)]
\item When $d>2$, we need a more complicated inductive construction of the coarse tangent field $\tau$ and a corresponding ``multi-stage'' version of the planar argument. In particular, the construction of $\tau$, while it does extend the construction in the planar case, is perhaps not the most obvious extension of it to higher dimensions.
\item The fact that the ambient space $\bX$ may not be finite-dimensional makes some arguments more difficult. For instance, we cannot na\"ively discretize the collection of all possible lines, as we did in the planar case.
\end{enumerate}

To begin, we fix a (finite or infinite dimensional) Hilbert space $\bX$. If $E$ is a doubling subset of $\bX$, we can apply apply Proposition \ref{prop:christ} to construct a family $\Delta=\bigcup_n \Delta_n$ of ``dyadic cubes'' for $E$. Recall that these cubes have a scaling factor $s\in (0,1)$ (which is an absolute constant), and each $Q\in\Delta_n$ comes equipped with a ball
\[
B(Q) = B(z(Q), C_1 s^n) \subset \bX
\]
containing $Q$, and that
$$ B(z(Q), a_0 s^n) \cap E \subseteq Q \subseteq B(Q).$$

In this section, will show the following proposition, which (modulo some small arguments in Section \ref{sec:finalproofs}) essentially proves Theorem \ref{thm:main}.

\begin{proposition}\label{prop:epstangent}
Let $E$ be a doubling subset of $\bX$ that badly fits $d$-planes (with parameter $\epsilon_0$). Fix $A\geq 1$.

Let $\hDelta$ be a system of dyadic cubes for $E$, as above, and let $\epsilon>0$. Then there is a $(d-1)$-dimensional coarse plane field $\tau=\tau_\epsilon$ defined on the collection
$$ \{AB(Q) : Q\in\hDelta\}$$
such that
\begin{equation}\label{eq:epstangentprop}
 \sum_{\substack{Q\in\hDelta \\ \Gamma \cap B(Q)\neq \emptyset\\ \beta^{\tau}_{\Gamma}(AB(Q))\geq \epsilon}} \diam(B(Q)) \lesssim \ell(\Gamma)
\end{equation}
for every rectifiable curve $\Gamma\subseteq \bX$.

The implied constant depends on $A, \epsilon, \epsilon_0, d$ and the doubling constant of $E$.
\end{proposition}

Some ideas in the proof of Proposition \ref{prop:epstangent} are similar to the case in which $d=2$ and $\bX=\mathbb{R}^2$, handled in the previous section, but there are significant additional complexities in the details. Where arguments are essentially the same as in the planar case, we reference the previous section.

\subsection{Construction of the coarse plane field}\label{subsec:construction}
Let $\bX$ be a (finite or infinite dimensional) Hilbert space. Also fix $A\geq 1$, $d\in\mathbb{N}$, $\epsilon_0\in (0,1)$, $\epsilon>0$ and a set $E\subseteq \bX$ that badly fits $d$-planes (with parameter $\epsilon_0>0$). Let $\hDelta$ be a family of dyadic cubes on $E$ as above. For $Q\in\Delta$, write $\ball(Q) = AB(Q)$.

Our goal in this subsection is to construct a coarse $(d-1)$-dimensional plane field $\tau = \tau_\epsilon$ that assigns to each ball $\ball(Q)$ with $Q\in\hDelta$ a $(d-1)$-dimensional plane in $\bX$ and satisfies the conclusion of Proposition \ref{prop:epstangent}. We may assume without loss of generality that $\epsilon \leq \epsilon_0$. To compress the notation, we will write $\tau(Q)$ for $\tau(\ball(Q))$, i.e., we will consider $\tau$ as being defined on cubes rather than on their associated balls. The construction of $\tau$ will depend on both $\epsilon$ and the set $E$.

As in the planar construction, we need some parameters. Let $\epsilon_1 = \frac{\epsilon}{2A}$ and $\Lambda = 100(\epsilon_1)^{-1}A$. Next, for $k\in\{1,\dots, d-1\}$, we choose positive constants $\epsilon^{(k)}_2$ so that
$$ \epsilon^{(1)}_2 \ll \Lambda^{-2} \epsilon^{(2)}_2 \leq  \epsilon^{(2)}_2 \ll \Lambda^{-2} \epsilon^{(3)}_2 \leq \epsilon^{(3)}_2 \ll \dots \ll \Lambda^{-2} \epsilon_2^{(d)} \leq \epsilon_2^{(d)}  \ll \epsilon_0 < 1.$$
The implied constants here depend on $d$ and the doubling constant of $E$, and will be determined in the course of the argument.

Before defining $\tau$, we make some preliminary definitions analogous to those in the planar case.
\begin{definition}
Let $Q\in \Delta$ and $V$ an affine subspace of dimension $k$ in $\bX$ (i.e., not necessarily through the origin). The subspace $V$ is called \textit{good for $Q$} if
\begin{itemize}
\item $ V \cap 3\ball(Q) \neq \emptyset,$ and
\item $ V \cap 5\ball(Q) \subseteq N_{\epsilon^{(k)}_2 \diam(\ball(Q))}(E).$
\end{itemize}
\end{definition}

For each $Q\in\hDelta$ and subspace $V$, let $N(V,Q)$ be the smallest non-negative integer $n$ such that $V$ is \emph{not} good for $Q^{\uparrow n}$. For each $\tau\in \cL_k$, let $N(\tau,Q)$ be the maximum of $N(V,Q)$ over all $k$-dimensional affine subspaces $V$ parallel to $\tau$.

For each $Q\in\hDelta$, let $\tau_1(Q)$ be a choice of $\tau\in\cL_1$ that maximizes $N(\tau,Q)$ among $\tau \in \cL_1$. We then recursively define $\tau_{k+1}(Q)$ to be a choice of $\tau\in \cL_{k+1}$ that maximizes $N(\tau,Q)$ among those $\tau\in  \cL_{k+1}$ with the property that $\tau_k\subset \tau$. In other words,
$$ N(\tau_{k+1}(Q), Q) = \max\{ N(\tau, Q) : \tau\in \cL_{k+1}, \tau \supseteq \tau_k\}.$$

Finally, we set $\tau(Q)=\tau_{d-1}(Q)$. This completes the definition of the coarse plane field $\tau\colon \hDelta \rightarrow \cL_{d-1}$.

\begin{remark}
This construction assigns to each $Q$ a $(d-1)$-dimensional subspace, while Definition \ref{def:coarsetangent} permits us to be less ``wasteful'' and assign lower-dimensional subspaces if possible. Of course, one can always increase the dimension of a subspace of lower dimension, but there may be times where it is possible to have, e.g., a $2$-dimensional coarse tangent field that assigns only lines on some balls. In general, we our argument will bound the dimension of each $\tau(Q)$ in Proposition \ref{prop:epstangent} by a parameter $\hat{k}(Q)\leq d-1$ that is defined in \eqref{eq:khat} below. 
\end{remark}

\subsection{Proof of Proposition \ref{prop:epstangent}}
\begin{proof}[Proof of Proposition \ref{prop:epstangent} for $d>2$]
We have now fixed parameters $A\geq 1$, $d\in\mathbb{N}$, a doubling set $E\subseteq \bX$ that badly fits $d$-planes (with parameter $\epsilon_0>0$), and $\epsilon>0$. We have applied Proposition \ref{prop:christ} to obtain a dyadic decomposition $\hDelta_E$ of $E$. With parameters $\epsilon_1$, $\Lambda$, and $\epsilon^{(k)}_2$ as in the previous subsection, we have defined a coarse $(d-1)$-dimensional plane field $\tau\colon \hDelta_E \rightarrow \cL_{d-1}$ as above. We now prove that $\tau$ satisfies the conclusion of Proposition \ref{prop:epstangent}.

Consider an arbitrary connected set $\Gamma$ in $Q_0$. Choose parameters $\alpha, \delta, d_0>0$ sufficiently small; it will suffice to make sure that
$$ \Lambda(\alpha + \delta + d_0) \ll \epsilon^{(1)}_2,$$
where the implied constant will depend on $A$, $d$ and the doubling constant of $E$, and will be determined in the course of the argument. We emphasize that no properties of $\Gamma$ come into the choices of these constants.

Let
$$ \cC_\Gamma = \{Q\in\hDelta_E : B(Q) \cap \Gamma \neq \emptyset, \diam(B(Q))\leq \diam(\Gamma)\}.$$

In proving proposition \ref{prop:epstangent}, we need only concern ourselves with  $Q\in\cC_\Gamma$. Indeed, it follows from the doubling property of $E$ and a simple geometric series argument (using the trivial bound $\beta^\tau_\Gamma(AB) \leq \diam(\Gamma)/\diam(AB)$) that
\begin{equation}\label{eq:bigcubes}
\sum_{\substack{Q\in\Delta_E\\ B(Q) \cap \Gamma \neq \emptyset\\ \diam(B(Q))> \diam(\Gamma)\\ \beta^\tau_\Gamma(AB(Q))\geq \epsilon}} \diam(B(Q)) \lesssim \diam(\Gamma) \lesssim \ell(\Gamma).
\end{equation}

Apply Proposition \ref{prop:str} with these parameters $\alpha,\delta, d_0>0$ and $\lambda = A\Lambda$. This partitions $\cC_\Gamma$ as
$$ \cC_\Gamma = \mathcal{B} \cup \bigcup_i \mathcal{S}_i$$
with the properties of Proposition \ref{prop:str}. To prove Proposition  \ref{prop:epstangent}, it suffices to show that
\begin{equation}\label{eq:epstangenttext}
\sum_{\substack{Q\in\cC_\Gamma\\\beta^\tau_\Gamma(\ball(Q))\geq \epsilon}} \diam(B(Q)) \lesssim \ell(\Gamma).
\end{equation}

If $Q\in \cC_\Gamma$, then $\lambda B(Q) = \Lambda \ball(Q)$. We write
$$ \tau_\Gamma(\Lambda \ball(Q))$$
to denote a line that achieves the infimum in $\theta_\Gamma(\Lambda\ball(Q))$ up to a factor of $2$.

Given a stopping time region $\cS_i$ in our decomposition of $\cC_\Gamma$, let
\[
 \mathcal{T}_i = \{ Q\in \cS_i: \angle(\tau(Q), \tau_\Gamma(\Lambda\ball(Q(\cS_i)))) > \epsilon_1\}.
 \]

If we can prove, for each stopping time region $\cS_i$, that
$$ \sum_{Q\in\mathcal{T}_i} \diam(\ball(Q)) \lesssim \ell(\Gamma),$$
then the proof of Proposition \ref{prop:epstangent} will follow in a way similar to the conclusion of Section \ref{sec:d=2}. Thus, for now we fix a single stopping time region $\cS_i$. We assume for convenience that $Q(\cS_i) \in \hDelta_0$, and given $Q\in\cS_i$ we write $n(Q)$ for the choice of $n$ for which $Q\in\hDelta_n$. We also fix
$$ \tau_\Gamma = \tau_\Gamma(\Lambda\ball(Q(\cS_i))).$$

Our first important observation is that $\tau_\Gamma$ is good for $Q$ and all its ancestors in $\cS_i$:
\begin{claim}\label{claim:gammagood}
If $Q\in\cS_i$, then $N(\tau_\Gamma, Q)\geq n(Q)+1$.
\end{claim}
\begin{proof}
It suffices to show the following: if $R\in \cS_i$ and $L$ is a line in direction $\tau_\Gamma$ passing through a point of $\Gamma \cap B(R)$, then $L \cap 5\ball(R) \subseteq N_{\epsilon_2\diam(\ball(R))}(E).$ (Then one can apply this to each ancestor of $Q$ in $\cS_i$.)

Fix such an $L$. 
 From Proposition \ref{prop:str}, we have that $d_{\Gamma, E}(\Lambda \ball(R))<d_0$, that $\theta_\Gamma(\Lambda \ball(R))<\delta$, and that $\tau_\Gamma$ has angle at most $\alpha$ from a line minimizing $\theta_\Gamma(\Lambda \ball(R))$. It follows that
$$L\cap 5\ball(R)\subseteq N_{\Lambda(\delta + d_0+\alpha)\diam(\ball(R))}(E).$$
Since $\Lambda(\delta+d_0+\alpha)\ll \epsilon^{(1)}_2$, this suffices. 
\end{proof}

As we did in Section \ref{sec:d=2} for lines, we would like to now discretize the possible choices of $k$-planes $\tau_k(Q)$. This is more complicated than in Section \ref{sec:d=2} for two reasons. For one, we may be in an infinite-dimensional setting and so we cannot simply discretize the collection of all possible $k$-planes and maintain control. For another, we would like to preserve the ``nested'' character $\tau_{k}\subset \tau_{k+1}$ in the discretization. Given a cube $Q\in\cS_i$, we discretize the $k$-planes ``visible'' at the location and scale of $Q$ as follows:

Fix $R\in \cS_i$, $k\in\{1, 2, \dots, d-1\}$, and $\sigma_{k-1}\in\mathcal{L}_{k-1}$.  
Consider the collection $\Sigma_k(R, \sigma_{k-1})$ of all $\sigma_k\in\mathcal{L}_k$ containing $\sigma_{k-1}$ such that some affine $k$-plane making angle at most $\epsilon^{(2)}_k$ from $\sigma_k$ is good for $R$.

Let $\mathcal{L}_k(R, \sigma_{k-1})$ be a maximal subcollection in $\Sigma_k(R,\sigma_{k-1})$ with the property that if $\sigma, \sigma'\in\mathcal{L}_k(R,\sigma_{k-1})$, then $D(\sigma,\sigma')\geq C_d \epsilon^{(2)}_k$. Here $C_d$ is a fixed constant depending only on $d$. If $k=1$, then the only the choice of $\sigma_{k-1}\in\mathcal{L}_{k-1}$ is $\{0\}$, and we may write $\cL_1(R)$ instead of $\cL_1(R,\{0\})$.

\begin{claim}\label{claim:numplanes}
For $C_d$ sufficiently large, the following holds: If $R\in\cS_i$, $k\in\{1, 2, \dots, d-1\}$, and $\sigma_{k-1}\in\mathcal{L}_{k-1}$, then the cardinality of $\mathcal{L}_k(R, \sigma_{k-1})$ is bounded by a constant depending only on $\epsilon$ and the doubling constant of $E$.
\end{claim}
\begin{proof}

Let $N$ be an  $\epsilon^{(k)}_2\diam(\ball(R))$-net in $E\cap 20\ball(R)$. Note that the cardinality $|N|$ of $N$ is controlled by constants depending ultimately only on $\epsilon$ and the doubling constant of $E$. We will control the cardinality of $\mathcal{L}_k(R, \sigma_{k-1})$ using $|N|$.

Suppose $\sigma\in\mathcal{L}_k(R, \sigma_{k-1})$. Then there is an affine $k$-plane $V_{\sigma}$, with angle at most $\epsilon^{(k)}_2$ from $\sigma$ such that $V\cap 3\ball(R) \neq\emptyset$ and $V\cap 5\ball(R)$ is in the $\epsilon^{(k)}_2\diam(\ball(R))$-neighborhood of $E$.

Choose an ordered list of $k+1$ points $\{x_0, \dots, x_k\}\subset V\cap 5\ball(R')$ that are mutually separated by distance $\geq \diam(\ball(R))$ and for which $\{x_i-x_0\}_{i=1}^k$ are orthogonal. Moving these points to their closest counterparts $\{y_i\}$ in $N$, we can therefore find an ordered list $S_{\sigma}=\{y_i\}$ of $k+1$ points in $N\cap 20\ball(R)$ that are $(1 - 4\epsilon^{(k)}_2)\diam(\ball(R))$-separated and within pairwise distances $\epsilon^{(k)}_2 \diam(\ball(R))$ of $\{x_i\}$. 

Let us now observe that if $\sigma, \sigma'\in\mathcal{L}_k(R,\sigma_{k-1})$ and $S_\sigma = S_{\sigma'}$, then $\sigma=\sigma'$. Indeed, in this case the planes $V_\sigma$ and $V_{\sigma'}$ each admit orthonormal bases with pairwise mutual angles from each other $\lesssim \epsilon^{(k)}_2$, which implies that every vector in $V_{\sigma}$ is within angle $\lesssim_d \epsilon^{(k)}_2$ from a vector in $V_{\sigma'}$, and vice versa. The same therefore holds for $\sigma$ and $\sigma'$, and the defining property of $\mathcal{L}_k(R,\sigma_{k-1})$ then forces $\sigma=\sigma'$.

It now follows that  $|\cL_k(R,\sigma_{k-1})| \leq |N|^{k+1}$, which as noted above is controlled as desired.

\end{proof}

Recall that each cube $Q\in\mathcal{T}_i$ was assigned an increasing list of subspaces $\tau_1(Q)\subset\dots\subset \tau_{d-1}(Q)$ in subsection \ref{subsec:construction}. We now discretize this list using our construction above.

\begin{claim}\label{claim:nk}
For each $Q\in \mathcal{T}_i$, there are integers
$$ n(Q)+1=n_0(Q)\geq n_1(Q) \geq n_2(Q) \geq \dots \geq n_{d-2}(Q) \geq n_{d-1}(Q)=0$$
and $k$-dimensional subspaces $\tau^\epsilon_k(Q)$ with
$$ \tau^\epsilon_0(Q) \subseteq \tau^\epsilon_1(Q) \subseteq \tau^\epsilon_2(Q)  \subseteq \dots \subseteq \tau^\epsilon_{d-1}(Q).  $$
We construct these objects to have the following properties for each $Q\in\mathcal{T}_i$ and $k\in\{1, \dots, d-1\}$:
\begin{enumerate}[(i)]
\item $n_0(Q) = \min\left\{ n(Q)+1, N(Q,\tau_\Gamma)\right\}$.
\item If $k>0$, $n_k(Q) = \min\left\{ n_{k-1}(Q), N(Q, \langle \tau^\epsilon_k(Q), \tau_\Gamma \rangle) \right\}$.
\item If $n_{k-1}(Q)>0$, then $\tau^\epsilon_k(Q)\in \mathcal{L}_k(Q^{\uparrow (n_{k-1}(Q)-1)}, \tau^\epsilon_{k-1}(Q))$. 
\item If $n_{k-1}(Q)>0$, then $D(\tau^\epsilon_k(Q), \tau_k(Q))\lesssim_{d} \epsilon^{(k)}_2$.
\end{enumerate}
\end{claim}
\begin{proof}

Given $Q\in\mathcal{T}_i$, we will define $n_k(Q)$ and $\tau^\epsilon_k(Q)$ inductively. First, we set $n_0(Q)=n(Q)+1$ and $\tau^\epsilon_0(Q)=\{0\}$. In this case, $N(Q,\tau_\Gamma)\geq n(Q)+1$ by Claim \ref{claim:gammagood}.

Items (ii) through (iv) are vacuous in the case $k=0$.

Suppose now that $n_0(Q)$ through $n_{k-1}(Q)$ as well as $\tau^\epsilon_0(Q)$ through $\tau^\epsilon_{k-1}(Q)$ have been defined and satisfy properties (i) through (iv). We next define $\tau^\epsilon_k(Q)$ and $n_k(Q)$. 

If $n_{k-1}(Q)=0$, we define $\tau^\epsilon_k(Q)$ to be an arbitrary $k$-plane containing $\tau^\epsilon_{k-1}(Q)$, and we set $n_k(Q)=0$. All four conclusions of the claim are verified easily in this case.  

Now suppose that $n_{k-1}(Q) >0$. Let $R=Q^{\uparrow (n_{k-1}(Q)-1)}$. Observe that 
$$ N(Q,\tau_k(Q)) \geq N(Q, \langle \tau^\epsilon_{k-1}, \tau_\Gamma \rangle) \geq n_{k-1}(Q) > 0.$$
Therefore $\tau_k(Q)$ is good for $R$. We know by induction that $D(\tau^\epsilon_{k-1}(Q), \tau_{k-1}(Q))\lesssim_{d} \epsilon^{(2)}_{k-1}$ and $\tau_{k-1}(Q) \subseteq \tau_k(Q)$. Rotate $\tau_k(Q)$ by angle $\lesssim_{d} \epsilon^{(2)}_{k-1} \leq \epsilon^{(k)}_2$ to obtain a $k$-plane $\sigma_k\in\Sigma_k(R,\tau^\epsilon_{k-1}(Q))$ that contains $\tau^\epsilon_{k-1}(Q)$. Choose $\tau^\epsilon_k(Q)\in\mathcal{L}_k(R,\tau^\epsilon_{k-1}(Q))$ to have angle $\lesssim_d \epsilon^{(k)}_2$ from $\sigma_k$, which is possible by the definition of $\mathcal{L}_k(R,\tau^\epsilon_{k-1}(Q))$. It follows that
$$ D(\tau^\epsilon_k(Q), \tau_k(Q))\lesssim_{d} \epsilon^{(k)}_2.$$ 
Thus, items (iii) and (iv) hold.

Once $\tau^\epsilon_k(Q)$ has been defined, we set
$$ n_k(Q) = \min\left\{ n_{k-1}(Q), N(Q, \langle \tau^\epsilon_k(Q), \tau_\Gamma \rangle) \right\},$$
and the inductive construction can continue through $k=d-1$. Conclusions (i) and (ii) immediately hold.

It remains only to verify that $n_{d-1}(Q)=0$, as claimed. Since $E$ badly fits $d$-planes (with parameter $\epsilon_0)$ and $\epsilon^{(d)}_{2}\ll \epsilon_0$, the $d$-plane $\langle \tau^\epsilon_{d-1}(Q), \tau_\Gamma \rangle$ cannot be good for $Q$. Indeed, if it were, then this plane would contain an isometric copy of a $d$-ball that lies very close to $E$ in $5\ball(Q)$. Hence $n_{d-1}(Q)=0$.

\end{proof}

Our next goal is to associate to each $R\in\cS_i$, $k\in\{1, \dots, d-1\}$, and $k$-plane $\tau_k$ a collection of cubes $\mathcal{H}_k(R,\tau_k)\subseteq \mathcal{C}_\Gamma$ lying in or near $R$. The claim is rather technical, but the rough idea is as follows. Given $R$, we consider cubes $Q$ for which $\tau^\epsilon_k(Q)=\tau_k$, $Q^{\uparrow n_k(Q)}\subset R$, and $R=Q^{n_{k-1}(Q)-1}$. For such cubes, the $(k+1)$-plane $\langle \tau_k, \tau_\Gamma \rangle$ stops being good at $Q^{\uparrow n_{k}(Q)}$, while the $k$-plane $\langle \tau^\epsilon_{k-1}(Q),\tau_\Gamma \rangle$ inside it continues to be good until $R$. We argue that this forces a large piece of $\Gamma$ near $Q^{\uparrow n_k(Q)}$ to be ``blocked'' in direction $\tau_k$: if $x$ is in this piece, then $x+\tau_k$ gets far from $E$ near $Q^{\uparrow n_k(Q)}$. These pieces of $\Gamma$, which are essentially cubes of $\cC_\Gamma$, are collected in $\mathcal{H}_k(R,\tau_k)$, and then some are removed so that they do not overlap too much.

The following more precise language will be helpful below.

\begin{definition}
Suppose $\sigma\in\mathcal{L}_k$, $\eta>0$, $B$ is a ball in $\bX$, and $x\in \mathbb{X}$. We will say that \emph{$x$ is $(\eta,\sigma)$-blocked in $B$} if $x\in 5B$ and $5B \cap (x+\sigma)$ contains a point at distance at least $\eta\diam(B)$ from $E$.
\end{definition}

\begin{claim}\label{claim:holes>2}
Let $R\in\cS_i$, $k\in\{1, \dots, d-1\}$, and $\tau_k\in \cL_k$. There is a collection of cubes $\mathcal{H}_k(R,\tau_k) \subseteq \mathcal{C}_\Gamma$ with the following properties:
\begin{enumerate}[(i)]
\item If $H\in\mathcal{H}_k(R,\tau_k)$, then $\ball(H)\subseteq \frac{\Lambda}{2} \ball(R')$ for some $R'\in\cS_i$ with $R'\subseteq R^{\uparrow}$, $\Lambda \ball(R') \subseteq 5\ball(R)$ and $\diam(\ball(R'))\lesssim_{d,\epsilon^{(k+1)}_2} \diam(\ball(H))$. 
\item If $H,H'\in\mathcal{H}_k(R,\tau_k)$ then $7\ball(H)\cap 7\ball(H')= \emptyset$.
\item If $H\in\mathcal{H}_k(R,\tau_k)$, then there exists a $Q\in\mathcal{T}_{i}$ with $R=Q^{\uparrow (n_{k-1}(Q)-1)}$, $\Lambda \ball(Q^{\uparrow n_k(Q)}) \subseteq 5\ball(R)$, and $\tau^\epsilon_k(Q)=\tau_k$ such that each $x\in 10\Lambda \ball(H)$ is $\left(\frac{1}{2\Lambda}\epsilon^{(k+1)}_2,\tau_k\right)$-blocked  
in $\frac{\Lambda}{10}\ball(Q^{\uparrow n_k(Q)})$.
\item If $Q\in\mathcal{T}_{i}$, 
 $n_k(Q)<n_{k-1}(Q)$, $\Lambda \ball(Q^{\uparrow n_k(Q)}) \subseteq 5\ball(R)$, $R=Q^{\uparrow (n_{k-1}(Q)-1)}$, and $\tau^\epsilon_k(Q)=\tau_k$, then there is an $H_k(Q,\tau_k)\in\mathcal{H}_k(R,\tau_k)$ such that
$$ 7\Lambda \ball(Q^{\uparrow n_k(Q)}) \cap 7\ball(H_k(Q,\tau_k)) \neq \emptyset.$$
and
$$ \diam(\ball(H_k(Q,\tau_k))) \gtrsim_{\epsilon^{(k)}_2, d} \diam \ball(Q^{\uparrow n_k(Q)}).$$
\end{enumerate}

\end{claim}
\begin{proof}
Consider the collection
\begin{equation}\label{eq:holesprelim}
\{Q^{\uparrow n_k(Q)} : Q\in\mathcal{T}_{i}, R=Q^{\uparrow (n_{k-1}(Q)-1)},  n_k(Q)<n_{k-1}(Q), \Lambda \ball(Q^{\uparrow n_k(Q)}) \subseteq 5\ball(R), \tau^\epsilon_k(Q)=\tau_k\},
\end{equation}
and note that all these cubes are in $R^{\uparrow}$. (This collection may be empty, in which case $\mathcal{H}^k(R,\tau_k)$ is empty as well, and the claim holds vacuously.)

Fix a $Q^{\uparrow n_k(Q)}$ as in \eqref{eq:holesprelim} and let $\sigma = \langle \tau_k, \tau_\Gamma \rangle$. By Claim \ref{claim:nk}, $n_k(Q)<n_{k-1}(Q)$ means that $n_k(Q)=N(Q,\langle \tau_k, \tau_\Gamma\rangle)$. Thus, each point in $3\ball(Q)$ is $(\epsilon^{(k+1)}_2,\sigma)$-blocked in $\ball(Q^{\uparrow n_k(Q)})$. We will now argue that there is a cube $H$ in $\mathcal{C}_\Gamma$ inside  $\frac{\Lambda}{10} \ball(Q^{\uparrow n_k(Q)})$ such that
\begin{itemize}
\item $\diam(\ball(H)) \gtrsim \diam(Q^{\uparrow n_k(Q)})$, and
\item each point $x\in 10 \Lambda \ball(H)$ is $\left(\frac{1}{2\Lambda}\epsilon^{(k+1)}_2,\tau_k\right)$-blocked in $\frac{\Lambda}{10}Q^{\uparrow n_k(Q)}$.
\end{itemize}
Let $n_k=n_k(Q)$ and fix a point $z\in 3Q\cap \Gamma$. By definition of $n_k$, the point $z$ must be $(\epsilon^{(k+1)}_2, \sigma)$-blocked in $Q^{\uparrow n_k}$. Therefore, there is a point $y\in (z+\sigma) \cap 5\ball(Q^{\uparrow n_k})$ with $\dist(y,E)\geq \epsilon^{(k+1)}_2\diam(\ball(Q^{\uparrow n_k}))$.

Within the affine $(k+1)$-plane $z+\sigma$, the affine $k$-plane $y+\tau_k$ and affine line $z+\tau_\Gamma$ have mutual angle at least $\epsilon_1$. They therefore intersect at a point $p\in (y+\tau_k) \cap (z+\tau_\Gamma) \subset (z+\sigma)$, and this point $p$ is in $\frac{\Lambda}{20} \ball(Q^{\uparrow n_k})$, since $\Lambda$ was chosen sufficiently large depending on $\epsilon_1$. 

Since $\theta_\Gamma(\Lambda \ball(Q^{\uparrow n_k}))<\delta$ and $\delta$ is small, there must be a point $q\in \Gamma \cap \frac{\Lambda}{10} \ball(Q^{\uparrow n_k})$ such that $p\in B(y,\frac14\epsilon^{(k+1)}_2\diam(\ball(Q^{\uparrow n_k})))$. Then, $q+\tau_k$ intersects $\ball(y,\frac14\epsilon^{(k+1)}_2\diam(\ball(Q^{\uparrow n_k})))$. We can therefore choose a cube $H\in\mathcal{C}_\Gamma$ containing $q$ and with $\diam(\ball(H))\gtrsim \diam(\ball(Q^{\uparrow n_k}))$ such that if $x\in 10\Lambda \tilde{B}(H)$, then $(x+\tau_k)\cap \frac{\Lambda}{10}\ball(Q^{\uparrow n_k})$ contains a point of $B_y:=\ball(y,\frac12\epsilon^{(k+1)}_2\diam(\ball(Q^{\uparrow n_k})))$, a point at least $\frac12\epsilon^{(k+1)}_2\diam(\ball(Q^{\uparrow n_k}))$ from $E$. Such a point is therefore at distance at least
$$ \frac{1}{2\Lambda}\epsilon^{(k+1)}_2\diam(\Lambda \ball(Q^{\uparrow n_k})) $$
from $E$. Thus, each point of $10\Lambda \tilde{B}(H)$ is 
$\left(\frac{1}{2\Lambda}\epsilon^{(k+1)}_2,\tau_k\right)$-blocked in $\frac{\Lambda}{10}\ball(Q^{\uparrow n_k(Q)})$.

Let $\mathcal{H}^0_k(R,\tau_k)$ collect all these cubes $H$, for every choice of $Q^{\uparrow n_k(Q)}$ as in \eqref{eq:holesprelim}. We will define $\mathcal{H}_k(R,\tau_k)$ as a subset of $\mathcal{H}^0_k(R,\tau_k)$. It follows from the discussion above that items (i) (with $R'=Q^{\uparrow n_k(Q)}$ in the notation above) and (iii) will hold for every $H\in\mathcal{H}_k(R,\tau_k)$ .

We now build $\mathcal{H}_k(R,\tau_k)$ inductively one element at a time as follows: In the first step, add the largest cube of $\mathcal{H}^0_k(R,\tau_k)$ to $\mathcal{H}_k(R,\tau_k)$. In each following stage, we add to $\mathcal{H}^k(R)$ the cube $H\in \mathcal{H}^0_k(R,\tau_k)$ with largest $\diam(B(H))$ satisfying the property that $7\ball(H)\cap 7\ball(H')=\emptyset$ for all cubes $H'$ already in $\mathcal{H}_k(R,\tau_k)$. 

It follows immediately that (ii) holds for $\mathcal{H}_k(R,\tau_k)$. Finally, for item (iv), suppose $Q\in\mathcal{T}_{i,\bm \tau}\cap \Delta_n$, $n_k(Q)<n_{k-1}(Q)$, $\Lambda \ball(Q^{\uparrow n_k(Q)}) \subseteq 5\ball(R)$, and $R=Q^{\uparrow (n_{k-1}(Q)-1)}$. Then $Q^{\uparrow n_k(Q)}$ is in the collection \eqref{eq:holesprelim}. Let $H$ be the element of $\mathcal{H}^0_k(R,\tau_k)$ associated to it above.

If $H$ was selected for $\mathcal{H}_k(R, \tau_k)$, then $H$ itself can serve as $H_k(Q,\tau_k)$. If not, then $7\ball(H)\subseteq 7\Lambda Q^{\uparrow n_k(Q)}$ must intersect $7\ball(H')$ for some $H'\in\mathcal{H}_k(R,\tau_k)$ with $\diam(\ball(H'))\geq\diam(\ball(H))$, and this $H'$ can serve as $H_k(Q,\tau_k)$.
 
\end{proof}

A consequence of this construction is the following: 

\begin{claim}\label{claim:Hsum}
For each $R\in\cS_i$, $k\in\{1, \dots, d-1\}$ and $\tau_k\in\mathcal{L}_k$, we have
$$ \sum_{H\in\mathcal{H}_k(R,\tau_k)} \diam(B(H)) \lesssim \diam(B(R)).$$
\end{claim}
\begin{proof}
Let $P\colon \bX \rightarrow \tau_\Gamma$ denote the 
orthogonal projection onto $\tau_\Gamma$.

We first argue that the sets $\{P(5\ball(H)\cap\Gamma):H\in\mathcal{H}_k(R,\tau_k)\}$ are all disjoint. Suppose $P(5\ball(H)\cap\Gamma)\cap P(5\ball(H')\cap\Gamma) \neq \emptyset$ for some $H,H'\in\mathcal{H}_k(R,\tau_k)$. Choose $x\in 5\ball(H)\cap \Gamma$ and $y\in 5\ball(H')\cap \Gamma$ such that $P(x)=P(y)$.

By Claim \ref{claim:holes>2}(i), there are cubes $\hat{H}$ and $\hat{H'}$ contained in $R^{\uparrow}$ (playing the role of $R'$ from that claim) and both in $\cS_i$ such that $\ball(H)\subseteq \frac{\Lambda}{2}\ball(\hat{H})$, $\ball(H')\subseteq \frac{\Lambda}{2}\ball(\hat{H'})$, $\diam(\ball(\hat{H}))\lesssim \diam(\ball(H))$, and $\diam(\ball(\hat{H'}))\lesssim \diam(\ball(H'))$. From Claim \ref{claim:holes>2}(i), we also have $\Lambda \ball(\hat{H}) \subseteq 5\ball(R)$ and similarly for $\hat{H'}$.

Let $Q$ be a cube of $\cS_i$ with minimal $\diam(\ball(Q))$ such that $\Lambda \ball(Q)$ contains both $\frac{\Lambda}{2}\ball(\hat{H})$ and $\frac{\Lambda}{2}\ball(\hat{H'})$. (Note that $\ball(\hat{H}),\ball(\hat{H'})\subseteq \Lambda \ball(R)$,  so $R$ is a competitor for $Q$.) Since $7\ball(H)\cap 7\ball(H')=\emptyset$, we have
$$ \diam(\ball(Q)) \lesssim \dist(\ball(\hat{H}),\ball(\hat{H'})) + \diam(\ball(\hat{H})) + \diam(\ball(\hat{H'})) \lesssim_{\epsilon^{(k)}_2, \Lambda} |x-y|.$$

On the other hand, since $\theta_\Gamma(\Lambda \ball(Q))<\delta$ and $P(x)=P(y)$, it must be that
$$|x-y|\lesssim_{\epsilon_1, d, A, \Lambda} \delta\diam(\ball(Q)).$$
By our choice of $\delta$ small, this yields a contradiction. Therefore, the sets $\{P(5\ball(H)\cap\Gamma):H\in\mathcal{H}_k(R)\}$ are all disjoint.

Now, each $H\in\mathcal{H}_k(R,\tau_k)$ lies in $\mathcal{C}_\Gamma$, and so each $5\ball(H)\cap \Gamma$ contains an arc $\Gamma_H\subseteq\Gamma$ of diameter $\gtrsim \diam(\ball(H))$. Moreover, $\diam(P(\Gamma_H))\gtrsim \diam(H)$ by a very similar argument to that in the previous paragraph: Suppose $x,y\in\Gamma_H$ have $|x-y|=\diam(\Gamma_H)\gtrsim_d\diam(H)$ but $|P(x)-P(y)|\ll\diam(H)$. The cube $H$ is contained in a cube $\hat{H}$ of $\cS_i$ of size comparable to $H$, and so we reach a contradiction to the flatness condition $\theta_\Gamma(\Lambda \ball(\hat{H}))<\delta$. 

Therefore,
$$ \sum_{H\in\mathcal{H}_k(R,\tau_k)} \diam(B(H)) \lesssim \sum_{H\in\mathcal{H}_k(R,\tau_k)} \ell(P(\Gamma_H)) \lesssim \ell(P(5 \ball(R)))\lesssim \diam(B(R)).$$

\end{proof}

Given $Q\in \mathcal{T}_{i}$, let
\begin{equation}\label{eq:khat}
 \hat{k}(Q) = \min\{ k: n_k(Q) = 0 \} \in \{1, \dots, d-1\}.
\end{equation}
We assign to each cube $Q\in\mathcal{T}_{i}$ and $k\in\{0, \dots, \hat{k}(Q)\}$ a cube $H_k(Q)$ as follows:
First, set $H_0(Q)=Q(\cS_i)$ for all $Q\in\mathcal{T}_{i}$. Next, note that since $k\leq \hat{k}(Q)$, we have $n_{k-1}(Q)>0$.

If $n_k(Q)<n_{k-1}(Q)$ and $\Lambda Q^{\uparrow n_k(Q)} \subseteq 5Q^{\uparrow (n_{k-1}(Q)-1)}$, then we define $
H_k(Q)$ to be the cube $H_k(Q^{\uparrow (n_{k-1}(Q)-1)}, \tau^\epsilon_k(Q))\in\mathcal{H}_{k}(Q^{\uparrow (n_{k-1}(Q)-1)}, \tau^\epsilon_{k}(Q))$ assigned in Claim \ref{claim:holes>2}(iv).

If $n_k(Q)=n_{k-1}(Q)$ or $\Lambda Q^{\uparrow n_k(Q)} \not\subseteq 5Q^{\uparrow (n_{k-1}(Q)-1)}$, then we set $H_k(Q)=Q^{\uparrow n_{k-1}(Q)}$.

In the next step, we argue that (up to some caveats and scaling parameters), it should not be possible for a cube $Q\in\mathcal{T}_i$ to be contained within its associated $H_k(Q)$. Roughly, this is because points of $H_k(Q)$ are blocked ``too early'' in direction $\tau_k$, by Claim \ref{claim:holes>2}(iii).

\begin{claim}\label{claim:nocontain}
Suppose that $Q\in\mathcal{T}_{i}$,  $k\in\{1, \dots, \hat{k}(Q)\}$, $n_k(Q)<n_{k-1}(Q)$, and
$$ \Lambda \ball(Q^{\uparrow n_k(Q)}) \subseteq 5 \ball(Q^{\uparrow n_{k-1}(Q) - 1}).$$
Then we have 
$$ 5\ball(Q^{\uparrow n_k(Q)}) \not \subset 10\Lambda \ball(H_k(Q)).$$
\end{claim}
\begin{proof}
Suppose that $Q$ is as above but that nonetheless
$$ 5\ball(Q^{\uparrow n_k(Q)}) \subset 10\Lambda \ball(H),$$
where $H=H_k(Q)$.

Throughout, we will write $\tau^\epsilon_k = \tau^\epsilon_k(Q)$. Let $R = Q^{\uparrow n_{k-1}(Q) - 1}$, so $H\in\mathcal{H}_k(R, \tau^\epsilon_{k})$. By Claim \ref{claim:holes>2}(iii), there is also a cube $\hat{Q}\in\mathcal{T}_{i}$ with $R=\hat{Q}^{\uparrow (n_{k-1}(\hat{Q})-1)}$ and $\Lambda \ball(Q^{\uparrow n_k(Q)}) \subseteq 5\ball(R)$ such that each $x\in 10\Lambda \ball(H)$ is  
$\left(\frac{1}{2\Lambda}\epsilon^{(k+1)}_2,\tau_k\right)$-blocked in $\frac{\Lambda}{10}\ball(\hat{Q}^{\uparrow n_k(\hat{Q})})$.

Note that our assumptions imply that
\begin{equation}\label{eq:cubesizes}
5\ball(Q^{\uparrow n_k(Q)}) \subseteq 10\Lambda \ball(H) \subseteq \frac12 \Lambda \ball(\hat{Q}^{\uparrow n_k(\hat{Q})}).
\end{equation}

Let $m$ be the minimal non-negative integer such that $5\ball(Q^{\uparrow m})\supseteq \Lambda\ball(\hat{Q}^{\uparrow n_k(\hat{Q})})$. Note that $Q^{\uparrow m}\subseteq R$ and so $m\leq n_{k-1}(Q)-1$. Also, \eqref{eq:cubesizes} implies that $\diam(\ball(Q^{\uparrow m}))\lesssim \Lambda \diam (\ball(\hat{Q}^{\uparrow n_k(\hat{Q})}))$.

Thus, each $x\in 5Q^{\uparrow n_k}\subseteq 10\Lambda \ball(H)$ is  
$\left(C\Lambda^{-2}\epsilon^{(k+1)}_2,\tau^\epsilon_k\right)$-blocked in $Q^{\uparrow m}$, for some constant $C=C(A)$. Since $D(\tau^\epsilon_k, \tau^\epsilon_k(Q))\lesssim_d \epsilon^{(k)}_2$ (Claim \ref{claim:nk}), and
$$ C\Lambda^{-2}\epsilon^{(k+1)}_2 \gg_d \epsilon^{(k)}_2,$$
we have that each $x\in 5\ball(Q^{\uparrow n_k})$ is $(\epsilon^{(k)}_{2},\tau_k(Q))$-blocked in $Q^{\uparrow m}$.

It follows that
$$ N(Q, \langle \tau^\epsilon_{k-1}(Q), \tau_\Gamma \rangle) \leq N(Q, \tau_k(Q)) \leq m \leq n_{k-1}(Q)-1 < N(Q, \langle \tau^\epsilon_{k-1}(Q), \tau_\Gamma \rangle),$$
but this is a contradiction.

\end{proof}

Next, we show that a given cube $H$ controls the diameter sum of all cubes $Q^{\uparrow n_k(Q)}$ with $H_k(Q)=H$.

\begin{claim}\label{claim:cubes<hole}
Fix $k\in\{1, \dots, d-1\}$ and a cube $H\in\mathcal{C}_\Gamma$. Let 
$$\mathcal{C} = \{Q^{\uparrow n_k(Q)} : Q\in\mathcal{T}_{i}, H_k(Q) = H\}.$$

Then
$$ \sum_{R\in\mathcal{C}} \diam(B(R)) \lesssim \diam(B(H)).$$
\end{claim}
To be clear, if $R = Q^{\uparrow n_k(Q)} = (Q')^{\uparrow n_k(Q')}$ for two separate cubes $Q$ as above, then $R$ is summed only once in this claim.
\begin{proof}
Separate $\mathcal{C} = \mathcal{C}_{=} \cup \mathcal{C}_{<}$, where
$$\mathcal{C}_{=} = \{Q^{\uparrow n_k(Q)} : Q\in\mathcal{T}_{i}, H_k(Q) = H, n_k(Q)=n_{k-1}(Q)\}$$
and
$$\mathcal{C}_{<} = \{Q^{\uparrow n_k(Q)} : Q\in\mathcal{T}_{i}, H_k(Q) = H, n_k(Q)<n_{k-1}(Q)\}.$$

First consider $\mathcal{C}_{=}$. By definition, if $Q^{\uparrow n_k(Q)}\in \mathcal{C}_{=}$, then
$$ H = H_k(Q) = Q^{\uparrow n_{k-1}(Q)} = Q^{\uparrow n_k(Q)}.$$
Thus, if $\mathcal{C}_{=}$ is non-empty, then it consists only of the single cube $H$. So if $\mathcal{C}_{=}$ is non-empty, we must have
$$ \sum_{R\in\mathcal{C}_{=}} \diam(B(R)) \leq \diam(H)$$
in any case.

We now separate $\mathcal{C}_{<}$ further into
$$ \mathcal{C}_{<, \subseteq} = \{Q\in \mathcal{C}_{<}: \Lambda \ball(Q^{\uparrow n_k(Q)}) \subseteq 5\ball(Q^{\uparrow (n_{k-1}(Q)-1)})\}$$
and
$$ \mathcal{C}_{<, \not\subseteq} = \{Q\in \mathcal{C}_{<}: \Lambda \ball(Q^{\uparrow n_k(Q)}) \not\subseteq 5\ball(Q^{\uparrow (n_{k-1}(Q)-1)})\}$$

The cubes in $\mathcal{C}_{<, \subseteq}$ each have $n_k(Q)<n_{k-1}(Q)$ and $H_k(Q)=H$. Therefore, by Claims \ref{claim:holes>2} and \ref{claim:nocontain}, if $Q\in\mathcal{C}_{<, \subseteq}$, then
$$ 7\Lambda \ball(Q^{\uparrow n_k(Q)}) \cap 7\ball(H) \neq \emptyset$$
and
$$ 5\ball(Q^{\uparrow n_k(Q)}) \not\subseteq 10\Lambda \ball(H).$$

It follows that if $Q\in\mathcal{C}_{<, \subseteq}$, then
$$ \diam(B(Q^{\uparrow n_k(Q)}))\approx \diam(\ball(Q^{\uparrow n_k(Q)})) \approx \diam(\ball(H)) \approx \diam(B(H)).$$
It then follows from the doubling property of $E$ that the number of distinct $Q^{\uparrow n_k(Q)}$ for $Q\in \mathcal{C}_{<, \subseteq}$ is controlled, and hence
$$ \sum_{Q^{\uparrow n_k(Q)}\in\mathcal{C}_{<, \subseteq}} \diam(B(Q^{\uparrow n_k})) \lesssim \diam(B(H)).$$

Lastly, if $Q\in \mathcal{C}_{<, \not\subseteq}$, then $H=H_k(Q) = Q^{\uparrow n_{k-1}(Q)}$. In this case, $Q^{\uparrow n_k(Q)}\subseteq Q^{\uparrow n_{k-1}(Q)} = H$, but $\Lambda \ball(Q^{\uparrow n_k(Q)}) \not\subseteq 5\ball(H)$. Therefore, $B(Q^{\uparrow n_k(Q)})$ must again be comparable in size to $B(H)$. Thus, there can again be only a controlled number of distinct $Q^{\uparrow n_k(Q)}$ for $Q\in \mathcal{C}_{<, \not\subseteq}$, and hence
$$ \sum_{Q^{\uparrow n_k}\in\mathcal{C}_{<, \not\subseteq}} \diam(B(Q^{\uparrow n_k})) \lesssim \diam(B(H)).$$
This completes the proof of the claim.
\end{proof}

Next, we prove a similar claim in the opposite direction: a given cube $R\in\cS_i$ controls the diameter sum of all the cubes $H_k(Q)$ for which $R=Q^{\uparrow n_{k-1}(Q)}$. 

\begin{claim}\label{claim:holes<cube}
Fix $k\in\{1, \dots, d-1\}$ and $R\in\cS_i$. Let
$$ \mathcal{H} = \{H\in \mathcal{C}_\Gamma : H = H_k(Q) \text{ where } Q\in \mathcal{T}_{i} \text{ and } Q^{\uparrow n_{k-1}(Q)}=R\}.$$
Then
$$ \sum_{H\in\mathcal{H}} \diam(B(H)) \lesssim \diam(B(R)).$$
\end{claim}
\begin{proof}
We begin similarly to the previous claim. Write
$$ \mathcal{H} = \mathcal{H}_{=} \cup \mathcal{H}_{<},$$
where
$$ \mathcal{H}_{=} = \{H : H = H_k(Q), Q\in \mathcal{T}_{i},  Q^{\uparrow n_{k-1}(Q)}=R, \text{ and } n_k(Q)=n_{k-1}(Q)\}$$
and
$$ \mathcal{H}_{<} = \{H : H = H_k(Q), Q\in \mathcal{T}_{i},  Q^{\uparrow n_{k-1}(Q)}=R, \text{ and } n_k(Q)<n_{k-1}(Q)\}$$

Consider $\mathcal{H}_{=}$ first. If $H=H_k(Q)\in\mathcal{H}_{=}$, then by definition of $H_k$ we must have 
$$ H = H_k(Q) = Q^{\uparrow n_{k-1}(Q)} = R.$$
Thus, if $\mathcal{H}_{=}$ is non-empty, then it consists of only the single cube $R$. Therefore, 
$$ \sum_{H\in\mathcal{H}_{=}}\diam(B(H)) \leq \diam(B(R))$$
in any case.

We write $H\in \mathcal{H}_{<}$ as the union of two sub-collections:
$$ \mathcal{H}_{<,\subseteq} =  \{H : H = H_k(Q), Q\in \mathcal{T}_{i},  Q^{\uparrow n_{k-1}(Q)}=R, n_k(Q)<n_{k-1}(Q), \Lambda \ball(Q^{\uparrow n_k}) \subseteq 5\ball(Q^{\uparrow (n_{k-1}(Q)-1)})\},$$
$$ \mathcal{H}_{<,\not\subseteq} =  \{H : H = H_k(Q), Q\in \mathcal{T}_{i},  Q^{\uparrow n_{k-1}(Q)}=R, n_k(Q)<n_{k-1}(Q), \Lambda \ball(Q^{\uparrow n_k}) \not\subseteq 5\ball(Q^{\uparrow (n_{k-1}(Q)-1)})\}$$

If $H\in \mathcal{H}_{<,\not\subseteq}$, then 
$$ H = H_k(Q) = Q^{\uparrow n_{k-1}(Q)} = R,$$
and so if $\mathcal{H}_{<,\not\subseteq}$ is non-empty then it consists of only $R$ and we have
$$ \sum_{H\in\mathcal{H}_{<,\not\subseteq}}\diam(B(H)) \leq \diam(B(R)).$$

Finally, suppose that $H\in\mathcal{H}_{<,\subseteq}$. Then there is a child $S$ of $R$ with $S = Q^{\uparrow (n_{k-1}(Q) -1)}$ so that
$$ H = H_k(Q) = H_k(S, \tau^\epsilon_k(Q)) \in \mathcal{H}_k(S,\tau^\epsilon_k(Q)).$$
For each child $S$ of $R$, we have
\begin{align*}
&\sum_{\tau^\epsilon_k(Q),H: H\in \mathcal{H}_k(S,\tau^\epsilon_k(Q))} \diam(B(H))\\
&\leq
\sum_{\sigma_1 \in\mathcal{L}_1(S)} \sum_{\sigma_2 \in\mathcal{L}_2(S,\sigma_1)} \dots \sum_{\sigma_k\in\mathcal{L}_k(S,\sigma_{k-1})} \sum_{H\in\mathcal{H}_k(S,\sigma_k)} \diam(B(H))\\
&\lesssim \diam(B(S))\\
&\lesssim \diam(B(R))
\end{align*}
using Claims  \ref{claim:Hsum} and \ref{claim:numplanes}.

The doubling property of $E$ implies that $R$ has a controlled number of children $S$, and the claim now follows.

\end{proof}

Next, we use the previous two claims to show that the collection $\mathcal{T}_i$ of all ``bad'' cubes in one of our stopping time regions $\cS_i$ is controlled by the diameter of the top cube of $\cS_i$. 

\begin{claim}\label{claim:badd>2}
For each $i$, we have
$\sum_{Q\in\mathcal{T}_{i}} \diam(B(Q)) \lesssim \diam(B(Q(\cS_i))).$
\end{claim}
\begin{proof}
First of all, it suffices to fix $\hat{k}\in\{0,\dots, d-1\}$ and bound the sum over cubes $Q\in\mathcal{T}_{i}$ with $\hat{k}(Q)=\hat{k}$.

Suppose that $0\leq k < \hat{k}$ and we have a nested list of cubes in $\cS_i$ as follows:
\begin{equation}\label{eq:list}
 Q(\cS_i) = Q_0 \supseteq Q_1 \supseteq Q_2 \supseteq \dots \supseteq Q_k.
\end{equation}
We define two related families of cubes using this list:
\begin{align*}
 \mathcal{T}(Q_0, Q_1, \dots, Q_k) = &\{ Q\in\mathcal{T}_{i} :\hat{k}(Q)=\hat{k} \text{ and }\\
&Q^{\uparrow n_0(Q)} = Q_0, Q^{\uparrow n_1(Q)} = Q_1, \dots, Q^{\uparrow n_k(Q)} = Q_k\}
\end{align*}
and
\begin{equation*}
 \mathcal{S}(Q_0, Q_1, \dots, Q_k) = \{ R\in\cS_i : \exists Q\in \mathcal{T}(Q_0, Q_1, \dots, Q_k)\text{ with } R = Q^{\uparrow n_{k+1}(Q)}\}.
\end{equation*}

If we set $Q_0=Q(\cS_i)$, we can then write
\begin{equation*}
\sum_{Q\in\mathcal{T}_{i}, \hat{k}(Q)=k} \diam(B(Q)) = \sum_{Q_1\in\mathcal{S}(Q_0)} \sum_{Q_2\in\mathcal{S}(Q_0,Q_1)} \dots \sum_{ Q_{\hat{k}} \in \mathcal{S}(Q_0, \dots, Q_{\hat{k}-1})} \diam(B(Q_{\hat{k}})).
\end{equation*}
To replace $Q$ by $Q_{\hat{k}}$, we use the fact that if $\hat{k}(Q)=\hat{k}$ then $n_{\hat{k}}(Q)=0$.

We now alternate using Claims \ref{claim:cubes<hole} and \ref{claim:holes<cube} as follows:
\begin{align*}
&\sum_{Q_1\in\mathcal{S}(Q_0)} \sum_{Q_2\in\mathcal{S}(Q_0,Q_1)} \dots \sum_{ Q_{\hat{k}} \in \mathcal{S}(Q_0, \dots, Q_{\hat{k}-1})} \diam(B(Q_{\hat{k}}))\\
&\lesssim \sum_{Q_1\in\mathcal{S}(Q_0)} \sum_{Q_2\in\mathcal{S}(Q_0,Q_1)} \dots \sum_{Q_{\hat{k}-1}\in\mathcal{S}(Q_0,Q_1, \dots, Q_{\hat{k}-2}) } \sum_{ \substack{H=H_{\hat{k}}(Q)\\ Q \in \mathcal{T}(Q_0, \dots, Q_{\hat{k}-1}) } } \diam(B(H))\\
&\lesssim \sum_{Q_1\in\mathcal{S}(Q_0)} \sum_{Q_2\in\mathcal{S}(Q_0,Q_1)} \dots \sum_{Q_{\hat{k}-1}\in\mathcal{S}(Q_0,Q_1, \dots, Q_{\hat{k}-2}) } \diam(B(Q_{\hat{k}-1}))\\
&\lesssim \sum_{Q_1\in\mathcal{S}(Q_0)} \sum_{Q_2\in\mathcal{S}(Q_0,Q_1)} \dots \sum_{Q_{\hat{k}-2}\in\mathcal{S}(Q_0,Q_1, \dots, Q_{\hat{k}-3}) } \sum_{ \substack{H=H_{\hat{k}-1}(Q)\\ Q \in \mathcal{T}(Q_0, \dots, Q_{\hat{k}-2}) } } \diam(B(H))\\
&\lesssim \sum_{Q_1\in\mathcal{S}(Q_0)} \sum_{Q_2\in\mathcal{S}(Q_0,Q_1)} \dots \sum_{Q_{\hat{k}-2}\in\mathcal{S}(Q_0,Q_1, \dots, Q_{\hat{k}-3}) } \diam(B(Q_{\hat{k}-2}))
\end{align*}
We continue to alternate in this way, stripping off one sum at a time. The final steps conclude as follows:
\begin{align*}
\dots \lesssim \sum_{Q_1\in\mathcal{S}(Q_0)} \diam(B(Q_1)) &\lesssim \sum_{\substack{H=H_1(Q)\\Q\in\mathcal{T}(Q_0)}} \diam(B(H))\\
&\lesssim \diam(B(Q_0))\\
&= \diam(B(Q(\cS_i))).
\end{align*}
This completes the proof of the claim.
\end{proof}

Finally, the proof of Proposition \ref{prop:epstangent} now concludes as follows. Suppose that
$$ \beta^\tau_\Gamma( \ball(Q)) \geq \epsilon$$
for some $Q\in \hDelta$. If $\diam(B(Q))> \diam(\Gamma)$, such cubes $Q$ are handled in \eqref{eq:bigcubes}. 

If $\diam(B(Q))\leq\diam(\Gamma)$, then $Q\in\mathcal{C}_\Gamma = \mathcal{B} \cup \bigcup_i \mathcal{S}_i$. The sum of diameters of cubes in $\mathcal{B}$ is immediately handled by Proposition \ref{prop:str}(6). If $Q\in \mathcal{S}_i$ for some $i$, then Proposition \ref{prop:str}(4) tells us that
$$ \angle(\tau_\Gamma(\Lambda \ball(Q)), \tau_\Gamma(\Lambda \ball(Q(\cS_i))) < \alpha.$$
It follows that $\Gamma\cap \ball(Q)$ lies in the $\Lambda(\delta +\alpha)$-neighborhood of a line in direction $\tau_\Gamma(\Lambda \ball(Q(\cS_i)))$, and $\Lambda(\delta +\alpha)\ll \epsilon$. Since $\beta^\tau_\Gamma( \ball(Q)) \geq \epsilon$, it must be that
$$ \angle(\tau(Q), \tau_\Gamma(\Lambda \ball(Q(\cS_i))) \geq \epsilon_1,$$
i.e., $Q\in\mathcal{T}_i$.

Therefore, in total we have
\begin{align*}
\sum_{Q\in\hDelta_E, \beta_\Gamma^\tau(\ball(Q))\geq \epsilon} \diam(B(Q)) &\lesssim \ell(\Gamma) + \sum_{Q\in\mathcal{C}_\Gamma, \beta_\Gamma^\tau(\ball(Q))\geq \epsilon} \diam(B(Q)) \\
&\lesssim \ell(\Gamma) + \sum_{i} \sum_{Q\in\mathcal{T}_i}\diam(B(Q))\\ 
&\lesssim \ell(\Gamma) + \sum_{i} \diam(B(Q(\cS_i)))\\
&\lesssim \ell(\Gamma),
\end{align*}
using Proposition \ref{prop:str}(5) in the last step.
\end{proof}

\section{Proofs of the ``coarse'' Theorems \ref{thm:main}, \ref{thm:onedim}, \ref{thm:converse}}\label{sec:finalproofs}
In this section, we complete the proofs of Theorems \ref{thm:main}, \ref{thm:onedim}, and \ref{thm:converse} on the existence of coarse tangent fields.

\begin{proof}[Proof of Theorem \ref{thm:main}]
Let $E$ be a doubling subset of a Hilbert space $\bX$ that badly fits $d$-planes, with parameter $\epsilon_0>0$. Let $\cF$ be a multiresolution family of balls for $E$ and $\Delta=\Delta_E$ be a system of dyadic cubes for $E$ as in Proposition \ref{prop:christ}. Fix $\epsilon>0$, $A\geq 1$.

Fix $\epsilon'<<_{A} \epsilon$. Proposition \ref{prop:epstangent} provides a map
$$ \tau_{\epsilon'} \colon \{AB(Q):Q\in\hDelta\} \rightarrow \mathcal{L}_k$$
with the following property: If $\Gamma$ is a continuum in $\bX$, then
\begin{equation}\label{eq:epstangentproof}
 \sum_{\substack{Q\in\hDelta \\ \Gamma \cap B(Q)\neq \emptyset\\ \beta^{\tau_{\epsilon'}}_{\Gamma}(AB(Q))\geq \epsilon'}} \diam(B(Q)) \lesssim \ell(\Gamma)
\end{equation}

To each $B\in\cF$, we may associate a cube $Q_B\in \Delta$ with $AB\subseteq AB(Q_B)$ and $\diam(B(Q_B))\approx \diam(B)$. This assignment is bounded-to-one by the doubling property of $E$. We now simply define
$$ \tau \colon \cF_A \rightarrow \mathcal{L}_k$$
by $\tau(AB) = \tau(AB(Q_B))$.

We claim that $\tau$ is a coarse $\epsilon$-tangent field for $\cF$ with inflation $A$. Consider any continuum $\Gamma$ in $\bX$.  Let
$$ \mathcal{C} = \{ B\in\cF : \Gamma \cap B \neq \emptyset, \diam(B)\leq \diam(\Gamma)\}.$$
We then write
$$ \mathcal{C} = \mathcal{C}_1 \cup \mathcal{C}_2$$
where
$$ \mathcal{C}_1 = \{B \in\mathcal{C} : \theta_\Gamma(AB) \geq \epsilon'\} \ \ \text{ and }\ \ \mathcal{C}_2 = \{B \in\mathcal{C} : \theta_\Gamma(AB) < \epsilon'\}.$$

The balls in $\mathcal{C}_1$ have total diameter sum controlled by $\ell(\Gamma)$, by Proposition \ref{p:TST-betas-bounded}.

For $\mathcal{C}_2$, we split further as
$$ \mathcal{C}_{2,1} = \{B\in\mathcal{C}_2 : \beta^\tau_\Gamma(AB) \geq \epsilon' \}$$
and 
$$ \mathcal{C}_{2,2} = \{B\in\mathcal{C}_2 : \beta^\tau_\Gamma(AB) < \epsilon' \}.$$

Using Proposition \ref{prop:epstangent}, we see that
$$ \sum_{B\in \mathcal{C}_{2,1}} \diam(B) \lesssim \sum_{B\in \mathcal{C}_{2,1}} \diam(B(Q_B)) \lesssim \ell(\Gamma).$$

It remains to consider balls in $\mathcal{C}_{2,2}$. However, we claim that none of these balls actually appear in the sum \eqref{eq:epstangentdef}. Indeed, if $B\in \mathcal{C}_{2,2}$ then $\theta_\Gamma(AB)<\epsilon'$ and $\beta^\tau_\Gamma(AB) < \epsilon'$. This means that there are affine lines $L_1$ and $L_2$, with $L_2 \parallel \tau(B)$ such that
$$ \frac{1}{\diam(AB)} \left( \sup_{x\in \Gamma \cap AB} \dist(x,L_1) + \sup_{y\in L_1 \cap AB} \dist(x,\Gamma)\right) \leq \epsilon'$$
and
$$\frac{1}{\diam(AB)} \sup_{x\in \Gamma \cap AB} \dist(x,L_2) \leq \epsilon'.$$
From these two, it easily follows that
$$ \frac{1}{\diam(AB)} \left( \sup_{x\in \Gamma \cap AB} \dist(x,L_2) + \sup_{y\in L_2 \cap AB} \dist(x,\Gamma)\right) \lesssim_A \epsilon' < \epsilon/2.$$
Hence,
$$ \theta^\tau_\Gamma(AB) < \epsilon$$
for each $B\in \mathcal{C}_{2,2}$. This completes the proof that $\tau$ is a coarse $\epsilon$-tangent field on $\cF$ with inflation $A$, and hence proves Theorem \ref{thm:main}.
\end{proof}

To prove Theorem \ref{thm:onedim}, it now suffices to show the following.

\begin{theorem}\label{thm:coarse}
Let $\bX$ be a Hilbert space, $E$ a subset of $\bX$ with multiresolution family $\cF$, and $A\geq 1$. Suppose that for each $\epsilon>0$, $\cF$ admits a $1$-dimensional coarse $\epsilon$-tangent field with inflation $A$.

Then $\cF$ admits a $1$-dimensional coarse tangent field with inflation $A$.
\end{theorem}

The proof involves taking a sequence $\tau_{2^{-k}}$ of $1$-dimensional coarse $2^{-k}$-tangent fields and aggregating this information into a single coarse tangent field. This aggregation is a bit subtle, as the following simple example indicates.

\begin{example} Fix $A\geq 1$, $E=S^1\subset \bX = \mathbb{R}^{2}$, and $\cF$ a multiresolution family for $E$. Let us construct a sequence of coarse $2^{-k}$-tangent line fields for $\cF$ (with inflation $A$). If $B\in \cF$, then there is a value $k(B)$ so that for $k\leq k(B)$ the line $\tau_{2^{-k}}(AB)$ can be chosen as the tangent line to the circle $S^1$ at some point $p\in B\cap S^{1}$. However, if $k>k(B)$ then $\tau_{2^{-k}}(AB)$ can be chosen to be $\{0\}$ (or an arbitrary line), since for such $k$ the circle $S^1$ is not close enough to a line in $AB$ to ``merit'' a non-trivial tangent line. (The parameter $k(B)$ depends on the curvature of the circle in $B$ relative to the scale of $B$.)

We see that $\tau(AB)$ can therefore not be defined na\"ively as the limit of $\tau_{2^{-k}}(AB)$ as $k\rightarrow\infty$. Instead, we ought to have a type of stopping condition; in this case, we can set $\tau(B)=\tau_{2^{-k(B)}}(AB)$. In general, the set $E$ may not be smooth, and the choice is more involved. However, this case still can be helpful to consider while reading the following argument.
\end{example}

We also remark that the aggregation we now do to prove Theorem \ref{thm:coarse} is possible in the case of $1$-dimensional tangent fields, but it is not clear to us how to do it for higher-dimensional tangent fields.

\begin{proof}[Proof of Theorem \ref{thm:coarse}]
If $\tau$ is a line in $\bX$, we write 
${\rm Cone}(\tau,\epsilon)=\{v\in \bX : \|v\|=1, \angle(\langle v \rangle, \tau) \leq \epsilon\}.$

Let $\tau^{2^{-k}}$ be $1$-dimensional coarse $2^{-k}$-tangent fields for $\cF$ with inflation $A$. To simplify the notation, we will write $\tau^{2^{-k}}(B)$ for $\tau^{2^{-k}}(AB)$. Let $B\in \cF$ be arbitrary. 
 Let $K$ be the maximal integer such that 
$$\bigcap_{k=1}^{K} {\rm Cone}(\tau^{2^{-k}}(B),2^{-k})\neq \emptyset,$$
and choose a subspace $\tau(B)$ spanned by a non-zero vector in $\bigcap_{k=1}^{K} {\rm Cone}(\tau^{2^{-k}}(B),2^{-k})$. We will show that $B\mapsto \tau(B)$ is a coarse $\epsilon$-tangent field for $\cF$ with inflation $A$, for every $\epsilon>0$.

Fix $\epsilon>0$ and let $\Gamma$ be a curve in $\bX$. Write
$$ \mathcal{C}_\Gamma = \{B\in \cF:  \Gamma \cap B \neq \emptyset, \diam(B)\leq \diam(\Gamma). $$

 Our goal is to show that
\begin{equation}\label{eq:betatausum}
\sum_{\substack{B\in \mathcal{C}_\Gamma\\ \theta^\tau_\Gamma(AB)\geq \epsilon}} \diam(B) \lesssim_{\epsilon} \ell(\Gamma).
\end{equation}

We have 
\begin{equation}
\sum_{\substack{B\in \mathcal{C}_\Gamma\\ \theta_\Gamma(AB)\geq \epsilon/10}} \diam(B) \lesssim_{\epsilon} \ell(\Gamma)
\end{equation}
by Proposition \ref{p:TST-betas-bounded}.

Let $\mathcal{C}_1$ be the collection of $B\in \cC_\Gamma$ for which
$\theta_\Gamma(AB)< \epsilon/10$. 
For each $B\in\cF$, let $\tau_\Gamma(B)$ be a line attaining the infimum in the definition of $\theta_\Gamma(AB)$, up to a factor of $2$. Observe that if $B\in\cC_1$ and $\angle(\tau_\Gamma(B), \tau(B))<\epsilon/10$, then  $\theta^\tau_\Gamma(AB)<\epsilon$. We thus have
\begin{equation}
\sum_{\substack{B\in \mathcal{C}_\Gamma\\ \theta^\tau_\Gamma(AB)\geq \epsilon}} \diam(B) \lesssim_{\epsilon} \ell(\Gamma) + \sum_{\substack{B\in \cC_1\\ \angle(\tau_\Gamma(B), \tau(B))\geq \epsilon/10}} \diam(B).
\end{equation}

Set $k_0=\lceil -\log_2(\epsilon)\rceil + 10$. Let $\cC_2$ be the set of $B\in \cC_1$ for which $\angle(\tau_\Gamma(B), \tau^{2^{-k}}(B))\leq 2^{-k}$ for all $k=1,\dots, k_0$. Since $\tau^{2^{-k}}$ are each coarse $2^{-k}$-tangent fields and $k_0$ depends only on $\epsilon$, we get
\begin{equation}\label{eq:lastsum}
\sum_{\substack{B\in \cC_\Gamma\\ \theta^\tau_\Gamma(AB)\geq \epsilon}} \diam(B) \lesssim_{\epsilon} \ell(\Gamma) + \sum_{\substack{B\in \cC_2\\ \angle(\tau_\Gamma(B), \tau(B))\geq \epsilon/10}} \diam(B).
\end{equation}

We now argue that the last sum in \eqref{eq:lastsum} is empty, and thus the bound \eqref{eq:betatausum} follows. If $B\in \cC_2$, then $\angle(\tau_\Gamma(B), \tau^{2^{-k}}(B))\leq 2^{-k}$ for each $k=1,\dots, k_0$. Therefore, in that case, $\tau_\Gamma(B)\in \bigcap_{k=1}^{k_0} {\rm Cone}(\tau^{2^{-k}}(B),2^{-k})$. By construction, it follows that $\tau(B)$ is also $\bigcap_{k=1}^{k_0} {\rm Cone}(\tau^{2^{-k}}(B),2^{-k})$. Thus $\angle(\tau_\Gamma(B), \tau(B))\leq 2^{1-k_0}<\epsilon/10$, which implies that the final sum in \eqref{eq:lastsum} is empty. This proves \eqref{eq:betatausum} and hence the result.

\end{proof}

Finally, we show that the property of badly fitting $d$-planes is not only sufficient but also necessary to admit coarse $\epsilon$-tangent fields for small $\epsilon$.
\begin{proof}[Proof of Theorem \ref{thm:converse}]
With the notation of the theorem, suppose to the contrary that the set $E$ does \emph{not} badly fit $d$-planes. It follows immediately that for each $\delta>0$, we can find a set $D\subseteq \bX$, isometric to a $d$-ball of radius $r$, that is contained in the $\delta r$-neighborhood of $E$. By scaling, we may assume that $r=1$ and $\diam(E)\geq 1$.

Consider the collection $\mathcal{C}$ of all lines in the $d$-plane spanned by $D$ that pass through $D_0\subset D$, a $d$-ball concentric with $D$ of half the radius. We put a measure $m$ on this collection by pushing forward the natural product measure on $\mathbb{S}^{d-1}\times [-\frac14,\frac14]^{d-1}$, i.e., we describe a line by a direction in $\mathbb{S}^{d-1}$ and an orthogonal translation. Note that each line $L$ in this collection $\cC$ has $\frac14 \leq \ell(L\cap D) \lesssim_d 1$, and that $m(\cC)\approx_d 1$. 

Fix $\epsilon>0$ small depending on $d,A$. Observe that if $B\in \cF$ and $\frac14 B \cap D_0 \neq \emptyset$, then
$$ m(\{\Gamma\in\cC : \theta^\tau_\Gamma(AB) \geq \epsilon\}) \gtrsim \diam(B)^{d-1}.$$
Indeed, $\tau(AB)$ is a $(d-1)$-plane, so all lines $\Gamma$ in $\cC$ that are nearly orthogonal to $\tau(AB)$ and pass through $\frac12 B \cap D_0$ will qualify; this yields the bound above.

Because the $d$-ball $D_0$ lies in the $\delta$-neighborhood of $E$, we have
$$ \sum_{\substack{B\in \cF\\ \frac14 B \cap D_0 \neq \emptyset \\ \diam(B) \leq \frac14}} \diam(B)^d \gtrsim \log(1/\delta).$$
(The $d$-ball $D_0$ is covered by this collection of balls at each scale between $\frac14$ and $100\delta$.)

On the other hand, our bound above yields
\begin{align*}
\sum_{\substack{B\in \cF\\ \frac14 B \cap D_0 \neq \emptyset\\ \diam(B) \leq \frac14}} \diam(B)^d &\lesssim  \sum_{\substack{B\in \cF\\ \frac14 B \cap D_0 \neq \emptyset\\ \diam(B) \leq \frac14}} \diam(B)\cdot m(\{\Gamma\in\cC : \theta^\tau_\Gamma(AB) \geq \epsilon\})\\
&\lesssim \int_{\cC} \sum_{\substack{B\in\cF\\\theta^\tau_\Gamma(AB) \geq \epsilon\\ \diam(B) \leq \frac14}} \diam(B) \,dm.
\end{align*}

It follows that there is a line segment $\Gamma \in \cC$ (necessarily with diameter at least $\frac14$) with 
$$ \sum_{\substack{B\in\cF\\\theta^\tau_\Gamma(AB) \geq \epsilon\\ \diam(B) \leq \frac14}} \diam(B) \gtrsim \log(1/\delta).$$

Since $\delta$ could have been chosen arbitrarily small and $\ell(\Gamma)\lesssim 1$, this contradicts the assumption that $\tau$ was a coarse $\epsilon$-tangent field for $\cF$ with inflation $A$.
\end{proof}
\section{Counterexample: proof of Theorem \ref{thm:example}}\label{sec:example}
In this section, we prove Theorem \ref{thm:example}. Recall that this concerns the question of whether $L^2$-estimates on our coarse tangent field hold, as in the Traveling Salesman Theorem (Theorem \ref{thm:tst}), in comparison to the weak-type estimates given in Definition \ref{def:coarsetangent}. Here we give an example showing that $L^2$-estimates do \emph{not} hold in general.

Our counterexample is given by a construction in the plane reminiscent of a von Koch snowflake, as well as examples of Laakso \cite{La02} and Lang--Plaut \cite{LangPlaut}. Let $I=[x_0,x_1]$ be a segment in $\RR^2$, and let $x_t=tx_1+(1-t)x_0$. We will iteratively replace such intervals by diamond-shaped sets as follows. Let $J:\RR^2\to \RR^2$ be an orthogonal rotation of $90$ degrees. For $a\in(0,1)$ let 
\begin{align}
T_{a}(I)=&[x_0,x_{1/4}] \cup [x_{3/4},x_{1}] \cup [x_{1/4},x_{1/2}+aJ(y-x)] \cup [x_{1/4},x_{1/2}-aJ(y-x)] \cup \nonumber \\
&[x_{1/2}+aJ(y-x),x_{3/4}] \cup [x_{1/2}-aJ(y-x),x_{3/4}] \label{eq:Ta}
\end{align}
This operation constructs a polygonal ``diamond'' shape out of $I$; see Figure \ref{fig:counterexample}. By a polygonal set $P$ we mean a union of finitely many line segments. If $P$ is a polygonal set, let $S_N(P)$ be the subdivision of each segment of $P$ by a factor of $N$, and let $T_{a}(P)$ be the polygonal set obtained by applying the operation $T_a$ to each of the segments in $P$. 
\begin{figure}

\begin{tikzpicture}[scale=3]
    % Define endpoints of the original interval
    \coordinate (X) at (0,0);
    \coordinate (Y) at (4,0);
    
    % Compute subdivision points
    \coordinate (M1) at ($(X)!1/4!(Y)$);
    \coordinate (M2) at ($(Y)!1/4!(X)$);
    
    % Define parameters c and h
    \def\h{0.2}    % Controls the height of the diamond
    
    % Compute rotated vectors (J rotates by 90 degrees counterclockwise)
    \coordinate (Jh) at (0,\h);
    
    \coordinate (T1) at ($0.5*(X) + 0.5*(Y) + (Jh)$);
    \coordinate (T2) at ($0.5*(X) + 0.5*(Y) - (Jh)$);
    
    % Draw original interval
    \draw[thick] (X) -- (M1);
    \draw[thick] (M2) -- (Y);
    
    % Draw the diamond structure
    \draw[thick] (M1) -- (T1) -- (M2);
    \draw[thick] (M1) -- (T2) -- (M2);

\end{tikzpicture}
\caption{One step of the operator $T_{a}$ applied to a horizontal line segment.}
\label{fig:counterexample}
\end{figure}
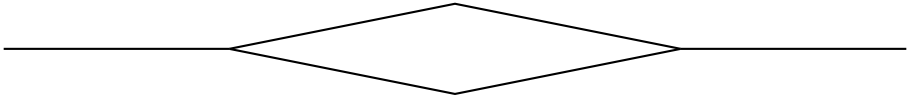

Let $a_n \in (0,10^{-3})$ be a decreasing sequence for which  we have $\sum_{n=1}^\infty a_n^2=1<\infty$ and $\sum_{n=1}^\infty a_n^2 \log(a_n^{-1})=\infty$.  Let $N_n$ be a sequence of integers converging rapidly to infinity, in a way to be described below, to ensure that the construction below produces a porous subset.

Let $P_1=[(0,0),(1,0)]$, and recursively define sets $P_{n+1}=T_{a_n}(S_{N_n}(P_n))$ by first subdividing the edges and then applying the diamond construction. We also fix an arbitrary multiresolution family $\mathcal{F}$ for $P$ and an inflation factor $A\geq 3$.

The following lemma immediately implies Theorem \ref{thm:example}.

\begin{lemma}\label{lem:counterexample} There exists a sequence $N_n\to \infty$ so that the following hold.
\begin{enumerate}[(i)]
\item The sequence $P_n$ converges in the Hausdorff metric to a compact, porous set $P\subset \mathbb{R}^2$.

\item There exists a probability measure $\eta$, supported on the collection of rectifiable curves contained in $P$, such that for any one-dimensional coarse tangent field $\tau\colon \cF_A \rightarrow \mathcal{L}_1$, we have
\begin{equation}\label{eq:notstrongtype}
\mathbb{E}_\eta\left(\sum_{\substack{B\in \cF\\ \Gamma \cap B \neq \emptyset}}\beta_\Gamma^\tau(AB)^2 \diam(B)\right)=\infty.
\end{equation}
\end{enumerate}
\end{lemma}

Recall that a set in the plane is \emph{porous} if there exists a constant $\delta>0$ so that for every $x\in \mathbb{R}^2$ and $r>0$ there exists a $y\in B(x,r)$ for which $B(y,\delta r) \cap P = \emptyset$.

\begin{proof}[Proof of Lemma \ref{lem:counterexample}]

We will prove the lemma through a sequence of simpler claims.
\begin{claim}\label{claim:claim1}%claim 1
The edges in the definition of $P_n$ are disjoint.
\end{claim}
\begin{proof}
 Fix $a,b\in (0,1/4).$  If $I=[x_0,x_1]$ is an edge, let  $D_b(I)$ be the parallelogram with corners at $x_0,x_{1/2}+bJ(x_1-x_0), x_1, x_{1/2}-bJ(x_1-x_0).$
The following facts are simple to check.
\begin{enumerate}
\item For any interval $I$, we have that $T_{a}(I)\subset D_{a}(I)$, and if $b>a$, then $T_{a}(I)\cap \partial D_b(I)=\{x_0,x_1\}$. 
\item If $I_1,I_2$ are distinct edges in $T_a(I)$, then $D_a(I_1), D_a(I_2)$ are pairwise disjoint except at their boundaries. 
\item If $b<a$ and $J$ is any subsegment of an edge in $T_a(I)$, then $D_b(J)\subset D_{a}(I)$. 
\end{enumerate}
From these, it is direct to see by induction that the edges in $P_n$ are pairwise disjoint. 
\end{proof}

\begin{claim}\label{claim:claim2}%claim 2
$d_H(P_n,P_{n+k})\leq 10^{-3}s_n/N_n \leq 2^{-n}$, where $s_n$ is the length of the longest edge in $P_n$.
\end{claim}
\begin{proof}
This is immediate since by the proof of the first claim we have that $P_{n+k} \subset \bigcup_{I\in S_{N_n}(P_n)} D_{a_n}(I)$ and $d_H(D_{a_n}(I),I)\leq a_n\diam(I)$. 
\end{proof}

It follows from Claim \ref{claim:claim2} that the sequence $P_n$ is Cauchy in the Hausdorff metric. We define $P$ be the Hausdorff limit of the sets $P_n$. 

\begin{claim}\label{claim:claim3}
There exist sequences $N_n\nearrow \infty$ and $r_n \searrow 0$ such that if $r\leq r_{n+1}$ then $B(x,2r)$ can intersect at most three edges of $P_n$. Moreover, these sequences can be chosen so that $s_{n}/N_n < r_{n+1}<s_n < r_n$ for each $n\in\mathbb{N}$.
\end{claim}
\begin{proof}
This follows easily from the fact that every point in $P_n$ is adjacent to at most three edges. Thus, we may first set $r_1=1$ and then choose $r_n$ and $N_n$ recursively so that $s_{n}/N_n < r_{n+1}$.
\end{proof}

\begin{claim}\label{claim:claim4}
The set $P$ is porous.%:
\end{claim}
\begin{proof}
Since $P$ is compact, it suffices to check porosity for radii $r\in (0,1)$. If $x\in \mathbb{R}^2$ and $r\in (0,1)$, then there exists an $m\geq 1$ so that $r\in (r_{m+1},r_m]$. If $m=1$, then $r>r_2$, and $P$ is within a $10^{-3} r_2 \leq 10^{-3} r$ neighborhood of six line segments, and thus by a simple volume argument we see that with $\delta<20^{-1}$ there exists a $y\in B(x,r)$ with $B(y,\delta r)\cap P=\emptyset$.

Next, let $m>1$. The ball $B(x,2r_m)$ can intersect at most three edges of $P_{m-1}$, and $P$ is contained within an $r_m$ neighborhood of $P_{m-1}$. This, together with the construction, implies that there are at most $54$ edges of $P_m$ with length greater than $r$ that intersect $B(x,r)$. Thus, $P\cap B(x,3r/2)$ is contained in the union of at most $54$ tubes of width at most $10^{-3}r_{m+1}\leq 10^{-3} r$ and three tubes of width at most $10^{-3}r$. Thus, again by a simple volume argument we see that with $\delta<10^{-3}$ there exists a $y\in B(x,r)$ with $B(y,\delta r)\cap P=\emptyset$.
\end{proof}

Next, we will construct the desired probability measure $\eta$ on curves. We write $\mathbf{b}=(b_n)_{n=1}^\infty$ to denote an element of $\{-1,1\}^\infty$ and let $\mathbb{P}$ be the product measure on $\{-1,1\}^\infty$ corresponding to independent and identically distributed coin flips. 

Observe that the vertices of $P_n$ are always vertices of $P_{n+1}$ for all $n\in \mathbb{N}$. Thus, we can define a path in $P$ by describing, for each $n$, an ordered list of vertices of $P_n$ in such a way that these orders are compatible for different $n$. Recall also that each $P_{n+1}$ is formed from $P_n$ by sudivision and then doubling the middle portion of each edge. We recursively associate to each edge of $P_n$ an orientation from left-to-right, in such a way that if $[x_0,x_1]$ was positively oriented, then all the edges in  \eqref{eq:Ta} are positively oriented. For subdivisions, there is similarly a natural orientation.  We call a path in some $P_n$ \emph{monotone} if every edge it traverses is traveled in the positive orientation. A curve in $P$ is called \emph{monotone}, if the vertex path in $P_n$ that it defines is monotone for each $n$. 

Now, for every $\mathbf{b}\in \{-1,1\}^\mathbb{N}$, choose a monotone path $\gamma_\mathbf{b}$ from $(0,0)$ to $(1,0)$ as follows: At every vertex of $P_n$, where a ``decision'' needs to be made, the path chooses the edge cyclically oriented to the left if $b_n=+1$ and the right edge if $b_n=-1$. To be precise, this defines a path in $P_m$ for every finite $m$, and the desired path in $P$ is obtained as their limit. Informally, these curves are constructed by choosing the top or bottom half of each diamond in the construction of $P_n$, depending on the $n$th digit of $b_n$. 

A direct calculation using the Pythagorean theorem gives 
\[
\ell(\gamma_{\mathbf{b}})\lesssim \prod_{n=1}^\infty \sqrt{1+a_n^2} \leq \sqrt{e^{\sum a_n^2}} <\infty
\]
for each $\mathbf{b}$.

We now suppose that $\tau\colon \cF_A \rightarrow \mathcal{L}_1$ is a coarse line field on $\cF$ with inflation $A$, and argue that \eqref{eq:notstrongtype} holds. 
For each $n\geq 2$ and each edge $I$ in $S_{N_{n-1}}(P_{n-1})$, let $\cF^n(I)$ denote the collection of balls in $\cF$ that intersect the middle half of $I$ and have diameter in $[a_n|I|,|I|]$; such balls $B$ will then have $AB$ intersect both sides of a diamond in $P_n$. Note that if we choose $N_n \gg 1/a_n$, then a fixed ball $B\in \cF$ can lie in $\cF^n(I)$ for at most one value of $n$ and at most a bounded number of different $I$. Because these balls span roughly $\log(a_n^{-1})$ scales, we have for each such $n$, $I$, that
\begin{equation}\label{eq:diameters}
\sum_{B\in \cF^n(I)} \diam(B) \gtrsim \log(a_n^{-1}) |I|.
\end{equation}

Let $\mathbf{b}^{-,n}$ be obtained from $\mathbf{b}$ by flipping the sign of $b_n$. 
Suppose that $B\in \cF^n(I)$ (for some $n$, $I$) and $\gamma_\mathbf{b}$ passes through an end point of $I$. Then, no matter what line $\tau(B)$ has chosen, we have 
\[
\beta^\tau_{\gamma_{\mathbf{b}}}(AB)+\beta^\tau_{\gamma_{\mathbf{b}^{-,n}}}(AB) \gtrsim a_n,
\]
since the two paths travel on opposite sides of the diamond and make an angle of order $a_n$ with each other. Let $\mathbf{I}_\mathbf{b}$ be the collection of edges $I$ of $P_n$ such that $\gamma_\mathbf{b}$ passes through an endpoint of $I$. With \eqref{eq:diameters} and the previous estimate, we get
\begin{equation}\label{eq:diameters2}
\sum_{I \in \mathbf{I}_\mathbf{b}} \sum_{B\in \cF^n(I)} (\beta^\tau_{\gamma_{\mathbf{b}}}(AB)^2+\beta^\tau_{\gamma_{\mathbf{b}^{-,n}}}(AB)^2)\diam(B) \gtrsim a_n^2 \log(a_n^{-1}).
\end{equation}

Note that $\mathbf{I}_\mathbf{b}=\mathbf{I}_{\mathbf{b}^{-,n}}$, and that $\mathbf{b}\to \mathbf{b}^{-,n}$ preserves the probability measure. By applying this, integrating over $\mathbf{b}$ and summing over $n$ we obtain the desired claim:
\begin{align*}
\mathbb{E}_\mathbf{b}\left(\sum_{B\in \cF, B \cap \gamma \neq \emptyset} \beta^\tau_{\gamma_{\mathbf{b}}}(AB)^2\diam(B)\right)&\gtrsim\mathbb{E}_\mathbf{b}\left(\sum_{n=1}^\infty\sum_{I \in \mathbf{I}_\mathbf{b}} \sum_{B\in \cF^n(I)} \beta^\tau_{\gamma_{\mathbf{b}}}(AB)^2\diam(B)\right)\\
&=\sum_{n=1}^\infty \mathbb{E}_\mathbf{b}\left(\sum_{I \in \mathbf{I}_\mathbf{b}} \sum_{B\in \cF^n(I)} \beta^\tau_{\gamma_{\mathbf{b}}}(AB)^2\diam(B)\right)\\
&\gtrsim \sum_{n=1}^\infty a_n^2\log(a_n^{-1})=\infty.
\end{align*}

\end{proof}

\bibliographystyle{plain}
\bibliography{coarsetangentbib}

@misc{bateteriseb,
      title={Fragment-wise differentiable structures}, 
      author={David Bate and Sylvester Eriksson-Bique and Elefterios Soultanis},
      year={2024},
      eprint={2402.11284},
      archivePrefix={arXiv},
      primaryClass={math.CA},
note={https://arxiv.org/abs/2402.11284}, 
}

@article {teriseb,
    AUTHOR = {Eriksson-Bique, Sylvester and Soultanis, Elefterios},
     TITLE = {Curvewise characterizations of minimal upper gradients and the
              construction of a {S}obolev differential},
   JOURNAL = {Anal. PDE},
  FJOURNAL = {Analysis \& PDE},
    VOLUME = {17},
      YEAR = {2024},
    NUMBER = {2},
     PAGES = {455--498},
      ISSN = {2157-5045,1948-206X},
   MRCLASS = {46E36 (26B05 30L99 49J52 53C23)},
  MRNUMBER = {4713106},
       DOI = {10.2140/apde.2024.17.455},
       URL = {https://doi.org/10.2140/apde.2024.17.455},
}

@misc{Csornyei-Jones,
    title={Product formulas for measures and applications to analysis
and geometry},
    author={Cs\"ornyei, M. and Jones, P.},
    year={2011},
URL={http://www.math.stonybrook.edu/Videos/dfest/video.php?f=38-Jones-vb}
}

@article {Christ,
    AUTHOR = {Christ, Michael},
     TITLE = {A {$T(b)$} theorem with remarks on analytic capacity and the
              {C}auchy integral},
   JOURNAL = {Colloq. Math.},
  FJOURNAL = {Colloquium Mathematicum},
    VOLUME = {60/61},
      YEAR = {1990},
    NUMBER = {2},
     PAGES = {601--628},
      ISSN = {0010-1354,1730-6302},
   MRCLASS = {42B20 (42B25 42B30)},
  MRNUMBER = {1096400},
MRREVIEWER = {Michael\ J.\ Wilson},
       DOI = {10.4064/cm-60-61-2-601-628},
       URL = {https://doi.org/10.4064/cm-60-61-2-601-628},
}

@book {weaver,
    AUTHOR = {Weaver, Nik},
     TITLE = {Lipschitz algebras},
 PUBLISHER = {World Scientific Publishing Co., Inc., River Edge, NJ},
      YEAR = {1999},
     PAGES = {xiv+223},
      ISBN = {981-02-3873-8},
   MRCLASS = {46-02 (46Bxx 46Exx 46H05 46J10)},
  MRNUMBER = {1832645},
MRREVIEWER = {S.\ J.\ Sidney},
       DOI = {10.1142/4100},
       URL = {https://doi.org/10.1142/4100},
}

@article {heinonennonsmooth,
    AUTHOR = {Heinonen, Juha},
     TITLE = {Nonsmooth calculus},
   JOURNAL = {Bull. Amer. Math. Soc. (N.S.)},
  FJOURNAL = {American Mathematical Society. Bulletin. New Series},
    VOLUME = {44},
      YEAR = {2007},
    NUMBER = {2},
     PAGES = {163--232},
      ISSN = {0273-0979,1088-9485},
   MRCLASS = {49J52 (28A75 31C45 46E35 53C23)},
  MRNUMBER = {2291675},
MRREVIEWER = {Piotr\ Haj\l asz},
       DOI = {10.1090/S0273-0979-07-01140-8},
       URL = {https://doi.org/10.1090/S0273-0979-07-01140-8},
}

@article {BadgerMcCurdy,
    AUTHOR = {Badger, Matthew and McCurdy, Sean},
     TITLE = {Subsets of rectifiable curves in {B}anach spaces {I}: {S}harp
              exponents in traveling salesman theorems},
   JOURNAL = {Illinois J. Math.},
  FJOURNAL = {Illinois Journal of Mathematics},
    VOLUME = {67},
      YEAR = {2023},
    NUMBER = {2},
     PAGES = {203--274},
      ISSN = {0019-2082,1945-6581},
   MRCLASS = {28A75 (26A16 28A80 46B20 65K10)},
  MRNUMBER = {4593892},
MRREVIEWER = {Stefan\ Steinerberger},
       DOI = {10.1215/00192082-10592363},
       URL = {https://doi.org/10.1215/00192082-10592363},
}

@article {DavidSchul,
    AUTHOR = {David, Guy C. and Schul, Raanan},
     TITLE = {A sharp necessary condition for rectifiable curves in metric
              spaces},
   JOURNAL = {Rev. Mat. Iberoam.},
  FJOURNAL = {Revista Matem\'atica Iberoamericana},
    VOLUME = {37},
      YEAR = {2021},
    NUMBER = {3},
     PAGES = {1007--1044},
      ISSN = {0213-2230,2235-0616},
   MRCLASS = {28A75 (28C15)},
  MRNUMBER = {4236801},
       DOI = {10.4171/rmi/1216},
       URL = {https://doi.org/10.4171/rmi/1216},
}

@article {LeDonneRajala,
    AUTHOR = {Le Donne, Enrico and Rajala, Tapio},
     TITLE = {Assouad dimension, {N}agata dimension, and uniformly close
              metric tangents},
   JOURNAL = {Indiana Univ. Math. J.},
  FJOURNAL = {Indiana University Mathematics Journal},
    VOLUME = {64},
      YEAR = {2015},
    NUMBER = {1},
     PAGES = {21--54},
      ISSN = {0022-2518,1943-5258},
   MRCLASS = {54F45 (53C17 53C23 54E35)},
  MRNUMBER = {3320519},
MRREVIEWER = {Takahisa\ Miyata},
       DOI = {10.1512/iumj.2015.64.5469},
       URL = {https://doi.org/10.1512/iumj.2015.64.5469},
}

@article {TysonWu,
    AUTHOR = {Tyson, Jeremy T. and Wu, Jang-Mei},
     TITLE = {Characterizations of snowflake metric spaces},
   JOURNAL = {Ann. Acad. Sci. Fenn. Math.},
  FJOURNAL = {Annales Academi\ae\ Scientiarum Fennic\ae. Mathematica},
    VOLUME = {30},
      YEAR = {2005},
    NUMBER = {2},
     PAGES = {313--336},
      ISSN = {1239-629X,1798-2383},
   MRCLASS = {54E35 (28A80 30C65)},
  MRNUMBER = {2173367},
MRREVIEWER = {Robert\ W.\ Vallin},
}

@article {Schul,
    AUTHOR = {Schul, Raanan},
     TITLE = {Subsets of rectifiable curves in {H}ilbert space---the
              analyst's {TSP}},
   JOURNAL = {J. Anal. Math.},
  FJOURNAL = {Journal d'Analyse Math\'ematique},
    VOLUME = {103},
      YEAR = {2007},
     PAGES = {331--375},
      ISSN = {0021-7670,1565-8538},
   MRCLASS = {49Q15 (28A75 54F45)},
  MRNUMBER = {2373273},
MRREVIEWER = {Herv\'e\ Pajot},
       DOI = {10.1007/s11854-008-0011-y},
       URL = {https://doi.org/10.1007/s11854-008-0011-y},
}

@article {Okikiolu,
    AUTHOR = {Okikiolu, Kate},
     TITLE = {Characterization of subsets of rectifiable curves in {${\bf
              R}^n$}},
   JOURNAL = {J. London Math. Soc. (2)},
  FJOURNAL = {Journal of the London Mathematical Society. Second Series},
    VOLUME = {46},
      YEAR = {1992},
    NUMBER = {2},
     PAGES = {336--348},
      ISSN = {0024-6107,1469-7750},
   MRCLASS = {28A75 (90C27)},
  MRNUMBER = {1182488},
MRREVIEWER = {Christopher\ Bishop},
       DOI = {10.1112/jlms/s2-46.2.336},
       URL = {https://doi.org/10.1112/jlms/s2-46.2.336},
}

@article {MalevaPreiss2019,
    AUTHOR = {Maleva, Olga and Preiss, David},
     TITLE = {Cone unrectifiable sets and non-differentiability of
              {L}ipschitz functions},
   JOURNAL = {Israel J. Math.},
  FJOURNAL = {Israel Journal of Mathematics},
    VOLUME = {232},
      YEAR = {2019},
    NUMBER = {1},
     PAGES = {75--108},
      ISSN = {0021-2172,1565-8511},
   MRCLASS = {28A75 (26A16 26B05 28A15 46G05)},
  MRNUMBER = {3990937},
MRREVIEWER = {Marius\ R\u adulescu},
       DOI = {10.1007/s11856-019-1863-9},
       URL = {https://doi.org/10.1007/s11856-019-1863-9},
}

@article {AzzamSchul,
    AUTHOR = {Azzam, Jonas and Schul, Raanan},
     TITLE = {An analyst's traveling salesman theorem for sets of dimension
              larger than one},
   JOURNAL = {Math. Ann.},
  FJOURNAL = {Mathematische Annalen},
    VOLUME = {370},
      YEAR = {2018},
    NUMBER = {3-4},
     PAGES = {1389--1476},
      ISSN = {0025-5831,1432-1807},
   MRCLASS = {28A75 (28A12 28A78)},
  MRNUMBER = {3770170},
MRREVIEWER = {Aapo\ Kauranen},
       DOI = {10.1007/s00208-017-1609-0},
       URL = {https://doi.org/10.1007/s00208-017-1609-0},
}

@article {ENV,
    AUTHOR = {Edelen, Nick and Naber, Aaron and Valtorta, Daniele},
     TITLE = {Effective {R}eifenberg theorems in {H}ilbert and {B}anach
              spaces},
   JOURNAL = {Math. Ann.},
  FJOURNAL = {Mathematische Annalen},
    VOLUME = {374},
      YEAR = {2019},
    NUMBER = {3-4},
     PAGES = {1139--1218},
      ISSN = {0025-5831,1432-1807},
   MRCLASS = {42B35 (28A75 49Q15)},
  MRNUMBER = {3985109},
MRREVIEWER = {Lubomira\ G.\ Softova},
       DOI = {10.1007/s00208-018-1770-0},
       URL = {https://doi.org/10.1007/s00208-018-1770-0},
}

@book {DavidSemmes,
    AUTHOR = {David, Guy and Semmes, Stephen},
     TITLE = {Analysis of and on uniformly rectifiable sets},
    SERIES = {Mathematical Surveys and Monographs},
    VOLUME = {38},
 PUBLISHER = {American Mathematical Society, Providence, RI},
      YEAR = {1993},
     PAGES = {xii+356},
      ISBN = {0-8218-1537-7},
   MRCLASS = {28A75 (30C65 30E20 42B20 42B25)},
  MRNUMBER = {1251061},
MRREVIEWER = {Christopher\ Bishop},
       DOI = {10.1090/surv/038},
       URL = {https://doi.org/10.1090/surv/038},
}

@article {Krandel-TST,
    AUTHOR = {Krandel, Jared},
     TITLE = {The {T}raveling {S}alesman {T}heorem for {J}ordan {C}urves in
              {H}ilbert {S}pace},
   JOURNAL = {Michigan Math. J.},
  FJOURNAL = {Michigan Mathematical Journal},
    VOLUME = {75},
      YEAR = {2025},
    NUMBER = {3},
     PAGES = {471--543},
      ISSN = {0026-2285,1945-2365},
   MRCLASS = {99-06},
  MRNUMBER = {4929112},
       DOI = {10.1307/mmj/20226254},
       URL = {https://doi.org/10.1307/mmj/20226254},
}

@article {AlbertiOttolini,
    AUTHOR = {Alberti, Giovanni and Ottolini, Martino},
     TITLE = {On the structure of continua with finite length and
              {G}olab's semicontinuity theorem},
   JOURNAL = {Nonlinear Anal.},
  FJOURNAL = {Nonlinear Analysis. Theory, Methods \& Applications. An
              International Multidisciplinary Journal},
    VOLUME = {153},
      YEAR = {2017},
     PAGES = {35--55},
      ISSN = {0362-546X,1873-5215},
   MRCLASS = {28A75 (26A45 49J45 49Q20 54F50)},
  MRNUMBER = {3614660},
       DOI = {10.1016/j.na.2016.10.012},
       URL = {https://doi.org/10.1016/j.na.2016.10.012},
}

@article {DavidToro_holes,
    AUTHOR = {David, Guy and Toro, Tatiana},
     TITLE = {Reifenberg parameterizations for sets with holes},
   JOURNAL = {Mem. Amer. Math. Soc.},
  FJOURNAL = {Memoirs of the American Mathematical Society},
    VOLUME = {215},
      YEAR = {2012},
    NUMBER = {1012},
     PAGES = {vi+102},
      ISSN = {0065-9266,1947-6221},
      ISBN = {978-0-8218-5310-8},
   MRCLASS = {49Q20 (28A75 42B10)},
  MRNUMBER = {2907827},
MRREVIEWER = {Christopher\ Bishop},
       DOI = {10.1090/S0065-9266-2011-00629-5},
       URL = {https://doi.org/10.1090/S0065-9266-2011-00629-5},
}

@article {Lerman,
    AUTHOR = {Lerman, Gilad},
     TITLE = {Quantifying curvelike structures of measures by using {$L_2$}
              {J}ones quantities},
   JOURNAL = {Comm. Pure Appl. Math.},
  FJOURNAL = {Communications on Pure and Applied Mathematics},
    VOLUME = {56},
      YEAR = {2003},
    NUMBER = {9},
     PAGES = {1294--1365},
      ISSN = {0010-3640,1097-0312},
   MRCLASS = {42B20 (30C85 49Q20)},
  MRNUMBER = {1980856},
MRREVIEWER = {Herv\'e\ Pajot},
       DOI = {10.1002/cpa.10096},
       URL = {https://doi.org/10.1002/cpa.10096},
}

@incollection {ACP,
    AUTHOR = {Alberti, Giovanni and Cs\"{o}rnyei, Marianna and Preiss, David},
     TITLE = {Structure of null sets in the plane and applications},
 BOOKTITLE = {European {C}ongress of {M}athematics},
     PAGES = {3--22},
 PUBLISHER = {Eur. Math. Soc., Z\"{u}rich},
      YEAR = {2005},
   MRCLASS = {26B30 (03E05 26A16 26A27 26B05)},
  MRNUMBER = {2185733},
MRREVIEWER = {Tam\'{a}s M\'{a}trai},
}

@article {Bate,
    AUTHOR = {Bate, David},
     TITLE = {Structure of measures in {L}ipschitz differentiability spaces},
   JOURNAL = {J. Amer. Math. Soc.},
  FJOURNAL = {Journal of the American Mathematical Society},
    VOLUME = {28},
      YEAR = {2015},
    NUMBER = {2},
     PAGES = {421--482},
      ISSN = {0894-0347},
   MRCLASS = {46G05 (30L99 49J52 53C23)},
  MRNUMBER = {3300699},
MRREVIEWER = {Riikka Korte},
       DOI = {10.1090/S0894-0347-2014-00810-9},
       URL = {https://doi-org.proxy.bsu.edu/10.1090/S0894-0347-2014-00810-9},
}

@article {Bate_perturb,
    AUTHOR = {Bate, David},
     TITLE = {Purely unrectifiable metric spaces and perturbations of
              {L}ipschitz functions},
   JOURNAL = {Acta Math.},
  FJOURNAL = {Acta Mathematica},
    VOLUME = {224},
      YEAR = {2020},
    NUMBER = {1},
     PAGES = {1--65},
      ISSN = {0001-5962,1871-2509},
   MRCLASS = {31E05 (26A16 28A75 49Q15)},
  MRNUMBER = {4086714},
MRREVIEWER = {Jeremy\ T.\ Tyson},
       DOI = {10.4310/acta.2020.v224.n1.a1},
       URL = {https://doi.org/10.4310/acta.2020.v224.n1.a1},
}

@article {Cheeger,
    AUTHOR = {Cheeger, J.},
     TITLE = {Differentiability of {L}ipschitz functions on metric measure
              spaces},
   JOURNAL = {Geom. Funct. Anal.},
  FJOURNAL = {Geometric and Functional Analysis},
    VOLUME = {9},
      YEAR = {1999},
    NUMBER = {3},
     PAGES = {428--517},
      ISSN = {1016-443X},
   MRCLASS = {53C23 (49J52)},
  MRNUMBER = {1708448},
MRREVIEWER = {William P. Minicozzi, II},
       DOI = {10.1007/s000390050094},
       URL = {https://doi-org.proxy.bsu.edu/10.1007/s000390050094},
}

@book {Heinonen,
    AUTHOR = {Heinonen, J.},
     TITLE = {Lectures on analysis on metric spaces},
    SERIES = {Universitext},
 PUBLISHER = {Springer-Verlag, New York},
      YEAR = {2001},
     PAGES = {x+140},
      ISBN = {0-387-95104-0},
   MRCLASS = {30C65 (28A75 28A78 46E35)},
  MRNUMBER = {1800917},
MRREVIEWER = {Christopher Bishop},
       DOI = {10.1007/978-1-4613-0131-8},
       URL = {https://doi-org.proxy.bsu.edu/10.1007/978-1-4613-0131-8},
}

@article {Jones,
    AUTHOR = {Jones, Peter W.},
     TITLE = {Rectifiable sets and the traveling salesman problem},
   JOURNAL = {Invent. Math.},
  FJOURNAL = {Inventiones Mathematicae},
    VOLUME = {102},
      YEAR = {1990},
    NUMBER = {1},
     PAGES = {1--15},
      ISSN = {0020-9910},
   MRCLASS = {26B15 (05C38 28A75 30E10 42C99 90C10)},
  MRNUMBER = {1069238},
MRREVIEWER = {Pertti Mattila},
       DOI = {10.1007/BF01233418},
       URL = {https://doi-org.proxy.bsu.edu/10.1007/BF01233418},
}

@article {La02,
    AUTHOR = {Laakso, T.},
     TITLE = {Plane with {$A_\infty$}-weighted metric not bi-{L}ipschitz
              embeddable to {${\mathbb R}^N$}},
   JOURNAL = {Bull. London Math. Soc.},
  FJOURNAL = {The Bulletin of the London Mathematical Society},
    VOLUME = {34},
      YEAR = {2002},
    NUMBER = {6},
     PAGES = {667--676},
      ISSN = {0024-6093},
     CODEN = {LMSBBT},
   MRCLASS = {30C62 (28A80 30C65)},
  MRNUMBER = {1924353 (2003h:30029)},
MRREVIEWER = {Vasily A. Chernecky},
       DOI = {10.1112/S0024609302001200},
       URL = {http://dx.doi.org/10.1112/S0024609302001200},
}

@article {LangPlaut,
    AUTHOR = {Lang, Urs and Plaut, Conrad},
     TITLE = {Bilipschitz embeddings of metric spaces into space forms},
   JOURNAL = {Geom. Dedicata},
  FJOURNAL = {Geometriae Dedicata},
    VOLUME = {87},
      YEAR = {2001},
    NUMBER = {1-3},
     PAGES = {285--307},
      ISSN = {0046-5755,1572-9168},
   MRCLASS = {51K05},
  MRNUMBER = {1866853},
       DOI = {10.1023/A:1012093209450},
       URL = {https://doi.org/10.1023/A:1012093209450},
}

@article {LangSchlichenmaier,
    AUTHOR = {Lang, Urs and Schlichenmaier, Thilo},
     TITLE = {Nagata dimension, quasisymmetric embeddings, and {L}ipschitz
              extensions},
   JOURNAL = {Int. Math. Res. Not.},
  FJOURNAL = {International Mathematics Research Notices},
      YEAR = {2005},
    NUMBER = {58},
     PAGES = {3625--3655},
      ISSN = {1073-7928,1687-0247},
   MRCLASS = {53C23 (30C65 53C21 54F45)},
  MRNUMBER = {2200122},
MRREVIEWER = {Jeremy\ T.\ Tyson},
       DOI = {10.1155/IMRN.2005.3625},
       URL = {https://doi.org/10.1155/IMRN.2005.3625},
}

@book {Bogachev,
    AUTHOR = {Bogachev, V. I.},
     TITLE = {Measure theory. {V}ol. {I}, {II}},
 PUBLISHER = {Springer-Verlag, Berlin},
      YEAR = {2007},
     PAGES = {Vol. I: xviii+500 pp., Vol. II: xiv+575},
      ISBN = {978-3-540-34513-8; 3-540-34513-2},
   MRCLASS = {28-02 (28Axx 28Cxx 46G12 60G42 60G44)},
  MRNUMBER = {2267655},
MRREVIEWER = {Ren\'e\ L.\ Schilling},
       DOI = {10.1007/978-3-540-34514-5},
       URL = {https://doi.org/10.1007/978-3-540-34514-5},
}

@article {Sc16,
    AUTHOR = {Schioppa, Andrea},
     TITLE = {Derivations and {A}lberti representations},
   JOURNAL = {Adv. Math.},
  FJOURNAL = {Advances in Mathematics},
    VOLUME = {293},
      YEAR = {2016},
     PAGES = {436--528},
      ISSN = {0001-8708},
   MRCLASS = {58C20 (46J15 53C23)},
  MRNUMBER = {3474327},
MRREVIEWER = {Thomas Z\"{u}rcher},
       DOI = {10.1016/j.aim.2016.02.013},
       URL = {https://doi-org.proxy.bsu.edu/10.1016/j.aim.2016.02.013},
}

\end{document}